\documentclass[12pt]{amsart}

\usepackage[a-1b]{pdfx}   
\makeatletter \AtBeginDocument{\let\mathaccentV\AMS@mathaccentV} \makeatother

\usepackage{amsmath, amsthm, amssymb}
\usepackage{fullpage}
\usepackage{color}
\usepackage{hyperref}
\usepackage{soul}
\usepackage{enumitem}
\usepackage{graphicx}

\newcommand{\blue}[1]{{\color{black}#1}}

\newtheorem{theorem}{Theorem}
\newtheorem{lemma}{Lemma}
\newtheorem{question}{Question}
\newtheorem{proposition}[lemma]{Proposition}

\newtheorem{definition}[lemma]{Definition}
\newtheorem{remark}[lemma]{Remark}
\newtheorem*{conjecture}{Conjecture}
\numberwithin{lemma}{section}

\numberwithin{equation}{section}

\newcommand{\R}{{\mathbb R}}
\newcommand{\N}{{\mathbb N}}

\newcommand{\C}{{\mathbb C}}

\newcommand{\cc}{\mathbf c}
\renewcommand{\R}{\mathbb R}

\newcommand{\bM}{\mathbf M}

\newcommand{\bB}{\mathbf B}

\newcommand{\du}{\mathfrak{u}}
\newcommand{\dv}{\mathfrak{v}}

\newcommand{\bI}{\mathbf I}
\newcommand{\bJ}{\mathbf J}
\newcommand{\bK}{\mathbf K}
\newcommand{\bC}{\mathbf C}

\newcommand{\bu}{{\bar u}}
\newcommand{\bv}{{\bar v}}
\newcommand{\bfu}{{\mathbf u}}

\newcommand{\la}{\langle}
\newcommand{\ra}{\rangle}

\newcommand{\ol}{\overline}
\newcommand{\ms}{M^\sharp}
\newcommand{\ps}{P^\sharp}
\newcommand{\is}{\bI^\sharp}
\newcommand{\calR}{\mathcal{R}}
\newcommand{\LWP}{L\!W\!P}

\DeclareMathOperator{\Sym}{Sym}

\newcommand{\rP}{\check{P}}
\newcommand{\rPs}{\rP^\sharp}
\newcommand{\rp}{\check{p}}
\newcommand{\step}{{\mathfrak c}}
\newcommand{\coeff}{\text{coeff}}

\begin{document}

\title{Global solutions for 1D cubic defocusing dispersive equations, Part IV: general dispersion relations}

\author{Mihaela Ifrim}
\address{Department of Mathematics, University of Wisconsin, Madison}
\email{ifrim@wisc.edu}

\author{ Daniel Tataru}
\address{Department of Mathematics, University of California at Berkeley}
\email{tataru@math.berkeley.edu}

\begin{abstract}

A broad conjecture, formulated by the authors in earlier work, reads as follows: ``\emph{Cubic defocusing dispersive one dimensional flows with small initial data have global dispersive solutions}''. Notably, here smallness is only assumed in $H^s$ Sobolev spaces, without any localization assumption.

The conjecture was initially proved by the authors first for  a class of semilinear  Schr\"odinger type models, and then for quasilinear Schr\"odinger flows.  
In this work we take the next natural step, and prove the above conjecture for a much larger class of one dimensional semilinear dispersive  problems with a cubic nonlinearity, where the dispersion relation is no longer of Schr\"odinger type. 

This result is the first of its kind,
for any 1D cubic problem not of Schr\"odinger type.
Furthermore, it only requires initial data smallness \emph{at critical regularity}, a threshold that has never been reached before for any 1D cubic dispersive flow. In terms of dispersive decay, we prove that our global in time solutions satisfy both global $L^6_{t,x}$ Strichartz estimates and bilinear $L^2_{t,x}$ bounds. 
\end{abstract}


\maketitle

\setcounter{tocdepth}{1}
\tableofcontents

\section{Introduction}
The question of obtaining scattering, global in time solutions for one dimensional dispersive flows with small initial data has been extensively
studied in recent years. Typically the nonlinearities are assumed to be cubic, though
some quadratic nonlinearities are also sometimes allowed under structural assumptions (null condition) which limit the number and strength of resonant interactions.

These are problems for which the nonlinear effects are stronger than the linear ones 
for any initial data, so scattering solutions
do not exist. Instead, the best one could hope for is solutions with $t^{-\frac12}$ dispersive
decay but with a modified scattering asymptotic behavior. Even this is only possible if the
initial data is \emph{ localized}, which is why until very recently the initial data localization 
was a standard assumption in all work in this direction, usually supplemented by a smoothness assumption largely motivated by technical considerations,   see for instance
see \cite{HN,HN1,LS,KP,IT-NLS} for cubic problems
and \cite{AD,IT-g,D,IT-c,LLS} for situations where
quadratic nonlinearities are also allowed.
 The smoothness requirement can be avoided to a large extent, see \cite{IT-wp} and the discussion therein.

The above predicament for this class of problems
was fundamentally changed when the authors 
proposed the following general
conjecture:

\begin{conjecture}
One dimensional dispersive flows with cubic defocusing nonlinearities and small initial data have global in time, scattering solutions.
\end{conjecture}

This opened the door to study  the much more difficult case where the initial data
is just \emph{small}, but without any localization assumption.
In subsequent work the authors proved the above conjecture for both semilinear
and quasilinear Schr\"odinger flows in \cite{IT-global}
and \cite{IT-qnls}. However, no such result exists so far for  problems whose dispersion relation is not of Schr\"odinger type. 

The aim of this paper is exactly to expand the reach of our work to more general dispersion relations, and prove the conjecture in that setting. In order to avoid compounding multiple difficulties, the problems we consider here are
semilinear.
However, in the process we develop new technical tools which we hope will prove useful for an even wider array of problems, including quasilinear flows.

Just as in our prior work, as part of our results, we also prove that our global solutions are scattering at infinity in a very precise, quantitative way, in the sense that they  satisfy both $L^6_{t,x}$ Strichartz estimates and bilinear $L^2_{t,x}$ bounds.
This is despite the fact that the nonlinearity is non-perturbative on large time scales.

\subsection{ Cubic dispersive problems in one space dimension}

Our interest in this paper is in defocusing 
cubic problems of the form
\begin{equation}\label{eq:main}
i \partial_t u - A(D_x) u =  C(u,\bar u, u), \qquad u(0) = u_0 \in H^s(\R), 
\end{equation}
where we consider complex solutions in one space dimension,
\[
u: \R \times \R \to \C.
\]

Here the symbol $a$, often referred to as the \emph{the dispersion relation}, is generally assumed to be a smooth real valued symbol which is strictly convex,
\begin{equation}\label{a-convex}
a''(\xi) > 0,    
\end{equation}
which makes our problem dispersive\footnote{ Of course one can assume $a''<0$ and still have a dispersive problem, so the sign choice 
here merely serves to fix the notations.}.

With these notations the characteristic set for the linear 
flow is given by
\[
\Sigma= \{\tau +  a(\xi)= 0\},
\]
and the group velocity of waves with frequency $\xi$ is given by
\[
v_\xi =   a'(\xi).
\]

On the other hand the  cubic nonlinearity $C$
will be taken to be a a trilinear translation
invariant form, which can be described by its symbol $c(\xi_1,\xi_2,\xi_3)$. The arguments $u,\bar u$ and $u$ of $C$ are chosen so that
our equation \eqref{eq:main} has the phase rotation symmetry,
$u \to u e^{i\theta}$, as it is the case in many examples of interest. We note here that higher 
homogeneity nonlinearities can also be easily added, but we have not done that in order to streamline the exposition. Cubic nonlinearities 
without the phase rotation symmetry can also be added in some cases, but this is highly dependent
on the exact dispersion relation, which we seek 
to avoid here.

We now describe the assumptions we  make on the symbols $a(\xi)$ and $c(\xi_1,\xi_2,\xi_3)$ as $\xi \to \infty$. Here there are many choices, and 
we will try to strike a balance between generality
on one hand, and streamlined arguments on the other hand. To some extent we will follow the lead
of our earlier paper \cite{IT-wp}, though with some differences. 
\medskip

\emph{I. The properties of the symbol $a$,} 
which capture the dispersive character of the flow:

\begin{enumerate}[label=\textbf{(A\arabic*)}]
\item The symbol $a$ satisfies the convexity property \eqref{a-convex}. In addition,  we assume that as $\xi \to \pm\infty$ we have the polynomial behavior
\begin{equation}
a''(\xi) \approx |\xi|^\gamma, \qquad \gamma \in \R.    
\end{equation}
\end{enumerate}
Here we can distinguish two cases:
\smallskip

a) $\gamma \geq -1$, which we will refer to as generalized NLS (GNLS), where $a$ is coercive 
at $\pm \infty$,
\[
\lim_{|\xi|\to \infty} \frac{a(\xi)}{|\xi|} = \infty,
\]
and which corresponds to an infinite speed of propagation in the high frequency limit.

\smallskip

b) $\gamma < -1$, which we will refer to as generalized Klein-Gordon (GK-G), where $a$ has linear asymptotes   at $\pm \infty$,
\[
\lim_{\xi \to \infty} \frac{a(\xi)}{\xi} = - v_{\pm},
\]
which corresponds to a finite speed of propagation in the high frequency limit, with a range of velocities $\left(v_+,v_-\right)$.

We remark on several special values for $\gamma$:

\begin{itemize}
    \item $\gamma = 1$ corresponds to mKdV, at high frequency.
   
   \item $\gamma = 0$ is the NLS exponent.

   \item $\gamma = -1$ is the threshold between finite and infinite propagation speed, and arises for instance in the SQG front equation,
   see \cite{AA1,AA2}. In more generalized SQG models one has $\gamma \in (-2,0)$, see for instance \cite{AA3}.

   \item $\gamma = -2$ represents another significant threshold, below which the local 
   well-posedness problem trivializes.  

\item $\gamma = -3$ corresponds to Klein-Gordon half-waves.
    
\end{itemize}
The threshold exponent $\gamma=-1$ between
the GNLS and the GK-G cases deserves separate consideration, and is excluded in
the present work.

\bigskip

\emph{II. The properties of the symbol $c$:} 
The symbol $c(\xi_1,\xi_2,\xi_3)$ will be required to satisfy a minimal set of assumptions describing its size, regularity and behavior as $\xi \to \infty$:

\begin{enumerate}[label=\textbf{(H\arabic*)}]

\item Size and regularity: 
\begin{equation}\label{c-smooth}
|\partial_\xi^\alpha c(\xi_1,\xi_2,\xi_3)| \leq c_\alpha \prod \la \xi_j \ra^{\delta-\alpha_j}
\quad \mbox{ for }\xi_1,\xi_2,\xi_3 \in \R,\,   \mbox{ and for every  multi-index $\alpha$}.
\end{equation}

\item Conservative: 
\begin{equation}\label{c-conserv}
\Im c(\xi,\xi,\xi) = 0, \ \Im \nabla c(\xi,\xi,\xi) = 0, \qquad \xi \in \R,
\end{equation}
where $\Im z$ denotes the imaginary part of $z\in \mathbb{C}$.

\item Defocusing: 
\begin{equation}\label{c-defocus}
c(\xi,\xi,\xi) \gtrsim \la \xi\ra^{3\delta}, \qquad \xi \in \R  \mbox{ and } c \in \mathbb{R^+}.
\end{equation}
\end{enumerate}

In reading these assumptions, one should distinguish between two regimes:
\medskip

a) balanced frequency interactions, 
\[
\la \xi_1 \ra \approx \la \xi_2 \ra \approx \la \xi_2 \ra \approx \la \xi_4\ra, 
\qquad \xi_4:= \xi_1-\xi_2+\xi_3,
\]
with both inputs and the output at comparable frequencies, where all three conditions (H1-H3) are effective, and where all the defocusing effects are localized. Here it is 
natural to impose similar regularity conditions with respect to all variables.

b)  unbalanced frequencies, $\la \xi \ra_{min} \ll \la \xi \ra_{max}$. In this regime we only need to control the size and regularity of the symbol, which is described by the H\"ormander-Mihlin type behaviour in (H1).

\begin{remark}
    In selecting these assumptions we have tried to strike a balance between the generality of the result on one hand, and a streamlined exposition on the other hand. One could for instance consider much more general unbalanced nonlinearities roughly of the form 
    \[
C(u,\bu,u) \approx D^\alpha (D^{\delta_1} u
D^{\delta_2} \bu D^{\delta_3} u),\quad \delta{_i}\in \mathbb{R}, \quad i=\overline{1,3},
    \]
which occur in many models. While the core of our approach remains valid in this case, 
the number of cases to consider would be a major distracting technicality. Instead it would be 
more interesting to directly apply ours ideas 
to specific models of interest.
\end{remark}

Concerning the allowed range of $\delta$ for a given $\gamma$, this will be discussed later 
in greater detail. But for now we make one key distinction:

\begin{remark}[Semilinear vs. quasilinear]
Depending on $\gamma$ and $\delta$,
the problem might be semilinear or quasilinear. Here we consider the semilinear case, which can be described as follows\footnote{The thresholds below will become clear 
in the proof of the local well-posedness result in Theorem~\ref{t:local}}:

\begin{equation}\label{semilinear}
\begin{aligned}
\delta \leq & \ \gamma+1,  \qquad &\gamma \geq  -1,
\\
\delta \leq & \ 0,   \qquad& \gamma <   -1.
\end{aligned}    
\end{equation}
In the quasilinear case, further qualitative constraints would be needed 
on $c(\cdot, \cdot, \cdot)$ also off-diagonal.
The case $\delta = \gamma+1$, $\gamma \geq -1$
should be thought of as borderline semilinear,
where additional challenges are present in the analysis.
\end{remark}

As noted earlier, our focus here is on the semilinear case, also excluding the
borderline case.

Even though under the above assumptions the problem \eqref{eq:main} does not admit an exact scaling law, we can still identify a natural 
scaling in the high frequency limit, 
\[
u(t,x) \to \lambda^{\frac{\gamma+2-3\delta}2} u(\lambda^{\gamma+2} t, \lambda x).
\]
Here we remark on the case $\gamma < -1$, where 
the scaling is implemented separately at positive and negative frequencies, after removing the asymptotic transport term, as if $a(\xi) \approx |\xi|^{\gamma+2}$.

The asymptotic scaling law immediately leads to 
a critical Sobolev exponent,
\begin{equation}
s_c = \frac12(3\delta - \gamma -1),  
\end{equation}
which serves as a guide for both the local and the global well-posedness theory. Precisely, 
we will restrict ourselves to Sobolev exponents $s \geq s_c$,
and one natural question is how close can we get to the critical Sobolev exponent in the global result ? The full question to ask would be 

\begin{question}
Find all triplets $(\gamma,\delta,s)$ for which   global well-posedness for \eqref{eq:main} holds 
for all initial data which is small in $H^s$.
\end{question}

 But this still involves too many cases, partly because partial normal form arguments appear to play a role even for unbalanced frequency interactions. Furthermore, as we will see later, the local well-posedness 
 threshold and the smallness threshold 
 are in general not the same. Decoupling the two Sobolev 
 exponents, it is the latter which we aim to bring down to scaling in the present work.
 For these reasons, here we will confine ourselves to asking a simpler question:

\begin{question}
Find the pairs $(\gamma,\delta)$ for which a small $H^{s_c}$ norm for the initial data guarantees the existence of global dispersive solutions.
\end{question}

Here we would also assume sufficient regularity for the initial data so that local well-posedness holds, and then ask that this regularity be propagated globally in time.

\subsection{ The main result}

In order to open the discussion of our main results, it is useful to set some limits on the parameter $\delta$. Our starting point is the 
bound from above in \eqref{semilinear}, which 
guarantees that our problem is semilinear. 
For what follows we complement this with a bound from below. A natural choice here would be
\begin{equation}\label{delta-main}
\begin{aligned}
0 \leq \delta \leq \gamma+1, & \qquad \gamma \geq -1
\\
\gamma+1 \leq \delta \leq 0, &  \qquad \gamma < -1.
\end{aligned}    
\end{equation}
Here, in a nutshell, the bound from below serves to guarantee 
that unbalanced resonant interactions are  ``perturbative" above scaling. Thus one may view
this as a fundamental constraint, as such interactions cannot be generically removed using any normal form analysis. 

However, nonresonant interactions also play a role, and bring forth further restrictions on $\gamma$.
Normal form tools may be brought to bear in this case. But, since 
this is the first result of this type, we choose to limit the use of such tools and instead to have 
a clean and streamlined argument.
Hence we assume, for the rest of the paper, that we have the stronger 
condition
\begin{equation}
\label{delta-gwp}
\begin{aligned}
\frac13(\gamma+1) \leq \delta < \gamma+1, & \qquad \gamma > -1
\\
\frac23(\gamma+1) \leq \delta \leq  0, &  \qquad \gamma < -1.
\end{aligned}    
\end{equation}
The case $\gamma = -1$, which corresponds for instance to the SQG front equation, is special in many
ways and will be considered in a separate article. Other interesting cases for $\gamma$ and $\delta$ will also be considered separately.

We first discuss the local well-posedness problem, where we encounter a different 
Sobolev threshold than scaling, which we denote by $s_{\LWP}$. For $\delta$ as in \eqref{delta-gwp}, this is defined as follows:
\begin{equation}
s_{\LWP}: = \left\{
\begin{aligned}
&\ \frac32 \delta - \frac14 \gamma  \qquad &\gamma \geq -2, \ \gamma \neq -1.
\\
&\ \frac32 \delta + \frac12  \qquad 
&\gamma < -2, \quad \delta \in [-\frac23,0]
\\
&\ \, \delta+\frac16  \qquad &\gamma < -2, \quad \delta <-\frac23
\end{aligned}
\right.
\end{equation}
We remark that $s_{\LWP} \geq s_c$,
with equality only at the 
endpoint $\gamma = -2$. Then the local well-posedness result is as follows:

\begin{theorem}\label{t:local}
Let $\gamma \in \R$, and $\delta$ in the range \eqref{delta-gwp}. Assume that the dispersion relation $a$ satisfies the condition (A1), and the symbol $c(\cdot,\cdot,\cdot)$ of the nonlinearity satisfies the condition (H1). Then the  evolution \eqref{eq:main} is locally well-posed for data in $H^s$ where
\begin{equation}
   s \geq s_{\LWP},
 \end{equation}  
with strict inequality in the case 
$\gamma \leq  -2$, $\delta = 0$ or  $\delta \leq -\frac23$.
\end{theorem}

Here we need to clarify the meaning of well-posedness. For this problem, we will establish a semilinear type of well-posedness result. Precisely, for each initial data $\du_0$ in $H^s$ a unique local solution exists in a well chosen function space $X^s \subset C([0,T];H^s)$, with Lipschitz dependence on the initial data, where $T$ only depends on the $H^s$ norm of the data. 

As noted earlier, this result is strictly above  scaling,  $s_{\LWP} > s_c$ unless $\gamma = -2$. Heuristically this 
is related to the fact that the local well-posedness result is semilinear, perturbative, whereas the cubic balanced interactions are  nonperturbative below $s_{\LWP}$.

We also remark that, reflecting its semilinear nature, this result only uses symbol bounds on $c$, without any need for the conservative and defocusing assumptions. In this 
general context, Theorem~\ref{t:local}
is very likely sharp.
However, it is an open question whether under additional assumptions on $c$ (e.g. conservative) one might be able to lower the local well-posedness threshold in a quasilinear fashion. So far, results of this type have been proved only in a very limited, completely integrable context, see for instance the 1D cubic NLS result in \cite{KV-NLS3}.

\bigskip

Now we turn our attention to the primary objective, namely the global well-posedness problem.
Our main global result asserts that, for a range of exponents $\delta$ which depends on $\gamma$,
global well-posedness holds for our problem for initial data which is small in $H^{s_c}$.
In addition, our solutions not only satisfy uniform $L^2$, but also global space-time $L^6_{t,x}$ estimates, as well as bilinear $L^2_{t,x}$ bounds, as follows:

\begin{theorem}\label{t:main}
Let $\gamma \in \R$, and $\delta$ in the range \eqref{delta-gwp}. Assume that the dispersion relation $a$ satisfies the condition (A1), and the symbol $c(\cdot,\cdot,\cdot)$ of the nonlinearity satisfies the condition (H1). Let $s$ be as in Theorem~\ref{t:local}. Then for any initial data $\du_0 \in H^s$ satisfying the smallness condition
\begin{equation}
\|\du_0\|_{H^{s_c}} \leq \epsilon \ll 1,
\end{equation}
there exists a unique global solution $u$ for \eqref{eq:main}, which satisfies the following bounds:

\begin{enumerate}[label=(\roman*)]
\item Uniform Sobolev bounds:
\begin{equation}\label{main-L2}
\| u \|_{L^\infty_t H^{s_c}_x} \lesssim \epsilon,
\end{equation}
respectively 
\begin{equation}\label{main-Hs}
\| u \|_{L^\infty_t H^{s}_x} \lesssim \| \du_0\|_{H^s_x}.
\end{equation}

\item Strichartz bound:
\begin{equation}\label{main-Str}
\|\la D \ra^{s_c
+\frac{\gamma}{6}} u \|_{L^6_{t,x}} \lesssim \epsilon^\frac23.
\end{equation}

\item Bilinear $L^2$ bounds:
\begin{equation}\label{main-bi}
\begin{aligned}
&\|  T_{\la D\ra^{s_c+\sigma} u} \bu^{x_0} \|_{L^2_t H^{s_c-\sigma-\frac{\gamma+1}2} }
\lesssim \epsilon^2, \qquad 
&\sigma > \min\{0, \frac{\gamma+1}2\}
\\
&\|\partial_x  \Pi(  \la D\ra^{s_c+\sigma} u, \la D\ra^{s_c+\sigma} \bu^{x_0}) \|_{L^2_{xt}}
\lesssim \epsilon^2, \qquad 
&\sigma = \frac{\gamma-1}4,
\end{aligned}
\end{equation}
where $u^{x_0}(t,x):=u(t, x +x_0)$.
\end{enumerate}
\end{theorem}
The bounds in \eqref{main-bi} are expressed using the standard paradiffferential language, where one should think of the paraproduct  $T_AB$ as the bilinear operator that selects the low frequencies of $A$ in comparison with $B$'s frequencies, and $\Pi$ as the bilinear operator selecting comparable frequencies of each factor; this is the Coifman-Meyer paradifferential decomposition.

We also remark on the translation parameter $x_0$ in \eqref{main-bi}.
This is not needed in the  Strichartz estimates, which are by default invariant with respect to translations.
Introducing $x_0$ and making this bound uniform  with respect to the  $x_0$ translation captures the natural separate translation invariance in the bilinear estimates, and is also quite useful in our proofs. The uniformity 
in particular is a semilinear feature; 
in the quasinear Schr\"odinger case
considered in \cite{IT-qnls} there
is some growth allowed in the implicit constant in the balanced case as $x_0$ increases beyond the uncertainty principle scale. 

We note here that the bounds \eqref{main-L2}-\eqref{main-bi}
are written in a compact, simplified form in the above theorem, in order 
to directly give the reader an idea
about the results without any additional preliminaries.  What we regard as the full, more accurate  form of the result is captured by the corresponding frequency envelope dyadic bounds, which are provided later in  Theorem~\ref{t:boot}.
We continue with several additional remarks:

\begin{remark} 
On the range of $\delta$:  the constraint in \eqref{delta-gwp} is somewhat  more restrictive than in \eqref{delta-main}. One should regard this difference as a technical one, to be improved in subsequent work (or even fully removed for some range of $\gamma$). 
\end{remark}

\begin{remark}
On the range of $s$: for all $\delta$ as in \eqref{semilinear} we expect a global result to hold for data which is small in $H^s$, if $s$ is sufficiently large. We do not pursue this in the present work, both in order to avoid distracting technicalities,
and because in many problems of interest the nonlinearity does not have the symmetric symbol bounds we assume here.
\end{remark}

\begin{remark}
On higher regularity bounds: the 
energy estimates in \eqref{main-L2}, the Strichartz estimates in \eqref{main-Str} and the bilinear $L^2$ estimates 
in \eqref{main-bi} are based solely 
on the initial data size. But if in addition we have more initial regularity $\du_0 \in H^s$ for some
$s > s_c$, then our proof shows that this regularity is propagated globally in time in all three estimates \eqref{main-L2}-\eqref{main-bi}. This is a direct consequence of the sharper, frequency 
envelope form of our estimates, as stated
in Theorem~\ref{t:boot}.
\end{remark}

There are several ideas which play key roles in our analysis, all of which are part of the authors' new approach 
to global solutions in nonlinear dispersive flows, as it is currently being developed in \cite{IT-global}
\cite{IT-focusing},\cite{IT-qnls}, \cite{IT-conjecture}, \cite{IT-qnls2}.
These ideas have been revisited, expanded in multiple ways and used in a nonstandard  fashion in these 
papers as well as in  the present work.
In particular, this is the first article where these ideas are retooled
for the case of non-Schr\"odinger dispersion relations:

\medskip

\emph{ 1. Energy estimates via density flux identities.} This is 
a classical idea in pde's, and particularly in the study of conservation laws, namely that the density-flux identities 
play a more fundamental role than just energy identities. The new 
twist in our context is that this analysis is carried out in 
a nonlocal setting, where both the densities and the fluxes involve 
translation invariant multilinear forms, and careful choices are essential.
\medskip

\emph{2. The use of energy corrections.}
This is an idea originally developed in the context of the so called 
I-method~\cite{I-method} or more precisely the second generation I-method \cite{I-method2}, whose aim was to construct more accurate almost conserved quantities. In this series of papers we implement this idea at the level of density-flux
identities. Prior to the article, 
this was primarily done for the frequency localized mass and momentum. Here we also
introduce a new functional, which we call reverse momentum. Another new idea, which 
plays an important role in the GKG case
$\gamma < - 1$, is that of relative mass, momentum and reverse momentum, which means
tracking them in a moving frame, whose velocity
is the asymptotic group velocity $v^{\pm}$ in the 
high frequency limit at $\pm \infty$.
\medskip

\emph{3. Interaction Morawetz bounds.} These were originally developed in the context of the three-dimensional NLS problems by Colliander-Keel-Stafillani-Takaoka-Tao in \cite{MR2053757},
and have played a fundamental role in the study of many nonlinear Schr\"odinger flows, see e,g. \cite{MR2415387,MR2288737}, the one-dimensional quintic flows in the work of Dodson~\cite{MR3483476,MR3625190}
and the one-dimensional  approach of Planchon-Vega \cite{PV}. In our series of works these bounds are recast in the setting and language of nonlocal multilinear forms.  Our adaptation 
of these ideas to the case of non-Schr\"odinger dispersion relations
is completely new and opens many doors.

\medskip

\emph{4. Tao's frequency envelope method.} Associated to a Littlewood-Paley decomposition of the solutions,  this is used as 
a way to accurately track the evolution of the energy distribution across frequencies. This is also very convenient as a bootstrap tool,  see e.g. Tao~\cite{Tao-WM}, \cite{Tao-BO} but  with the added twist of also bootstrapping bilinear Strichartz bounds, as in the authors' paper \cite{IT-BO}.

\subsection{ An outline of the paper} 
In the next section we begin by setting up the notations
for function spaces and multilinear forms and symbol classes. We also introduce the Littlewood-Paley decomposition of the solution as well as our class of admissible frequency envelopes associated to it. We also briefly describe and classify the cubic resonant interactions.

The linear flow is considered in Section~\ref{s:Strichartz}, where we 
discuss Strichartz and $L^2_{t,x}$ bilinear 
bounds in a dyadic setting. This is not directly useful from the perspective
of the global in time bounds, other than for 
orientation. However, it does play a role 
in the local well-posedness result. For that purpose, we also phrase the Strichartz and bilinear $L^2$ bounds in terms of the $U^2_A$
and $V^2_A$ spaces associated to the linear flow.

In Section~\ref{s:local} we carry our a preliminary step in the proof of our main result, namely we prove the local well-posedness result in Theorem~\ref{t:local}. This is independent of the global result, and uses a contraction argument
in well chosen function spaces (e.g. Strichartz, $U^2_A$) associated to the linear flow.

The goal of Section~\ref{s:energy} is to recast energy identities in density-flux form. This is done first for the 
mass, and then for the momentum, which 
we appropriately define in our context as the flux for the mass. Finally we also introduce a \emph{reverse momentum} functional, which has the mass as its flux. We supplement this with two additional steps, 
where we first consider frequency localized mass and momentum  densities, and then we improve their accuracy by
adding well chosen quartic corrections.

In Section~\ref{s:Morawetz} we begin with our general interaction Morawetz identities for the linear problem, and then we use our density-flux identities for the sharp frequency localized mass, momentum and reverse momentum in order to obtain a set of refined interaction 
Morawetz identities for the nonlinear problem. 
This was classically done for Schr\"odinger type equations, which exhibit a favourable algebraic structure. Such a structure is absent 
in the case of the general dispersion relations, which is why in this article we develop a ``second generation" interaction Morawetz 
analysis which no longer relies on 
the Schr\"odinger setup.
For clarity  of exposition we consider separately the \emph{diagonal case}, where the interaction of equal frequency components is considered,
the \emph{semi-diagonal case}, dealing with unequal but comparable frequencies,
and the \emph{unbalanced case}, which corresponds to widely separated frequency ranges.

The proof of our global result uses a complex bootstrap argument, involving both energy, Strichartz and bilinear
$L^2_{t,x}$ bounds in a frequency localized setting and based  on frequency envelopes. The bootstrap set-up is laid out in Section~\ref{s:boot}, which also contains a sharper, frequency envelope version of our result, in Theorem~\ref{t:boot}. Our main estimates closing the bootstrap argument are carried out in Section~\ref{s:fe-bounds}, using the density-flux and interaction Morawetz identities previously obtained.

\subsection{Acknowledgements} 
 The first author was supported   by the NSF CAREER grant DMS-1845037 and the NSF grant DMS-2348908, by a Miller Visiting Professorship at UC Berkeley for the Fall semester 2023, and by the Simons Foundation as a Simons Fellow. The second author was supported by the NSF grant DMS-2054975 as well as by a Simons Investigator grant from the Simons Foundation
 and as Simons Fellow.

\

\section{Notations and preliminaries}

\subsection{Littlewood-Paley decompositions}
For our analysis it will be convenient to localize functions  in (spatial) frequency on the  dyadic scale. Further, we need to separate positive and negative dyadic regions.
For this we consider a partition of unity 
\[
1 = p_0(\xi) + \sum_\pm \sum_{k \in \N^{*}} p_k^\pm(\xi),
\]
where $p_k^+$ are smooth bump functions localized in $[\step^{k-1},\step^{k+1}]$,   $p_k^-$ are smooth bump functions localized in $[-\step^{k+1},-\step^{k-1}]$, and $p_0$ is supported in 
in $[-2,2]$.
 As opposed to the traditional choice $\step=2$, here we use a smaller step with 
\begin{equation}\label{step}
  0 <   \step -1 \ll 1.
\end{equation}
The motivation for this choice is to allow for a cleaner description of 
balanced interactions in terms of our dyadic parameters. Correspondingly, our solution $u$ will be decomposed as
\[
u = u_0 + \sum_\pm \sum_{k \in \N^*} u^{\pm}_k, \qquad \mbox{ with } \qquad  u_0 = P_0 u, \quad u_k^\pm := P_k^\pm u. 
\]
The main estimates we will establish for our global solutions $u$ will be linear and bilinear estimates  for the dyadic components $u_k^{\pm}$.

We will mostly use the size of the frequency as opposed to the exponent $k$ to denote the Littlewood-Paley pieces of a function $u$,
\[
u = u_0 + \sum_{\pm} \sum_{\lambda \in \step^{\N}} u_{\lambda}^\pm.
\]
It will often be the case that 
the sign does not matter, in which case
we will simply omit it.

\subsection{ Frequency envelopes} \label{s:fe}
This is a tool which allows us 
to more accurately track the distribution of energy at various frequencies for the solutions to nonlinear evolution equations. In the present paper,
they play a key bookkeeping role in the proof of the linear and bilinear bounds for our solutions in the context of a complex bootstrap argument. In brief, given a Littlewood-Paley decomposition as above for a function $u \in H^s(\R)$,
a frequency envelope for $u$ in $H^s$ is a double sequence $\{c_0, c^\pm_k\}$ with the property that 
\[
\| u_k^{\pm} \|_{H^s} \lesssim c_k^{\pm}, \qquad \|c_k^{\pm}\|_{\ell^2} \approx \|u\|_{H^s}.
\]
In addition, one also limits how rapidly the sequence $\{c_k^{\pm}\}$ is 
allowed to vary. As  originally introduced  in work of Tao, see e.g. \cite{Tao-WM}, in the context of dyadic Littlewood-Paley decompositions,
 one assumes that the sequence $\{c_k\}$ is slowly varying, in the sense that
\[
\frac{c^\pm_j}{c^\pm_k} \leq \step^{\delta|k-j|}
\]
for matched signs, with the convention $c_0^+= c_0^- := c_0$.

A variation on this theme, also present in the literature, is to imbalance this constraint as follows:
\begin{equation}\label{slow-unbal}
 \step^{-\delta(k-j)}  \leq \frac{c^\pm_j}{c^\pm_k} \leq \step^{C(k-j)}, \qquad j \leq k, 
\end{equation} 
again for matched signs.
This is so that one can use the frequency envelopes to capture also the higher Sobolev norms, by insuring that
\[
\|2^{k(s_1-s)}c^\pm_k\|_{\ell^2} \approx \|u\|_{H^{s_1}}, \qquad s \leq s_1 < s+C.
\]
Frequency envelopes that have 
this property will be called \emph{admissible}.
An important observation is that admissible envelopes can always be found.

When using dyadic frequency indexes rather that 
integer ones, we denote the frequency envelope
terms by $c_0$, respectively $c_\lambda^\pm$. 
Again signs will be omitted where they play no role.

\subsection{Multilinear forms and symbols}\label{s:multi}

A key notion which is used throughout the paper is that 
of multilinear form. 
All our multilinear forms are invariant with respect to 
translations, and have as arguments either complex valued  functions or their complex conjugates. Our conventions will be the same as in \cite{IT-global},
which we recall here.

For an integer $k \geq 2$, we will use 
translation invariant $k$-linear  forms 
\[
(\mathcal D(\R))^{k} \ni (u_1, \cdots, u_{k}) \to     B(u_1,\bu_2,\cdots) \in \mathcal D'(\R),
\]
where the nonconjugated and conjugated entries are alternating.

Such a form is uniquely described by its symbol $b(\xi_1,\xi_2, \cdots,\xi_{k})$
via
\[
\begin{aligned}
B(u_1,\bu_2,\cdots)(x) = (2\pi)^{-k} & 
\int e^{i(x-x_1)\xi_1} e^{-i(x-x_2)\xi_2}
\cdots 
b(\xi_1,\cdots,\xi_{k})
\\ & \qquad 
u_1(x_1) \bu_2(x_2) \cdots  
dx_1 \cdots dx_{k} d\xi_1\cdots d\xi_k,
\end{aligned}
\]
or equivalently on the Fourier side
\[
\mathcal F B(u_1,\bu_2,\cdots)(\xi)
= (2\pi)^{-\frac{k-1}2} \int_{D}
b(\xi_1,\cdots,\xi_{k})
\hat u_1(\xi_1) \bar{\hat u}_2(\xi_2) \cdots  
d\xi_1 \cdots d\xi_{k-1},
\]
where, with alternating signs, 
\[
D = \{ \xi = \xi_1-\xi_2 + \cdots \}.
\]

They can also be described via their kernel
\[
B(u_1,\bu_2,\cdots)(x) =  
\int K(x-x_1,\cdots,x-x_{k})
u_1(x_1) \bu_2(x_2) \cdots  
dx_1 \cdots dx_{k},
\]
where $K$ is defined in terms of the  
Fourier transform  of $b$
\[
K(x_1,x_2,\cdots,x_{k}) = 
(2\pi)^{-\frac{k}2} \hat b(-x_1,x_2,\cdots,(-1)^k x_{k}).
\]

All the symbols in this article will be 
assumed to be smooth, bounded and with bounded derivatives.

We remark that our notation is slightly nonstandard because of the alternation of complex conjugates, which is consistent with the set-up of this paper. Another important remark is that, for $k$-linear forms, the cases of odd $k$, respectively even $k$ play different roles here, as follows:

\medskip

i) The $2k+1$ multilinear forms will be thought of as functions, e.g. those which appear 
in some of our evolution equations.

\medskip

ii) The $2k$ multilinear forms will be thought of as densities, e.g. which appear 
in some of our density-flux pairs.

\medskip
Correspondingly,  
to each $2k$-linear form $B$ we will associate
a $2k$-linear functional $\bB$ defined by 
\[
\bB(u_1,\cdots,u_{2k}) := \int_\R B(u_1,\cdots,\bu_{2k})(x)\, dx,
\]
which takes real or complex values.
This may be alternatively expressed 
on the Fourier side as 
\[
\bB(u_1,\cdots,u_{2k}) = (2\pi)^{1-k} \int_{D}
b(\xi_1,\cdots,\xi_{2k})
\hat u_1(\xi_1) \bar{\hat u}_2(\xi_2) \cdots  
\bar{\hat u}_{2k}(\xi_{2k})\, d\xi_1 \cdots d\xi_{2k-1},
\]
where, with alternating signs, the diagonal $D_0$ is given by
\[
D_0 := \{ 0 = \xi_1-\xi_2 + \cdots \}.
\]
Note that in order to define the multilinear functional $\bB$ we only need to know the symbol $b$ on $D_0$. There will be however 
more than one possible smooth extension of 
$b$ outside $D_0$. This will play a role in our story later on.

Frequently in our analysis in this article
we will apply $2k$-linear forms to identical 
arguments, as in $\bB(u,u,\cdots,u)$, particularly for real valued forms. Then 
one can freely symmetrize its symbol.
We will denote the symmetrized symbol by $\Sym(b)$. The group of symmetries here
contains:
\begin{itemize}
    \item permutations of the even indices,
    \item permutations of the odd indices, and
    \item switching the even and odd indices.
\end{itemize}
One reason symmetrizations are useful is that 
they may capture cancellations and yield a better structure, akin to the role played by integrations by parts in the case of differential operators.

In many cases it is convenient to think 
of multilinear forms as generalizations of the notion of product. In particular, 
in the translation invariant setting it is often straightforward to pass from multiplicative estimates to bounds for multilinear forms. An elegant umbrella 
for this is provided by Tao's $L$ notation \cite{Tao-WM}, which stands for translation invariant bilinear or multilinear forms
with integrable kernel. For instance in the bilinear setting, if we have a multiplicative bound 
\[
X \cdot Y \to Z
\]
for some Sobolev spaces $X,Y,Z$, then we also immediately have the bilinear bound
\[
L: X \times Y \to Z.
\]
We will use this notation repeatedly, 
primarily in frequency ocalized settings where the $L^1$ bound for the kernel is a direct consequence of the corresponding symbol bounds.

\subsection{Symbol classes}

As a general rule, our symbols will be smooth, with frequency depending regularity akin to the standard 
H\"ormander-Mikhlin classes but expanded to  the multilinear setting separately in each variable.
Thus given a positive weight $m$ which is slowly varying with respect to each variable,
\[
|\partial_{j} m (\xi_1, \cdots,\xi_k)| \lesssim  
\la \xi_j \ra^{-1} m (\xi_1, \cdots,\xi_k),
\]
we will define the symbol class  $S(m)$ as those 
symbols $b(\xi_1, \cdots,\xi_k)$ so that
\[
|\partial^\alpha  b (\xi_1, \cdots,\xi_k)| \lesssim_\alpha  
\la \xi \ra^{-\alpha} m (\xi_1, \cdots,\xi_k)
\]
for all multi-indices $\alpha$. Here $\langle \xi \rangle$ is the usual Japane
se bracket, $\langle \xi \rangle =\sqrt{1+\vert \xi\vert^2}$.

Often we will work with symbols which are localized at a single frequency $\lambda$, in which case we indicate this with the notation $S_\lambda(m)$. An even simpler class of symbols is that associated to $m = \lambda^{\sigma}$, which we will simply denote by $S^{\sigma}_\lambda$.

\subsection{ Multilinear expansions}

Since we are working with a problem with a cubic nonlinearity, all the nonlinear expressions in this article admit multilinear expansions. For instance, energy type densities will have multilinear expansions of even homogeneity,
\[
e(u) = e^2(u,\bu) + e^4(u,\bu,u,\bu) + \cdots
\]
Similarly for functions we will have odd expansions.

In this context, it is convenient to have a notation which selects the terms of a given 
homogeneity. We denote these linear operators 
by $\Lambda^j$. For example, in the case of the energy functional $e$ above we have
\[
\Lambda^4 (e(u)) = e^4(u).
\]

\subsection{Cubic interactions in the nonlinear  flow}
Given three input frequencies $\xi_1, \xi_2,\xi_3$ for 
our cubic nonlinearity, the output will be at frequency 
\[
\xi_4 = \xi_1-\xi_2+\xi_3.
\]
This relation can be described in a more symmetric fashion as 
\[
\Delta^4 \xi = 0, \qquad \Delta^4 \xi := \xi_1-\xi_2+\xi_3-\xi_4 .
\]
We remark that our choice of the dyadic step $\step$ so that $\cc-1 \ll 1$ serves to insure
that when three of the frequencies are in the same (or nearby) dyadic regions, then the fourth
must also be in a nearby dyadic region.

This is a resonant interaction if and only if we have a similar relation for the associated time frequencies, namely 
\[
\Delta^4 a(\xi) = 0, \qquad \Delta^4 a(\xi) := a(\xi_1)-a(\xi_2)+a(\xi_3)-a(\xi_4).
\]
Hence, we define the resonant set in a symmetric fashion as 
\[
\calR := \{ \Delta^4 \xi = 0, \ \Delta^4 a(\xi) = 0\}.
\]
Due to the strict convexity of $a$, it is easily seen that this set may be characterized as
\[
\calR = \{ \{\xi_1,\xi_3\} = \{\xi_2,\xi_4\}\}.
\]
We will also distinguish the subset 
\[
\calR_2 = \{\xi_1= \xi_3 = \xi_2 = \xi_4 \},
\]
which we will call the \emph{doubly resonant set}. We will use the (overlapping) classification of resonant interactions into 
\begin{description}
    \item[(i) balanced] if $\xi_1$ and $\xi_3$ are in the same or in adiacent dyadic regions,
    \item[(ii) semi-balanced] $\xi_1$ and $\xi_3$ are not in the same or in adiacent dyadic regions, but are still comparable in size and with the same sign, and
    \item[(iii) unbalanced] if either $\langle \xi_1\rangle \ll \langle \xi_3\rangle$ or viceversa, or if 
    the two signs are different.
\end{description}
We will also extend this terminology to bilinear interactions as needed.

\subsection{A decomposition of the nonlinearity}
As we consider the evolution \eqref{eq:main},
it is important to focus on the important, nonperturbative interactions, while peeling off perturbatively the other interactions. The intuition is that the worst interactions are 
the doubly resonant ones, where all three input frequencies are close. The remaining interactions
we would like to think of as perturbative, either 
directly or after suitable normal form type considerations; here the fact  that we consider 
semilinear problems definitely plays a role.

Based on the above heuristic discussion, we will 
decompose the cubic nonlinearity into two parts
\begin{equation}
C(u,\bu,u) = C^{bal}(u,\bu,u) + C^{tr}(u,\bu,u) ,  
\end{equation}
where the two components are chosen as follows:
\begin{itemize}
    \item The balanced part $C^{bal}$ has symbol supported in the region where $|\xi_j - \xi_k| \lesssim (\step -1)\la \xi \ra$.
 \item The transversal part $C^{tr}$ has symbol supported in the region where at least two of the three frequencies are dyadically separated.
\end{itemize}
To make this precise, we define 
\[
c^{bal}(\xi_1,\xi_2,\xi_3) := 
c(\xi_1,\xi_2,\xi_3) c_{diag}(\xi_1,\xi_2,\xi_3)
\]
where $c_{diag}$ is a cutoff with properties as follows:
\begin{itemize}
    \item $c_{diag}= 1$ in the region $|\xi_j - \xi_k| \leq 4(\step-1) \la \xi_j \ra$.
     \item $c_{diag}$ is supported in the region $|\xi_j - \xi_k| \leq 8(\step-1) \la \xi_j \ra$.
     \item $c_{diag}$ is smooth on the dyadic scale.
\end{itemize}


\subsection{ Galilean type transformations}
\label{s:galilei}
These have played a significant role in our first
article \cite{IT-global} proving the conjecture
in the setting of the Schr\"odinger type dispersion relation, allowing us to both streamline the arguments and to formulate both the result and the 
entire argument in a Galilean invariant fashion.
This is no longer possible in this article, with a general dispersion relation. Nevertheless, it will 
still be useful and instructive to consider a more 
form of Galilean symmetry, namely a linear change 
of coordinates
\[
 x \to y = x+v t,
\]
without a matching frequency shift. This changes the equation \eqref{eq:main} to 
\begin{equation}
i \partial_t u - (A(D_y)-v D_y) u = C(u,\bu,u).      \end{equation}
Here the nonlinearity is left unchanged, but 
the dispersion relation $a$ is changed to 
\begin{equation}\label{change-a}
a(\xi) \to a(\xi) - v \xi.     
\end{equation}
The straightforward observation later on is that 
such a transformation changes density flux identities simply by adding in the advection velocity $v$, while it does not change at all our interaction Morawetz identities, which are integrated spatially. In particular, this observation will allow us to provide a shorter proof for Lemma~\ref{l:j6-diag}.

The final remark here is that such a transformation 
does not change at all our assumption \textbf{(A1)} on the 
dispersion relation $a$. Instead, it shows that 
we do not need to require any bounds  on $a$ 
or $a'$, just on $a''$.

\section{Strichartz and bilinear \texorpdfstring{$L^2_{t,x}$}{} 
bounds for the linear flow}
\label{s:Strichartz}
\
Here we begin by recalling  the  Strichartz inequalities,  which apply to solutions to the inhomogeneous linear evolution:
\begin{equation}\label{bo-lin-inhom}
(i\partial_t - A(D))u = f, \qquad u(0) = \du_0.
\end{equation}
These are $L^p_t L^q_x$ bounds which are used as measures of the dispersive effect. 

\begin{definition}
The pair $(p,q)$ is a sharp Strichartz exponent in one space dimension if
\begin{equation}
\frac{2}{p}+\frac{1}q = \frac{1}{2}, \qquad 2\leq q \leq \infty   . \end{equation}

\end{definition}
With this definition, the  frequency localized Strichartz estimates in the $L^2$ setting are as follows 

\begin{lemma}
Assume that $u_{\lambda}$ solves \eqref{bo-lin-inhom} in $[0,T] \times \R$. Then the following estimate holds for all sharp Strichartz pairs $(p,q), (p_1,q_1)$:
\begin{equation}
\label{strichartz}
\lambda^{\frac{\gamma}{p}}\| u_{\lambda}\|_{L^p_t L^q_x} \lesssim \|\du_{0\lambda} \|_{L^2} +  \lambda^{-\frac{\gamma}{p}}\|f_{\lambda}\|_{L^{p'_1}_t L^{q'_1}_x}.
\end{equation}
\end{lemma}

The proof follows in a standard fashion from the dispersive decay estimate
\begin{equation}
\label{eq:dispersion}
|\Vert e^{itA(D)}u_{\lambda}\Vert_{L^{\infty}}\leq \lambda^{-\frac{\gamma}{2}}t^{-\frac{1}{2}}.
\end{equation}
This estimate follows in turn from a stationary phase argument using the convexity property $a''(\xi) \approx \la\xi\ra^{\gamma}$. The details are left for the reader.
 
 \bigskip
 
 For practical purposes it is useful to place all 
 these estimates under a single umbrella, defining the Strichartz space $S_\lambda$ associated to the $L^2$ flow at frequency $\lambda$ by
\[
S_{\lambda} := L^\infty_t L^2_x \cap \lambda^{-\frac{\gamma}{4}} L^4_t L^{\infty}_x,
\]
as well as its dual 
\[
S'_{\lambda} = L^1_t L^2_x + \lambda^{\frac{\gamma}{4}}L^{\frac{4}{3}} _t L^{1}_x .
\]

Then the Strichartz estimates can all be summarized  as  
\begin{equation}
\label{strichartz3}
\| u_{\lambda}\|_{S_{\lambda}} \lesssim \|\du_{0\lambda} \|_{L^2} + \|f_{\lambda}\|_{S'_{\lambda}}.
\end{equation}

 For a more robust choice, it will be convenient to use the $U^p_A$ and $V^p_A$ spaces. These spaces, adapted to the linear $A$ flow, first were introduced in unpublished work of the second author~\cite{T-unpublished}, in order to have better access to both Strichartz estimates and bilinear $L^2_{t,x}$ bounds,  in particular without losing endpoint estimates; see also \cite{KT} and \cite{HTT} for some of the first applications of these spaces. 
 
 In this setting the 
linear estimates for the inhomogeneous problem \eqref{bo-lin-inhom} can be summarized as
\begin{equation}
\label{strichartz2}
\| u\|_{U^2_A} \lesssim \|\du_0 \|_{L^2} + \|f\|_{DU^2_A},
\end{equation}
where the transition to the Strichartz bounds is provided by the embeddings at fixed frequency $\lambda$ 
\begin{equation}\label{UV-embed}
V^2_A \subset U^p_A \subset \lambda^{-\frac{\gamma}p} L^p_t L^q_x,
\qquad \lambda^{\frac{\gamma}p} L^{p'}_t L^{q'}_x \subset DV^{p'}_A \subset DV^2_A.
\end{equation}

\medskip
The last property of the linear evolution \eqref{bo-lin-inhom} equation we will use here is the bilinear $L^2_{t,x}$ estimate, which is as follows:

\begin{lemma}
\label{l:bi}
Let $u^{\pm}_{\lambda}$, $v^{\pm}_{\mu}$ be two  functions which are frequency localized at  spatial frequencies $\lambda$ and $\mu$.  

a) If $\lambda$ and $\mu$ are  dyadically separated then we have 
\begin{equation}
\label{bi-di}
\| u^{\pm}_{\lambda} \bv^{\pm}_{\mu}\|_{L^2_{t,x}} \lesssim  (\lambda^{\gamma+1}+\mu^{\gamma+1})^{-\frac12} 
\|u_{\lambda}^{\pm} \|_{ U^2_A } \|v^{\pm}_{\mu} \|_{ U^2_A}
\end{equation}
with any combination of signs. Further, for the case $\gamma < -1$ and mismatched signs
we have the improvement 
\begin{equation}
\label{bi-di+}
\| u^{+}_{\lambda} \bv^{-}_{\mu}\|_{L^2_{t,x}} \lesssim   
\|u_{\lambda}^{+} \|_{ U^2_A } \|v^{-}_{\mu} \|_{ U^2_A}.
\end{equation}

b) If $\lambda$ and $\mu$ are comparable,  then we have
\begin{equation}
\label{bi-di-bal}
\| \partial_x (u^{\pm}_{\lambda} \bv^{\pm}_{\mu})\|_{L^2_{t,x}} \lesssim  \lambda^{-\frac{\gamma-1}2 }
\|u_{\lambda}^{\pm} \|_{ U^2_A } \|v^{\pm}_{\mu} \|_{ U^2_A}
\end{equation}
with matched signs.

\end{lemma}

These estimates are not so much connected to dispersion, and should be instead  thought of 
as transversality estimates. The constant appearing in the estimates is exactly $(\delta v)^{-\frac12}$, where $\delta v$ represents the difference between the group velocities associated 
to the respective frequency localizations. Finally, the derivative in \eqref{bi-di-bal} accounts for the vanishing of the  group velocity differential 
as the two frequencies get closer.

\begin{proof}
The proof of the lemma is relatively standard, so we just review the main steps.

\medskip

Step 1: Since $U^2_A$ is an atomic space, it suffices to prove the estimates in the case when $u_\lambda$ and $v_\mu$ are $U^2_A$ atoms.

\medskip

Step 2: The $U^2_A$ atoms are $l^2$ concatenations of solutions for the homogeneous equation, allowing us 
to further reduce to the case when $u^\pm_\lambda$ and $v^\pm_\mu$ are solutions to the homogeneous linear equation.
Denoting their initial data by $\du^\pm_{0\lambda}$ respectively $\dv^\pm_{0\mu}$, the estimate \eqref{bi-di} becomes
\begin{equation}
\label{bi-di-hom}
\| u^{\pm}_{\lambda} \bv^{\pm}_{\mu}\|_{L^2_{t,x}} \lesssim  (\lambda^{\gamma+1}+\mu^{\gamma+1})^{-\frac12} 
\|\du_{0\lambda}^{\pm} \|_{ L^2 } \|\dv^{\pm}_{0\mu} \|_{L^2},
\end{equation}
with similar modifications for \eqref{bi-di+} and \eqref{bi-di-bal}.

\medskip

Step 3: The last estimate can be now viewed as a convolution estimate in the Fourier space.
The spatial Fourier transform of $u_{\lambda}^{\pm}$ is
\[
\widehat{u}_{\lambda}^{\pm} (t, \xi)= e^{ita(\xi)}\widehat\du^{\pm}_{0\lambda}(\xi). 
\]
Then the space-time Fourier transform of $u_{\lambda}^{\pm}$ is
\[
\widehat{u}_{\lambda}^{\pm} (\tau, \xi)= \delta_{\tau =a(\xi)}\widehat\du^{\pm}_{0\lambda}(\xi). 
\]
Similarly, the Fourier transform of $\bv_\mu^{\pm}$ is
\[
\widehat{\bv}_{\mu}^{\pm} (\tau, \xi)= \delta_{\tau =-a(-\xi)}\overline{\widehat\dv}^{\pm}_{0\mu}(-\xi).
\]
This leads to
\[
\begin{aligned}
\widehat{u_{\lambda}^{\pm}\bv_{\mu}^{\pm} } (\tau, \zeta)= \widehat{u}_{\lambda}^{\pm}\ast \widehat{\bv}_{\mu}^{\pm} (\tau, \zeta) &= \int_{\zeta=\xi+\eta}\widehat{u}^{\pm}_{\lambda}(\tau, \xi) \ast_{\tau}\widehat{\bv}_{\mu}^{\pm} (\tau, \eta)\, d\eta\\
&= \int_{\zeta=\xi-\eta}\delta_{\tau =a(\xi)-a(\eta)}\widehat\du^{\pm}_{0\lambda}(\xi) \overline{\widehat\dv}^{\pm}_{0\mu}(\eta)\, d\eta.
\end{aligned}
\]
Applying both sides to a test function $w$ we get
\[
\widehat{u_{\lambda}^{\pm}\bv_{\mu}^{\pm} }(w)
= \int w(a(\xi) - a(\eta),\xi-\eta) \widehat\du^{\pm}_{0\lambda}(\xi) \overline{\widehat\dv}^{\pm}_{0\mu}(\eta)\, d\xi d\eta.
\]
The map 
\[
(\xi,\eta) \to (\tau,\zeta):= (a(\xi) - a(\eta),\xi-\eta)
\]
is a local diffeomorphism away from $\xi = \eta$,
with Jacobian $|a'(\xi) - a'(\eta)|$, therefore
changing variables we have 
\[
\widehat{u_{\lambda}^{\pm}\bv_{\mu}^{\pm} }(w)
= \int w(\tau,\zeta) \frac{1}{|a'(\xi)-a'(\eta)|}\widehat\du^{\pm}_{0\lambda}(\xi) \overline{\widehat\dv}^{\pm}_{0\mu}(\eta) \, d\tau d\zeta,
\]
which finally leads to 
\[
\widehat{u_{\lambda}^{\pm}\bv_{\mu}^{\pm} } (\tau, \zeta) = \frac{1}{|a'(\xi)-a'(\eta)|}\widehat\du^{\pm}_{0\lambda}(\xi) \overline{\widehat\dv}^{\pm}_{0\mu}(\eta).
\]
Thus,
\[
\Vert \widehat{u_{\lambda}^{\pm}\bv_{\mu}^{\pm} }  \Vert^2_{L^2} = \int  \frac{1}{|a'(\xi)-a'(\eta)|^2}|\widehat\du^{\pm}_{0\lambda}(\xi) |^2 |\overline{\widehat\dv}^{\pm}_{0\mu}(\eta)|^2 \, d\tau d\zeta = \int  \frac{1}{|a'(\xi)-a'(\eta)|}|\widehat\du^{\pm}_{0\lambda}(\xi) |^2 |\overline{\widehat\dv}^{\pm}_{0\mu}(\eta)|^2 \, d\xi d\eta.
\]
When $\xi\approx \pm \lambda$ and $\eta\approx \pm \mu$ are dyadically separated 
\[
|a'(\xi)-a'(\eta)| \approx \lambda^{\gamma+1}+\mu^{\gamma+1},
\]
provided that either $\gamma > -1$ or that $\gamma < -1$ and the signs are matched. In the remaining case when $\gamma < -1$ and the signs are mismatched
then we have instead  
\[
|a'(\xi)-a'(\eta)| \approx 1. 
\]
This yields the bounds \eqref{bi-di} and \eqref{bi-di+}. For the remaining bound \eqref{bi-di-bal} the same argument applies, after noting that the 
derivative yields an additional $\xi-\eta$ factor in the above computation.

\end{proof}

For the local well-posedness result we also need the following variant of \eqref{bi-di}:

\begin{lemma}\label{l:bi-V}
Let $u^{\pm}_{\lambda}$, $v^{\pm}_{\mu}$ be two  functions which are frequency localized at  spatial frequencies $\lambda$ and $\mu$.   If $\lambda$ and $\mu$ are  dyadically separated then we have 
\begin{equation}
\label{bi-di-V}
\| u^{\pm}_{\lambda} v^{\pm}_{\mu}\|_{L^2_{t,x}[0,T]} \lesssim  T^\sigma \min\{ \lambda^{-\sigma \gamma}, \mu^{-\sigma \gamma}\}(\lambda^{\gamma+1}+\mu^{\gamma+1})^{-\frac{1-4\sigma}2}
\|u_{\lambda}^{\pm} \|_{ U^2_A } \|v^{\pm}_{\mu} \|_{ V^2_A}, \qquad 0 < \sigma \leq \frac14,
\end{equation}
with any combination of signs.
\end{lemma}
Similar variants for \eqref{bi-di+} and \eqref{bi-di-bal} are also valid, but will not be needed.
\begin{proof}
The proof is a simple interpolation argument. Using 
one $L^4_t L^\infty_x$ Strichartz estimate and one energy bound we get
\[
\| u^{\pm}_{\lambda} v^{\pm}_{\mu}\|_{L^4_tL^2_{x}} \lesssim  \min\{ \lambda^{-\frac{\gamma}4}, \mu^{-\frac{\gamma}4}\}
\|u_{\lambda}^{\pm} \|_{ U^4_A } \|v^{\pm}_{\mu} \|_{ U^4_A},
\]
which by H\"older's inequality gives
\[
\| u^{\pm}_{\lambda} v^{\pm}_{\mu}\|_{L^2_{t,x}} \lesssim  T^\frac14 \min\{ \lambda^{-\frac{\gamma}4}, \mu^{-\frac{\gamma}4}\}
\|u_{\lambda}^{\pm} \|_{ U^4_A } \|v^{\pm}_{\mu} \|_{ U^4_A}.
\]
We now interpolate this with \eqref{bi-di} to obtain
\begin{equation}
 \| u^{\pm}_{\lambda} v^{\pm}_{\mu}\|_{L^2_{t,x}} \lesssim  T^\sigma  \min\{ \lambda^{-\sigma \gamma}, \mu^{-\sigma \gamma}\}(\lambda^{\gamma+1}+\mu^{\gamma+1})^{-\frac{1-4\sigma}2} 
\|u_{\lambda}^{\pm} \|_{ U^r_A } \|v^{\pm}_{\mu} \|_{ U^r_A}, 
\end{equation}
where $r$ is between $2$ and $4$,
\[
\qquad \frac1r = \sigma + \frac{1-4\sigma}2.  \]
Since $U^r_A \subset V^2_A$, this implies \eqref{bi-di-V}.

\end{proof}

\section{Local well-posedness }
\label{s:local}

This section is devoted to the proof of the local well-posedness result in Theorem~\ref{t:local}. Depending on $\gamma$ we 
split the proof into two parts. If $\gamma > -2$ then we will take advantage of the dispersive properties of the linear flow. On the other hand if $\gamma \leq -2$, then there is no dispersion on short times, so we will largely 
treat the equation \eqref{eq:main} as an ode in time.  In both cases, we will treat
the nonlinearity perturbatively in appropriate function spaces.

\subsection{The dispersive case \texorpdfstring{$\gamma \geq -2$}{}}

The first step in our proof is to select a suitable function space $X^s$ where we seek the solutions, as well as a space $Y^s$ where we place the source terms. 
Here it is most efficient to use $U^2_A$ type spaces for the solutions, respectively the $DU^2_A$ type spaces for the source terms.  In our setting we will employ a slight variation of this, namely  the $\ell^2 U^{2,s}_A$ spaces for the solutions, respectively the $\ell^2 DU^{2,s}_A$ for the source terms,
\[
X^s = \ell^2 U^{2,s}_A, \qquad Y^s = \ell^2 D U^{2,s}_A.
\]
Here the $\ell^2$ summation corresponds exactly to our spatial Littlewood-Paley decomposition, 
and $s$ indicates the Sobolev regularity, 
\[
\| u \|_{X^s}^2 = \sum_{\pm} \sum_\lambda 
\lambda^{2s} \| u^\pm_\lambda \|_{U^2_A}^2,
\]
and similarly for $Y^s$.

\medskip

In order to prove the local well-posedness result perturbatively in this setting, arguing in a standard fashion via the contraction principle, we need to have the following two properties in a time interval $[0,T]$:

\begin{enumerate}[label=(\roman*)]
    \item A bound for the inhomogeneous linear flow
    \[
(i \partial_t + A(D))u = f, \qquad u(0) = \du_0,
    \]
    namely
\begin{equation}\label{lwp-lin}
\| u\|_{X^s} \lesssim \|\du_0\|_{H^s} + \| f\|_{Y^s}  .   
\end{equation}

    \item A bound for the nonlinearity,
 \begin{equation}\label{lwp-nlin}
 \| C(u,\bu,u)\|_{Y^s} \lesssim T^{\sigma} \|u\|_{X^s}^3, \qquad \sigma > 0.    
 \end{equation}   
\end{enumerate}
The small power of $T$ in the second estimate is needed in order to gain smallness for short times $T$, which is required for the large data result.

The linear property \eqref{lwp-lin}
is standard, so our goal in what follows will be to prove the cubic bound \eqref{lwp-nlin}. Using duality, 
this reduces to the quadrilinear bound
\begin{equation}
 \left| \int C(u_1,\bu_2,u_3) \cdot \bu_4 \, dx dt\right | \lesssim T^{\sigma} \|u_1\|_{\ell^2 U^{2,s}_A}
  \|u_2\|_{\ell^2 U^{2,s}_A}
 \|u_3\|_{\ell^2 U^{2,s}_A}
 \|u_4\|_{\ell^2 V^{2,-s}_A},
\end{equation}
where we have decoupled the three entries of $C$. We use a Littlewood-Paley decomposition for each of the factors,  with frequencies denoted by $\lambda_j$, $j = 1,\cdots,4$,
\[
\int C(u_1,\bu_2,u_3) \cdot \bu_4\, dx dt = 
\sum_{\lambda_1,\lambda_2,\lambda_3,\lambda_4}
\int C(u_{1,\lambda_1},\bu_{2,\lambda_2},u_{3,\lambda_3}) \cdot \bu_{4,\lambda_4} \, dx dt := \sum_{\lambda_1,\lambda_2,\lambda_3,\lambda_4} \bC^4_{\lambda_1,\lambda_2,\lambda_3,\lambda_4}.
\]
Here  the two largest frequencies must be comparable.
Then we have several cases
to consider, depending on the balance of the four frequencies. The main constraint will arise from balanced interactions, so we consider this case first.
\bigskip

\emph{ a) Balanced interactions.}
Here we consider the case when all four frequencies are comparable to a dyadic scale $\lambda$. For this we will use the Strichartz 
norms via the embeddings in~\eqref{UV-embed},
\[
\begin{aligned}
|\bC^{4}_{\lambda,\lambda,\lambda,\lambda}|\lesssim &\ 
 T^{\frac12} \lambda^{3\delta}\|u_{1,\lambda}\|_{L^6}   \|u_{2,\lambda}\|_{L^6} \|u_{3,\lambda}\|_{L^6} \|u_{4,\lambda}\|_{L^\infty L^2} 
 \\
 \lesssim & \ 
 T^\frac12 \lambda^{3\delta-2s-\frac12 \gamma} 
  \|u_{1,\lambda}\|_{\ell^2 U^{2,s}_A}
  \|u_{2,\lambda}\|_{\ell^2 U^{2,s}_A}
 \|u_{3,\lambda}\|_{\ell^2 U^{2,s}_A}
 \|u_{4,\lambda}\|_{\ell^2 V^{2,-s}_A}.
\end{aligned}
\]
The power of $\lambda$ must be non-positive, so this  yields the constraint
\[
s \geq \frac32 \delta - \frac14 \gamma = s_{LWP},
\]
as needed, in which case the dyadic $\lambda$ summation is straightforward.

\bigskip

\emph{b) Unbalanced interactions.}
We now consider off-diagonal interactions, where we are also able to use bilinear $L^2_{t,x}$ bounds.  
Since the two highest frequencies must be equal,
we denote our four frequencies by $(\alpha,\lambda,\mu,\mu)$ with
\[
\alpha, \lambda < \mu.
\]
Here, due to the choice of the dyadic step $\step$, we can assume strict inequality, i.e. dyadic separation between $\alpha,\lambda$ and $\mu$.
A-priori there are two possible ways to distribute these frequencies:

\medskip

b1) $u_4$ has the highest frequency $\mu$. Then we pair the factors with frequencies $(\alpha,\mu)$ and $(\lambda,\mu)$, using Lemma~\ref{l:bi} for the  $U^2 \times U^2$ product, respectively Lemma~\ref{l:bi-V} for the $U^2 \times V^2$ product. We separate two cases depending on $\gamma$:

\medskip

b1(i) $\gamma > -1$. Then we obtain
\[
\begin{aligned}
|\bC^4_{\alpha,\lambda,\mu,\mu}|\lesssim &\ T^\sigma 
\mu^{\sigma(\gamma+2)}
(\alpha \mu \lambda)^{\delta} (\alpha \lambda)^{-s} \mu^{-\gamma -1} \|u_{1,\alpha}\|_{\ell^2 U^{2,s}_A}
  \|u_{2,\lambda}\|_{\ell^2 U^{2,s}_A}
 \|u_{3,\mu}\|_{\ell^2 U^{2,s}_A}
 \|u_{4,\mu}\|_{\ell^2 V^{2,-s}_A},
\end{aligned}
\]
with the same outcome if the $V^2$ norm is on another entry. Here we note the factor
\[
\mu^{\sigma(\gamma+2)} (\alpha \lambda)^{\delta -s} \mu^{\delta-\gamma-1}.
\]
If $\sigma=0$ then the sum of the three exponents is nonpositive as long as $s \geq s_c$. 
As $\sigma$ has to be positive but can be chosen arbitrarily small, in order
to guarantee the dyadic summation, we 
simply need a negative power for $\mu$, which yields the requirement
\[
\delta <  \gamma+1,
\]
and which is in turn guaranteed by the range \eqref{delta-main}. 

\medskip

b1(ii) $-2 <  \gamma < -1$. Then,
using the same bilinear pairing and  Lemmas~\ref{l:bi}, \ref{l:bi-V} we obtain
\[
\begin{aligned}
|\bC^4_{\alpha,\lambda,\mu,\mu}|\lesssim &\ T^\sigma 
\mu^{\sigma(\gamma+2)} (\alpha \mu \lambda)^{\delta} (\alpha \lambda)^{-s} (\alpha \lambda)^{-\frac12(\gamma +1)} \|u_{1,\alpha}\|_{\ell^2 V^{2,-s}_A}
  \|u_{2,\lambda}\|_{\ell^2 U^{2,s}_A}
 \|u_{3,\mu}\|_{\ell^2 U^{2,s}_A}
 \|u_{4,\mu}\|_{\ell^2 V^{2,-s}_A},
\end{aligned}
\]
again with no change if the $V^2$ norm is on another
entry. We now have the factor
\[
\mu^{\sigma(\gamma+2)} (\alpha \lambda)^{\delta -s-\frac12(\gamma +1)} \mu^{\delta}.
\]
As before, the sum of the exponents
vanishes if $\sigma = 0$ and $s = s_c$, so
in order to guarantee the dyadic summation for $s > s_c$ we  again need a negative power for $\mu$, which yields the requirement
\[
\delta <  0,
\]
and which is also guaranteed by the range \eqref{delta-main}.

\medskip 

b2) $u_4$ has a lower frequency, say $\lambda$.
Again we pair 
the factors with frequencies $(\alpha,\mu)$ and $(\lambda,\mu)$ and use  Lemma~\ref{l:bi} and Lemma~\ref{l:bi-V} in the same two cases:

\medskip

b2(i) $\gamma > -1$.
Then we obtain
\[
\begin{aligned}
|\bC^4_{\alpha,\mu,\mu,\lambda}|\lesssim &\ T^\sigma 
\mu^{\sigma(\gamma+2)} (\alpha \mu^2)^{\delta} \alpha^{-s} \lambda^{s} \mu^{-2s} \mu^{-\gamma -1} \|u_{1,\alpha}\|_{\ell^2 U^{2,s}_A}
  \|u_{2,\lambda}\|_{\ell^2 V^{2,-s}_A}
 \|u_{3,\mu}\|_{\ell^2 U^{2,s}_A}
 \|u_{4,\mu}\|_{\ell^2 U^{2,s}_A},
\end{aligned}
\]
where we have the factor
\[
\mu^{\sigma(\gamma+2)} \alpha^{\delta-s}  \lambda^{s} \mu^{-2s +2\delta - \gamma -1}.  
\]
To insure dyadic summation here we need the 
conditions 
\[
-2s + 2 \delta - \gamma - 1 <  0, \qquad
-3s +3 \delta -\gamma -1 < 0, \qquad -s+ 2\delta - \gamma - 1 < 0,
\]
where the second is needed for $\alpha > \lambda$  and the third is needed for $\alpha < \lambda$. But these are easily verified to hold 
strictly in our range  $s \geq s_{LWP}$ for all
$0 \leq \delta < \gamma+1$.

\medskip

b2(ii) $-2 \leq  \gamma < -1$.
Then we obtain
\[
\begin{aligned}
|\bC^4_{\alpha,\mu,\mu,\lambda}|\lesssim &\ T^\sigma \mu^{\sigma(\gamma+2)} (\alpha \mu^2)^{\delta} \alpha^{-s} \lambda^{s} \mu^{-2s} (\alpha \lambda)^{-\frac12(\gamma +1)} \|u_{1,\alpha}\|_{\ell^2 U^{2,s}_A}
  \|u_{2,\lambda}\|_{\ell^2 V^{2,-s}_A}
 \|u_{3,\mu}\|_{\ell^2 U^{2,s}_A}
 \|u_{4,\mu}\|_{\ell^2 U^{2,s}_A},
\end{aligned}
\]
where we have the factor
\[
\mu^{\sigma(\gamma+2)} \alpha^{\delta-s-\frac{\gamma+1}{2}}  \lambda^{s-\frac{\gamma+1}2} \mu^{-2s +2\delta}.  
\]
To insure dyadic summation here we need the 
conditions 
\[
-2s + 2 \delta \leq 0, \qquad
-3s +3 \delta -\frac{\gamma +1}2 < 0, \qquad -s+ 2\delta - \frac{\gamma + 1}2 < 0,
\]
where the second is needed for $\alpha > \lambda$  and the third is needed for $\alpha < \lambda$. But these are also easily verified to hold strictly in our range  $s \geq s_{LWP}$ for all $\frac23(\gamma+1) \leq   \delta \leq 0$.

\bigskip

\subsection{The non-dispersive case \texorpdfstring{$\gamma \leq -2$}{}}
Here the dispersion is nonexistent on short time scales, so we simply need to know that $C$ is bounded in $H^s$,
\[
C: H^s \to H^s.
\]
By duality this reduces to the fixed time bound
\begin{equation}\label{C-bounded}
\left|\int C(u_1,u_2,u_3) \bar u_4 \, dx\right| \lesssim \|u_1\|_{H^s} \|u_2\|_{H^s} \|u_3\|_{H^s} \|u_4\|_{H^{-s}}.
\end{equation}

We consider a dyadic decomposition as before,
with four frequencies $\lambda_1$, $\lambda_2$,
$\lambda_3$ and $\lambda_4$. The two highest frequencies must be comparable, so we denote them by
$\alpha,\lambda,\mu,\mu$ with 
\[
\alpha,\lambda \lesssim  \mu.
\]
We apply Bernstein's inequality to estimate the $L^\infty$ norm  for the two lowest frequencies and consider two cases:
\medskip

a) $\lambda_4 = \mu$. Then we obtain \eqref{C-bounded} with the implicit constant 
\[
(\alpha \lambda)^{-s+\frac12} (\alpha \lambda\mu)^\delta ,
\]
in which case the dyadic summation simply requires 
\[
\delta \leq 0, \qquad s \geq \frac12+\delta,
\]
with at least one strict inequality.
\medskip

b) $\lambda_4 = \lambda$. Then we we obtain \eqref{C-bounded} with the implicit constant 
\[
\alpha^{-s+\delta+\frac12}  \lambda^{s+\frac12} \mu^{-2s+2\delta} ,
\]
in which case the dyadic summation requires 
\[
-2s+2\delta \leq 0, \qquad -3s+3\delta + \frac12 <  0, \qquad -s+2\delta +\frac12 < 0,\qquad  -2s+3\delta+1 \leq 0. 
\]
If $\delta \in (-\frac23,0]$ then the last constraint, which is the scaling constraint, 
dominates, and we arrive at 
\[
s \geq s_{LWP} := \frac{3\delta+1}2,
\]
with the endpoint excluded at $s=0$,
as needed. 

But if $\delta \leq  -\frac23$ then it is the second one which dominates, 
and we need
\[
s > s_{LWP} : = \delta+\frac16.
\]

\section{Energy estimates and conservation laws}
\label{s:energy}

\subsection{Conservation laws for the linear problem}
We begin our discussion with the linear evolution
\begin{equation}
i u_t - A(D) u = 0, \qquad u(0) = \du_0.    
\end{equation}
The $L^2$ norm of the solution, which we will refer to as the \emph{mass}, is conserved,
\[
\frac{d}{dt} \bM(u) = 0, \qquad \bM(u) := \int |u|^2\,  dx.
\]

We will work instead with the associated  density 
\[
M(u) := |u|^2,
\]
which we view as a bilinear form with symbol 
\[
m(\xi,\eta) = 1.
\]
We look for a corresponding associated flux, to be called \emph{the momentum density}, by computing
\[
\partial_t M(u) = - iA(D) u \bar u + i u \overline{A(D) u} ,
\]
where the right hand side is a bilinear form with symbol $ - i( a(\xi) - a(\eta))$. On the other hand,
a derivative applied to a bilinear form  in $u$ and $\bar{u}$ contributes to the symbol with a factor of $i(\xi-\eta)$. Hence it is natural to divide the two and write the above symbol in the form
\[
a(\xi) - a(\eta) = - (\xi-\eta) p(\xi,\eta),
\]
where $p$ is a smooth, symmetric symbol,
\begin{equation}\label{p}
    p(\xi,\eta) = - \frac{a(\xi)-a(\eta)}{\xi-\eta}.
\end{equation}
Then we can rewrite the above time derivative as
\begin{equation}\label{df-m-lin}
    \partial_t M(u)  =  \partial_x P(u).
\end{equation}
We call the argument on the right hand side the \emph{momentum density}. 

\medskip

Repeating the same computation for the momentum we have
\begin{equation}\label{df-p-lin}
\partial_t P(u) = \partial_x E(u),
\end{equation}
where the \emph{energy density} $E$ is a bilinear form in $(u,\bu)$ with symbol
\begin{equation}
e(\xi,\eta) = p^2(\xi,\eta).
\end{equation}It will also be helpful in our analysis to have 
a \emph{reverse momentum}, which we will denote by $\rP$, with symbol
\[
\rp(\xi,\eta) := \frac{1}{p(\xi,\eta)}.
\]
Formally this has the property that its flux is exactly the mass,
\begin{equation}
\partial_t \rP(u) = \partial_x M(u).
\end{equation}
However, the momentum symbol $p$ is not 
everywhere nonzero, so the reverse momentum is not globally defined. Because 
of this, we will have to be careful and use it only in frequency regions where the momentum symbol stays away from zero. 

We will use the above definitions directly 
in the GNLS case $\gamma > -1$. However, in the GKG case $\gamma < -1$, the momentum looks asymptotically like 
a multiple of the mass in the matched\footnote{i.e. with matched signs for $\xi$ and $\eta$.}  high frequency limit,
\[
p(\xi,\eta) \approx v^{\pm} + O((|\xi|+|\eta|)^{\gamma+1}), \qquad \xi,\eta \to \pm \infty.
\]
Then it becomes useful to eliminate 
the leading order term by considering 
instead a \emph{relative momentum density}
\begin{equation}
p^{\pm} = p - v^{\pm} m,
\end{equation}
where the $\pm$ choice is adapted to high positive, respectively negative frequencies.

Then the density-flux relation \eqref{df-m-lin} becomes 
\begin{equation}\label{df-m-lin-rel}
    (\partial_t + v^{\pm} \partial_x) M(u)  =  \partial_x P^\pm (u).
\end{equation}
Similarly, the density-flux relation for the momentum \eqref{df-p-lin} can be recast in the form  
\begin{equation}\label{df-p-lin-rel}
(\partial_t+ v^\pm \partial_x) P^\pm(u) = \partial_x E^\pm(u),
\end{equation}
where the \emph{relative energy density} has symbol
\begin{equation}
e^{\pm} = e - 2 v^{\pm} p + (v^{\pm})^2 m.  \end{equation}
In this case we will replace the reverse momentum with the \emph{relative reverse momentum}, with symbol
\begin{equation}
\rp^{\pm} = \frac{1}{p^{\pm}},   
\end{equation}
which in particular has the advantage that 
the denominator never vanishes. Its associated density-flux relation now has 
the form 
\begin{equation}\label{df-rp-lin-rel}
(\partial_t+ v^\pm \partial_x) \rP^\pm(u) = \partial_x M(u).
\end{equation}
\bigskip

In our analysis we will not use directly the global mass, momentum, energy and reverse momentum, but instead we will need frequency localized counterparts, where for the purpose of this paper, we will always localize on the dyadic frequency scale. So we start with a localization symbol $\psi^\pm_\lambda(\xi,\eta)$ associated to a dyadic frequency $\lambda$, where 
the sign allows us to distinguish between positive and negative frequencies.
This can be taken of product type, 
\[
\psi_\lambda(\xi,\eta) = \phi_\lambda(\xi) \phi_\lambda(\eta),
\]
with matched signs. To avoid cluttering 
the notations, here we omit using an additional $\pm$ superscript to indicate
whether we localize to a positive or negative dyadic region. This choice will not be  important for the most part, and will be directly specified where it matters.

At this point we
distinguish again between the two cases
for $\gamma$:

\medskip

A. The GNLS case $\gamma > -1$.
In this case we will work with the localized densities
\[
m_\lambda(\xi,\eta) := m(\xi,\eta) \psi_\lambda(\xi,\eta), \quad
p_\lambda(\xi,\eta) := p(\xi,\eta) \psi_\lambda(\xi,\eta), \qquad
e_\lambda(\xi,\eta) := e(\xi,\eta) \psi_\lambda(\xi,\eta).
\]
Away from frequency zero $\lambda \gg 1$ we will also use the reverse momentum
\[
\rp_\lambda(\xi,\eta) := \rp(\xi,\eta) \psi_\lambda(\xi,\eta).
\]
Then we have the density-flux relations
\begin{equation}
\partial_t M_\lambda(u) = \partial_x P_\lambda(u),
\qquad
\partial_t P_\lambda(u) = \partial_x E_\lambda(u),
\qquad
\partial_t \rP_\lambda(u) = \partial_x M_\lambda(u),
\end{equation}
with the restriction $\lambda \gg 1$ for the last one.
It is also useful to examine the size of these symbols, where we  have
\begin{equation}
m_\lambda \in S^{0}_\lambda, \qquad 
p_\lambda \in S^{\gamma+1}_\lambda, \qquad 
e_\lambda \in S^{2(\gamma+1)}_\lambda,
\qquad \rp_\lambda \in S^{-(\gamma+1)}_\lambda, \qquad \gamma > -1.
\end{equation}
In this case we will not need to distinguish between positive and negative frequencies.

\bigskip

\emph{B. The generalized KG  
case $\gamma < -1$.}
Here the above symbol bounds are no longer directly valid, as the momentum symbol
will pick up a leading term $v^{\pm}$
as frequencies go to $\pm \infty$.
To account for this, we will work instead with relative densities, and
distinguish between positive and negative frequencies for the momentum and the energy. Precisely, for positive frequencies
we define our densities relative to the $v^+$ velocities, while for negative 
velocities we define our densities relative to the $v^-$ velocities,
\[
m_\lambda(\xi,\eta) := m(\xi,\eta) \psi_\lambda(\xi,\eta), \quad
p_\lambda^\pm(\xi,\eta) := p^\pm(\xi,\eta) \psi_\lambda(\xi,\eta), \qquad
e_\lambda^\pm(\xi,\eta) := e^\pm(\xi,\eta) \psi_\lambda(\xi,\eta),
\]
respectively
\[
\rp^\pm_\lambda(\xi,\eta) := \rp^\pm(\xi,\eta) \psi_\lambda(\xi,\eta).
\]
For bounded frequencies $\lambda \lesssim 1$ we have the freedom to use either of
the two choices, as needed, and the relative reverse momentum is always well defined.

Then we have the density-flux relations
\begin{equation}
\begin{aligned}
& (\partial_t+v^\pm \partial_x) M_\lambda(u) = \partial_x P^\pm_\lambda(u),
\qquad
(\partial_t+v^\pm \partial_x) P^\pm_\lambda(u) = \partial_x E^\pm_\lambda(u),
\\
& \qquad \qquad \qquad (\partial_t+v^\pm \partial_x) \rP^\pm_\lambda(u) = \partial_x M_\lambda(u). \end{aligned}
\end{equation}

As a bonus of using the relative quantities, we also gain the desired 
symbol regularity
\begin{equation}
m_\lambda \in S^{0}_\lambda, \qquad 
p^\pm_\lambda   \in S^{\gamma+1}_\lambda, \qquad 
e^\pm_\lambda  \in S^{2(\gamma+1)}_\lambda,
\qquad \rp^\pm_\lambda   \in S^{-(\gamma+1)}_\lambda.
\end{equation}

\subsection{Nonlinear density flux identities} 
Here we develop the counterpart of the  analysis above for the linear problem, in the context of the  nonlinear problem \eqref{eq:main}.

\subsubsection{The modified mass} 
To motivate what follows, we begin with a simpler computation for the $L^2$ norm of a solution $u$ of \eqref{eq:main}, which we recall from our previous work \cite{IT-global}: 
\[
\frac{d}{dt} \| u\|_{L^2}^2 = \int - i C(u,\bar u,u) \cdot \bar u  + i\cdot \ol{C(u,\bar u,u)} \cdot  u \, dx
:=   \int C^4_m(u,\bar u, u, \bar u) \, dx.
\]
A-priori the symbol of the quartic form $C^4_m$, defined on the diagonal $\Delta^4 \xi = 0$,
is given by
\[
c^4_m(\xi_1,\xi_2,\xi_3,\xi_4) = - i c(\xi_1,\xi_2,\xi_3) + i \bar c(\xi_2,\xi_3,\xi_4).
\]
However, we can further symmetrize and replace it by 
\[
c^4_m(\xi_1,\xi_2,\xi_3,\xi_4) = \frac{i}2 \left( -  c(\xi_1,\xi_2,\xi_3) - c(\xi_1,\xi_4,\xi_3)+ \bar c(\xi_2,\xi_3,\xi_4)+ 
\bar c(\xi_2,\xi_1,\xi_4) \right),
\]
which is now assumed to be defined everywhere. Then 
we can write the density-flux relation
\begin{equation}\label{df-m}
    \partial_t M(u,\bu)  =  \partial_x P(u,\bar u) + C^4_m(u,\bu,u,\bu).
\end{equation}

As a general principle, quartic flux terms cannot 
be directly estimated in general, globally in time, as the best bound we 
can hope for is an $L^6_{t,x}$ bound for $u$. However, there are 
two favourable scenarios:

\begin{itemize}
\item quartic \emph{null terms} which can be removed using  quartic  density-flux corrections, modulo 
six-linear errors,
\item quartic \emph{transversal terms}  where the four frequencies can be divided into two pairs of separated frequencies and where we can use two bilinear $L^2_{t,x}$
bounds.
\end{itemize}

Hence our strategy will be to place all of $C^4_m$ 
into one of the above two boxes. To achieve this, we 
will need to consider  the behavior of $c^4_m(\xi_1,\xi_2,\xi_3,\xi_4)$ first on the resonant set
\[
\calR := \{(\xi_1, \xi_2, \xi_3, \xi_4)\in \mathbb{R}^4 \, / \,\Delta^4 \xi = 0, \,  \Delta^4 a(\xi) = 0\} = \{ \{ \xi_1,\xi_3\} = \{\xi_2,\xi_4\} \},
\]
and more importantly on the doubly resonant set
\[
\calR_2 := \{(\xi_1, \xi_2, \xi_3, \xi_4)\in \mathbb{R},\  \xi_1 =\xi_2 =\xi_3=\xi_4 \}.
\]
In our first article \cite{IT-global}, where the global well-posedness
conjecture was proved for a semilinear Schr\"odinger model where $a(\xi) = \xi^2$, we made a simplifying assumption guaranteeing that $c^4_m = 0$ on $\calR$, which in turn led to a clean division property
\begin{equation}\label{easy-div}
c^4_m +  i  \Delta^4 a(\xi) \,  b^4_m = i \Delta^4 \xi\, r^4_m,
\end{equation}
which corresponds to an improved density-flux relation of the form
\begin{equation}
\partial_t (M(u,\bu) + B^4_m(u,\bu,u,\bu))
= \partial_x ( P(u,\bu)+ R^4_m(u,\bu,u,\bu))
+ R^6_m(u,\bu,u,\bu,u,\bu).
\end{equation}

In our next result \cite{IT-qnls}, addressing quasilinear Schr\"odinger flows, again with $a(\xi) = \xi^2$, this assumption was relaxed to guarantee only 
that $c^4_m$ vanishes of second order on the doubly resonant set $\calR^2$. This in turn provided an enhanced division property  
\begin{equation}\label{full-div}
c^4_m +  i  \Delta^4 a(\xi) \,  b^4_m = i \Delta^4 \xi\, r^4_m + f^{4,tr}_m,
\end{equation}
where the additional term $f^{4,tr}_m$ contains ``transversal" interactions, in the sense that near 
$\calR$ it can be smoothly represented in the form
\begin{equation}
f^{4,tr}_m := q^{4,tr}_{m,1}(\xi_1-\xi_2)(\xi_3-\xi_4)+ q^{4,tr}_{m,2}(\xi_1-\xi_4)(\xi_2-\xi_3).
\end{equation}
This corresponds to a  density-flux relation
of the form
\begin{equation}\label{df-modified}
\begin{aligned}
\partial_t (M(u,\bu) + B^4_m(u,\bu,u,\bu))
= & \  \partial_x ( P(u,\bu)+ R^4_m(u,\bu,u,\bu))
\\ & \ + C^{4,tr}_m(u,\bu,u,\bu) + R^6_m(u,\bu,u,\bu,u,\bu),
\end{aligned}
\end{equation}
where the two source terms
$C^{4,tr}_m(u,\bu,u,\bu)$ respectively $R^6_m(u,\bu,u,\bu,u,\bu)$ can be estimated using two bilinear $L^2_{t,x}$ bounds, respectively the $L^6_{t,x}$ bound.

Our goal in this work will be to establish the counterparts of \eqref{full-div}, respectively \eqref{df-modified} for the frequency localized versions of the mass, momentum and reverse momentum, but working with a general dispersion relation $a(\xi)$. 

Hence, given a dyadic frequency
$\lambda$ (here we distinguish between positive and negative frequencies)  we consider the frequency localized counterpart of the relation \eqref{df-m},
namely 
\begin{equation}\label{df-ml}
    \partial_t M_\lambda(u,\bu)  =  \partial_x P_\lambda(u,\bar u) + C^4_{\lambda,m}(u,\bu,u,\bu),
\end{equation}
where, after symmetrization, the symbol of 
$C^4_{\lambda,m}$ is given by
\[
\begin{aligned}
c^4_{\lambda,m}(\xi_1,\xi_2,\xi_3,\xi_4) = \frac{i}2 & \left(  -  \psi_\lambda(\xi_4,\xi_1-\xi_2+\xi_3) c(\xi_1,\xi_2,\xi_3) - \psi_\lambda(\xi_2,\xi_1-\xi_4+\xi_3)c(\xi_1,\xi_4,\xi_3)\right. 
\\ & \ 
\left.+ \psi_\lambda(\xi_1,\xi_2-\xi_3+\xi_4)\bar c(\xi_2,\xi_3,\xi_4)+ 
\psi_\lambda(\xi_3,\xi_2-\xi_1+\xi_4) \bar c(\xi_2,\xi_1,\xi_4) \right).
\end{aligned}
\]
Similarly, for the localized momentum we get  
\begin{equation}\label{df-pl}
    \partial_t P_\lambda(u,\bu)  =  \partial_x E_\lambda(u,\bar u) + C^4_{\lambda,p}(u,\bu,u,\bu),
\end{equation}
where, after symmetrization, the symbol of 
$C^4_{\lambda,p}$ is given by
\begin{equation}\label{c4pl}
\begin{aligned}
c^4_{\lambda,p}(\xi_1,\xi_2,\xi_3,\xi_4) = \frac{i}2 & \left(  -  (\psi_\lambda p)(\xi_4,\xi_1-\xi_2+\xi_3) c(\xi_1,\xi_2,\xi_3) - (\psi_\lambda p)(\xi_2,\xi_1-\xi_4+\xi_3)c(\xi_1,\xi_4,\xi_3)\right. 
\\ & \ 
\left.+ (\psi_\lambda p)(\xi_1,\xi_2-\xi_3+\xi_4)\bar c(\xi_2,\xi_3,\xi_4)+ 
(\psi_\lambda p)(\xi_3,\xi_2-\xi_1+\xi_4) \bar c(\xi_2,\xi_1,\xi_4) \right).
\end{aligned}
\end{equation}
In the same way, for the localized reverse momentum we get  
\begin{equation}\label{df-rpl}
    \partial_t \rP_\lambda(u,\bu)  =  \partial_x M_\lambda(u,\bar u) + C^4_{\lambda,\rp}(u,\bu,u,\bu),
\end{equation}
where, after symmetrization, the symbol of 
$C^4_{\lambda,\rp}$ is given by
\begin{equation}\label{c4rpl}
\begin{aligned}
c^4_{\lambda,\rp}(\xi_1,\xi_2,\xi_3,\xi_4) = \frac{i}2 & \left(  -  (\psi_\lambda \rp)(\xi_4,\xi_1-\xi_2+\xi_3) c(\xi_1,\xi_2,\xi_3) - (\psi_\lambda \rp)(\xi_2,\xi_1-\xi_4+\xi_3)c(\xi_1,\xi_4,\xi_3)\right. 
\\ & \ 
\left.+ (\psi_\lambda \rp)(\xi_1,\xi_2-\xi_3+\xi_4)\bar c(\xi_2,\xi_3,\xi_4)+ 
(\psi_\lambda \rp)(\xi_3,\xi_2-\xi_1+\xi_4) \bar c(\xi_2,\xi_1,\xi_4) \right).
\end{aligned}
\end{equation}

The symbols  $C^4_{\lambda, m}$,  $C^4_{\lambda, p}$,  $C^4_{\lambda, \rp}$  are supported in the region where
\begin{itemize}
    \item at least one of the frequencies has size $\lambda$, and 
    \item the frequencies satisfy 
    $\Delta^4 \xi \ll \lambda$.
\end{itemize}
These two properties allow us, using a smooth partition of unit, to divide these symbols into two parts, 
\begin{equation}
  c^4_{\lambda,m} =  c_{\lambda,m}^{4,bal}
  + c_{\lambda,m}^{4,unbal},
\end{equation}
and similarly for the momentum and the reverse
momentum,
so that the balanced component has support
in the region where all frequencies have size $\lambda$, while the unbalanced component is supported where at least 
two of the four frequencies are away from $\lambda$. The remaining case is precluded 
by the support property of $c^4_{\lambda,m}$.

The unbalanced component will simply go in the transversal box, so we now turn our attention to the balanced component,
which has the following three properties:

\begin{itemize}
\item Symmetry: the symbol $ c_{\lambda,m}^{4,bal}$ is symmetric with respect to permutations of the arguments $\xi_1$ and $\xi_3$, then $\xi_2$ and $\xi_4$, and finally of the pairs $(\xi_1,\xi_3)$ and $(\xi_2,\xi_4)$.
    
    \item Size and regularity: by (H1) we have 
\begin{equation}\label{c4-size}
  c_{\lambda,m}^{4,bal} \in S^{3\delta}_\lambda
\end{equation}    
   \item Vanishing condition: by (H2) we
   have 
\begin{equation}\label{c4-cancel}
   c_{\lambda,m}^{4,bal} =0, \qquad  \nabla  c_{\lambda,m}^{4,bal} = 0 \qquad \text{ on } \calR_2.
\end{equation}   
\end{itemize}
We use these properties to solve the appropriate division problem for $c_{\lambda,m}^{4,bal}$, in the spirit of the  relation \eqref{full-div}. Just as in the linear case, we will consider separately the cases $\gamma > -1$ and $\gamma < -1$. In the first case we have

\begin{proposition}\label{p:division-a}
Assume that $\gamma > -1$.  Let $c_{\lambda,m}^{4,bal}$ be a symmetric symbol which satisfies \eqref{c4-size} and \eqref{c4-cancel}. Then there exists a decomposition
of the form
\begin{equation}
  c_{\lambda,m}^{4,bal} =  b_{\lambda,m}^{4} \Delta^4 a(\xi) +  r_{\lambda,m}^{4}
  \Delta^4 \xi +  i f^{4,bal}_{\lambda,m},
\end{equation}
with  $f^{4,bal}_{\lambda,m}$, 
vanishing to second order on the doubly resonant set, of the form
\begin{equation}\label{fbal-m}
   f^{4,bal}_{\lambda,m} :=  q^{4,bal}_{\lambda,m}[ (\xi_1-\xi_2)(\xi_3-\xi_4)+ (\xi_1-\xi_4)(\xi_2-\xi_3)],
\end{equation}
and with symbol regularity 
\begin{equation}\label{brq4-size-nls}
 b_{\lambda,m}^{4} \in S_\lambda^{3\delta-|\alpha| - \gamma-2},   
 \qquad 
 r_{\lambda,m}^{4} \in S_\lambda^{3\delta-|\alpha|-1},
\qquad
 q_{\lambda,m}^{4,bal} \in S_\lambda^{3\delta-|\alpha|-2}.
\end{equation}
\end{proposition}

For the counterpart of this result in the second case we need to take into account the asymptotic velocities $v^{\pm}$, depending on whether we work in a positive
or a negative dyadic frequency region.

\begin{proposition}\label{p:division-b}
Assume that $\gamma < -1$.  Let $c_{\lambda,m}^{4,bal}$ be a symmetric symbol which satisfies \eqref{c4-size} and \eqref{c4-cancel}. Then there exists a decomposition
of the form
\begin{equation}\label{division-b}
  c_{\lambda,m}^{4,bal} =  b_{\lambda,m}^{4} (\Delta^4 a(\xi)+ v^\pm \Delta^4 \xi) +  r_{\lambda,m}^{4}
  \Delta^4 \xi +  i f^{4,bal}_{\lambda,m},
\end{equation}
with  $f^{4,bal}_{\lambda,m}$, 
vanishing to second order on the doubly resonant set, of the form
\begin{equation}\label{fbal-m-b}
   f^{4,bal}_{\lambda,m} :=  q^{4,bal}_{\lambda,m}[ (\xi_1-\xi_2)(\xi_3-\xi_4)+ (\xi_1-\xi_4)(\xi_2-\xi_3)],
\end{equation}
and with symbol regularity

\begin{equation}\label{brq4-size-kg}
 b_{\lambda,m}^{4} \in S_\lambda^{3\delta-|\alpha| - \gamma-2},   \qquad 
 r_{\lambda,m}^{4} + v^{\pm}  \in S_\lambda^{3\delta-|\alpha|-1},
\qquad
 q_{\lambda,m}^{4,bal} \in S_\lambda^{3\delta-|\alpha|-2}.
\end{equation}
\end{proposition}

\begin{proof}[Proof of Propositions~\ref{p:division-a}, \ref{p:division-b}]
The proof of these two results are identical, once we make a slight adjustment in the second one. Precisely, 
if $\gamma < -1$,
depending on whether we work in a positive 
or negative dyadic frequency region,
we substitute in a Galilean fashion
\[
a(\xi) \to a(\xi) + v^{\pm} \xi,
\]
removing the possible nonzero linear asymptote for $a$ at $\pm \infty$.
This is consistent with \eqref{division-b},
and serves to insure that in our dyadic region we have
\begin{equation}\label{a-prim}
a' \in S^{\gamma+1}_\lambda.    
\end{equation}
We remark that this property is automatic in the case $\gamma > -1$.

After this adjustment, the proof is done in two steps, broadly starting from ideas in \cite{IT-global} and \cite{IT-qnls}. 
\medskip

\emph{STEP 1.} Here we reduce the problem 
to the case when $c_{\lambda,m}^{4,bal} = 0$ on $\calR$, by constructing a suitable correction $f^{4,bal}_{\lambda,m}$. 

We begin by we examining the behavior of $c_{\lambda,m}^{4,bal}$ on the resonant set $\calR$.  This can be seen as the union of two transversal planes intersecting on a line in $\R^4$, which we parametrize with two frequencies $\xi$ and $\eta$,
so that 
\[
\{ \xi_1,\xi_3\} = \{\xi_2,\xi_4\} = \{ \xi,\eta\}.
\]
Restricted to this set, we can view $c^4_{\lambda,m}$ as a smooth
symmetric function of $\xi$ and $\eta$,
\[
c_{\lambda,m}^{4,bal}(\xi_1,\xi_2,\xi_3,\xi_4) = h(\xi,\eta) \qquad \text{on } 
\calR.
\]
Here $c_{\lambda,m}^{4,bal}$ vanishes of second order on $\calR_2$, so $h$ 
vanishes of second order when $\xi=\eta$. Then we can smoothly divide, 
\[
h(\xi,\eta) = h_1(\xi,\eta)(\xi-\eta)^2
\]
with $h_1$ still smooth and symmetric,
with symbol regularity $h_1 \in S^{3\delta -2}_\lambda$.

On $\calR$ we have  
\[
(\xi-\eta)^2 = -[ (\xi_1-\xi_2)(\xi_3-\xi_4)+ (\xi_1-\xi_4)(\xi_2-\xi_3)],
\]
where the right hand side is now extended to all of $\R^4$. It remains to find an extension of $h_1$ from $\calR$ to $\R^4$.
With a linear change of variable we have
\[
h_1(\xi,\eta) = h_2(\xi+\eta,\xi-\eta),
\]
where $h_2$ is smooth and even in the second variable, with similar regularity $h_1 \in S^{3\delta -2}_\lambda$. Then we can further smoothly write
\[
h_1(\xi,\eta) = h_3(\xi+\eta,(\xi-\eta)^2).
\]
But this last expression clearly admits a smooth extension with all the required symmetries and the same regularity, e.g.
\[
\begin{aligned}
    q^{4,bal}_{\lambda,m}:= & \ h_3\left(\frac12(\xi_1+\xi_2+\xi_3+\xi_4),\frac12((\xi_1-\xi_3)^2+(\xi_2-\xi_4)^2)\right)
    \\ 
    = & \ h_2(\xi+\eta,[\frac12((\xi_1-\xi_3)^2+(\xi_2-\xi_4)^2)]^\frac12).
\end{aligned}
\]
Now we subtract the contribution of $q^{4,bal}_{\lambda,m}$ from $c^{4,bal}_{\lambda,m}$, which reduces the problem to the case when $c^4_{\lambda,m}$ is zero on $\calR$, as desired.

We remark that in this first step we have only used the dispersion relation in a minimal way, via the resonant set $\calR$, without needing any other explicit bounds.

\medskip

\emph{STEP 2.} Here we make the additional assumption $c^{4,bal}_{\lambda,m} = 0$ on $\calR$, and show that we have the simplified representation
\begin{equation}\label{want-c4}
  c_{\lambda,m}^{4,bal} =  b_{\lambda,m}^{4} \Delta^4 a(\xi) +  r_{\lambda,m}^{4}
  \Delta^4 \xi. 
\end{equation}
To start with, we introduce coordinates which
better diagonalize the above decomposition, namely 
\[
4\eta_1 = \xi_1 + \xi_2-\xi_3 - \xi_4,
\quad 4\eta_2 = \xi_1 - \xi_2-\xi_3 + \xi_4,
\quad 4\eta_3 = \xi_1 - \xi_2+\xi_3 - \xi_4,
\quad 4\eta_4 = \xi_1 + \xi_2+\xi_3 + \xi_4.
\]
Expressed as a function of these variables,
the symbol $c^{4,bal}_{\lambda,m}$ is 
 symmetric in $(\eta_1,\eta_2)$,  even in
 any of the pairs of two frequencies $(\eta_1,\eta_2)$, $(\eta_1,\eta_3)$ and  $(\eta_1,\eta_3)$, and vanishes on the set 
\[
\tilde \calR =\{ \eta_1 \eta_2 = 0, \ \eta_3= 0\}.
\]
Writing in these coordinates
\[
c_{\lambda,m}^{4,bal}(\eta_1,\eta_2,\eta_3,\eta_4) = c_{\lambda,m}^{4,bal}(\eta_1,\eta_2,0,\eta_4)+ (c_{\lambda,m}^{4,bal}(\eta_1,\eta_2,\eta_3,\eta_4) - c_{\lambda,m}^{4,bal}(\eta_1,\eta_2,0,\eta_4))
\]
the difference on the right vanishes when $\eta_3=0$, so we can smoothly divide it by $\eta_3$. On the other hand the first term on the right is symmetric and even in $(\eta_1,\eta_2)$, vanishing on $\eta_1= 0$ and on $\eta_2= 0$; hence, it can be successively smoothly divided by $\eta_1$ and then $\eta_2$. Since  $4\eta_3 =\Delta^4 \xi$, we arrive at a representation
\begin{equation} \label{first-rep-c4}
    c_{\lambda,m}^{4,bal}
    = \tilde  b_{\lambda,m}^{4} \eta_1 \eta_2 +  \tilde r_{\lambda,m}^{4} \Delta^4 \xi, 
\end{equation}
where  the symbol regularities are  
\begin{equation}
 \tilde  b_{\lambda,m}^{4} \in S_\lambda^{3\delta -2},
 \qquad  \tilde  r_{\lambda,m}^{4} \in S_\lambda^{3\delta -1}.
\end{equation}
It now remains to connect the symbol $\eta_1\eta_2$ to $\Delta^4 a(\xi)$, and show that we have a representation
\begin{equation}\label{second-rep-c4}
\eta_1 \eta_2 = z^4_a(\xi)  \Delta^4 a(\xi) + z^4(\xi) \Delta^4 \xi
\end{equation}
with symbol regularities
\begin{equation}\label{second-rep-c4-reg-nls}
 z^4_a \in S^{-\gamma}_\lambda, \quad z^4 \in S^{1}_{\lambda}. 
 \end{equation}
Combining this with \eqref{first-rep-c4}
would give the desired representation \eqref{want-c4} with 
\[
b^4_{\lambda,m} = z_a^4 \tilde b^4_{\lambda,m}, \qquad 4 r^4_{\lambda,m}
= \tilde r^4_{\lambda,m}+ z^4 \tilde b^4_{\lambda,m}.
\]

It remains to establish \eqref{second-rep-c4} with symbol regularities as in \eqref{second-rep-c4-reg-nls}.
It is convenient to start from $\Delta^4 a$, for which we write
\begin{equation}\label{D4a-rep}
\begin{aligned}
2\Delta^4 a(\xi) = & - 2p(\xi_1,\xi_2) (\xi_1-\xi_2) + 2p(\xi_3,\xi_4) (\xi_3-\xi_4)
\\
= & \ -\eta_2 (p(\xi_1,\xi_2)-p(\xi_3,\xi_4)) + \Delta^4 \xi (p(\xi_1,\xi_2)+ p(\xi_3,\xi_4))
\end{aligned}
\end{equation}
Here the second term on the right is a multiple of $\eta_3 = \Delta^4 \xi$, 
whose coefficient, by \eqref{a-prim},
has regularity
\[
p(\xi_1,\xi_2)+ p(\xi_3,\xi_4) \in S^{\gamma+1}_\lambda.
\]
This is consistent with \eqref{second-rep-c4-reg-nls}, so we can dispense with this term.

It remains to consider the first term on the right in \eqref{D4a-rep}. Using the notation "$\approx$" to denote equality modulo $S^{\gamma+1}_\lambda \Delta^4 \xi$, we are left with 
\[
\begin{aligned}
2\Delta^4 a(\xi) \approx 
 & \ -\eta_2 (p(\xi_1,\xi_2)-p(\xi_3,\xi_4)) 
\\
= & \ -\eta_2 (p(\eta_1\!+\!\eta_2\!+\!\eta_3\!+\!\eta_4,\eta_1\!-\!\eta_2\! -\! \eta_3\!+\!\eta_4)-p(-\eta_1\!-\!\eta_2\! +\!\eta_3\!+\! \eta_4 , -\eta_1\!+\!\eta_2 \!-\!\eta_3\! +\!\eta_4))
\\
\approx & \ -\eta_2 (p(\eta_1+\eta_2+\eta_4,\eta_1-\eta_2+\eta_4)-p(-\eta_1-\eta_2 + \eta_4 , -\eta_1+\eta_2 +\eta_4))
\\
:= &  \ \eta_2 \Delta p.
\end{aligned}
\]
It remains to examine the last difference above, denoted by $\Delta p$. We can use the fundamental theorem of calculus with respect to $\eta_1$ to pull out an $\eta_1$ factor by writing
\[
\Delta p = -\eta_1 \int_{-1}^1 [(\partial_1+\partial_2) p](h\eta_1 +\eta_2+\eta_4, h \eta_1 - \eta_2+\eta_4)\,  dh .
\]
Here  for $\xi_1,\xi_2$ localized at frequency $\lambda$, we compute directly
\[
-(\partial_1+\partial_2) p(\xi_1,\xi_2) 
= \frac{a'(\xi_1)-a'(\xi_2)}{\xi_1-\xi_2} \approx \lambda^{\gamma},
\]
where at the last step we have used the convexity of $a$ at frequency $\lambda$, given by \eqref{a-convex}.
We conclude that $\Delta p$ is an elliptic symbol 
\[
\Delta p \approx \lambda^{\gamma},\qquad  \Delta p \in S^{\gamma}_\lambda.
\]
Hence we have proved that
\[
2\Delta^4 a(\xi) \approx \eta_1 \eta_2 \Delta p,
\]
which, after division by $\Delta p$,  yields the representation \eqref{second-rep-c4}.
Combined with \eqref{first-rep-c4}, this concludes the proof of the lemma.

\end{proof}

One may apply the above division result
not only to the localized mass but also to the localized momentum and reverse momentum. This immediately leads us 
to our main density-flux identities. For clarity we state these separately in the 
GNLS, respectively the GKG cases.

\bigskip

\emph{ A. The GNLS case, $\gamma > -1$.}
In this case we introduce the modified localized mass given by 
\begin{equation}\label{Ml-sharp}
\ms_{\lambda}:= M_{\lambda}(u) + B^4_{\lambda,m}(u,\bu,u,\bu).    
\end{equation}
For this we can write the improved density-flux relation
\begin{equation}\label{df-msharp}
    \partial_t \ms_\lambda  =  \partial_x (P_\lambda +  R^4_{\lambda,m})+ F^4_{\lambda,m} + R^{6}_{\lambda,m},
\end{equation}
where the quartic source term $F^4_{\lambda,m}$ only involves transversal interactions and has a balanced and an unbalanced component,
\begin{equation}\label{F4-dec}
    F^4_{\lambda,m} = F^{4,bal}_{\lambda,m}
+ C^{4,unbal}_{\lambda,m},
\end{equation}
while the six-linear source term $R^6_{\lambda,m}$ arises 
from the time derivative of $B^4_{\lambda,m}$ via the cubic nonlinearity in the equation \eqref{eq:main},
\begin{equation}\label{R6-def}
   R^6_{\lambda,m}(u) = \Lambda^6  \left(\frac{d}{dt} B^4_{\lambda,m}(u)\right) .
\end{equation}
The symbol regularities in these identities are given by \eqref{brq4-size-nls} in Proposition~\ref{p:division-a}.

In the same manner we obtain a similar set of equations for the momentum and the reverse momentum,
\begin{equation}\label{Pl-sharp}
\ps_{\lambda}:= P_{\lambda}(u) + B^4_{\lambda,p}(u,\bu,u,\bu),   
\end{equation}
\begin{equation}\label{rPl-sharp}
\rPs_{\lambda}:= \rP_{\lambda}(u) + B^4_{\lambda,\rp}(u,\bu,u,\bu),   
\end{equation}
and the associated  improved density-flux relations
\begin{equation}\label{df-psharp}
    \partial_t \ps_\lambda  =  \partial_x (E_\lambda +  R^4_{\lambda,p})+ F^4_{\lambda,p} + R^{6}_{\lambda,p},
\end{equation}
\begin{equation}\label{df-rpsharp}
    \partial_t \rPs_\lambda  =  \partial_x (M_\lambda +  R^4_{\lambda,\rp})+ F^4_{\lambda,\rp} + R^{6}_{\lambda,\rp},
\end{equation}
with the added caution that the reverse momentum cannot be used at low frequency $\lambda \lesssim 1$.

Compared with \eqref{brq4-size-nls}, all symbols have similar size and regularity but with an extra $\lambda^{\gamma+1}$ factor for the momentum related quantities, respectively
an extra $\lambda^{-\gamma-1}$ factor for the reverse momentum related quantities.

\bigskip

\emph{ B. The GKG case, $\gamma < -1$.}
In this case the modified localized mass is still given by \eqref{Ml-sharp}.
However, depending on whether we 
are at positive or negative frequency, the improved density-flux relation now reads
\begin{equation}\label{df-msharp-b}
    (\partial_t + v^{\pm}\partial_x) \ms_\lambda  =  \partial_x (P_\lambda^\pm +  R^4_{\lambda,m})+ F^4_{\lambda,m} + R^{6}_{\lambda,m},
\end{equation}
where $F^4_{\lambda,m}$ and $R^6_{\lambda,m}$ remain as in \eqref{F4-dec}, respectively \eqref{R6-def}, and the
symbol regularities in these identities are given by \eqref{brq4-size-kg} in Proposition~\ref{p:division-b}.

Similarly we have the corrected  relative  localized momentum and  reverse momentum,
\begin{equation}\label{Pl-sharp-rel}
P^{\sharp,\pm}_{\lambda}:= P^\pm_{\lambda}(u) + B^4_{\lambda,p}(u,\bu,u,\bu),   
\end{equation}
\begin{equation}\label{rPl-sharp-rel}
\rP^{\sharp,\pm}_{\lambda}:= \rP^\pm_{\lambda}(u) + B^4_{\lambda,\rp}(u,\bu,u,\bu),   
\end{equation}
and the associated  improved density-flux relations
\begin{equation}\label{df-psharp-rel}
 (\partial_t + v^{\pm}\partial_x)  P^{\sharp,\pm}_{\lambda}  =  \partial_x (E^\pm_\lambda +  R^{4,\pm}_{\lambda,p})+ F^4_{\lambda,p} + R^{6}_{\lambda,p},
\end{equation}
\begin{equation}\label{df-rpsharp-rel}
  (\partial_t + v^{\pm}\partial_x)  \rP^{\sharp,\pm}_{\lambda}  =  \partial_x (M_\lambda +  R^{4,\pm}_{\lambda,\rp})+ F^4_{\lambda,\rp} + R^{6}_{\lambda,\rp},
\end{equation}
where we have omitted the $\pm$ superscript
for the quantities which do not depend on our choice of $v^+$ vs. $v^-$. This is 
particularly relevant for low frequencies
$\lambda \lesssim 1$ where both $\pm$ choices are useful and indeed directly equivalent, via a Galilean type transformation. 

Compared with \eqref{brq4-size-kg}, all symbols associated to the  momentum and reverse momentum  have similar size and regularity but with an extra $\lambda^{\gamma+1}$ factor for the momentum related quantities, respectively
an extra $\lambda^{-\gamma-1}$ factor for the reverse momentum related quantities.


\section{Interaction Morawetz identities}
\label{s:Morawetz}

The aim of the interaction Morawetz identities 
is to capture the interaction strength of 
two solutions to either linear or nonlinear 
flows. In our setup these will be always applied to two frequency localized components of our solution $u$ for \eqref{eq:main}, corresponding to two dyadic frequencies
$\lambda$ and $\mu$. It will be natural to divide the analysis into three cases:

\begin{enumerate}
    \item the balanced case, $\lambda= \mu$.
This is the most difficult case, which aims to capture the interaction potential of waves which travel with nearby velocities.

 \item the semi-balanced case, where $\lambda \neq \mu$ but still $\lambda \approx \mu$.
This is a simpler case, where on one hand we are pairing transversal waves, and on the other hand we still have only one primary frequency parameter.

 \item the unbalanced case, where  $ \mu \ll \lambda$. This should also be a simpler case, where we are again pairing transversal waves.
 But a direct application of the ideas in the first two cases runs into some technical difficulties, due to the presence of two primary frequency scales. Hence, an additional insight will be needed. 
\end{enumerate}

For clarity we will discuss first the interaction Morawetz identities for the linear evolution in each of these three cases, omitting the bilinear frequency localization multiplier $\psi_\lambda$.
Then we consider their nonlinear counterparts,
also with frequency localization added.
We will first present the setup in the GNLS case $\gamma > -1$, and then point out the differences in the GKG case $\gamma < -1$.

\subsection{The linear equation, balanced case}

With the linear mass and momentum densities defined earlier, we define the  interaction Morawetz functional of two solutions $u$ and $v$ by
\begin{equation}\label{I-def}
\bI(u,v) := \int_{x > y} M(u)(x) P(v)(y) - P(u)(x) M(v)(y)  \,dx dy,
\end{equation}
which is akin to the Schr\"odinger case studied in \cite{IT-global} (see also the earlier work \cite{PV}), but with the 
new momentum density defined in the previous section by \eqref{p}. Then we compute $d\bI/dt$ using the above conservation laws. We have
\[
\begin{aligned}
\frac{d}{dt} \bI(u,v) =  & \ \int_{x > y} \partial_x P(u)(x)
\, P(v)(y) +  M(u)(x) \, \partial_y  E(v)(y)
\\ & \ \ \ \ 
-\partial_x E(u)(x)\, 
M(v)(y) - P(u)(x)\, \partial_y  P(v)(y)\,
dx dy 
\\
= &  \int M(u) E(v) + M(v) E(u) 
- 2 P(u)P(v) \,dx:= 
\int J^4(u,\bu,v,\bv)\, dx.
\end{aligned}
\]
Here  $J^4$ has symbol
\begin{equation}\label{j4-def}
j^4(\xi_1,\xi_2,\xi_3,\xi_4) =( p(\xi_1,\xi_2) - p(\xi_3,\xi_4))^2,
\end{equation}
where we recall that 
a-priori the symbol  $j^4$ is only determined uniquely on the diagonal $\Delta^4 \xi = 0$.

This is a nonnegative symbol, which when restricted to $\Delta^4 \xi = 0$ vanishes only on 
a subset of $\calR$, precisely on 
$\{ \xi_1 = \xi_4,\ \xi_2 =\xi_3\}$. However, this is not enough to guarantee  integral, even less pointwise nonnegativity for a quadrilinear form. In the Schr\"odinger case we were able to factor
\begin{equation}\label{j4-sch}
j^4(\xi_1,\xi_2,\xi_3,\xi_4) =    4(\xi_1-\xi_4)(\xi_2-\xi_3) \qquad \text{on } \Delta^4 \xi = 0,
\end{equation}
 because of the following computation on the diagonal $\Delta^4 \xi = 0$:
\[
(\xi_1+\xi_2)^2 + (\xi_3+\xi_4)^2 - 2 (\xi_1+\xi_2)(\xi_3+\xi_4) = (\xi_1+\xi_2-\xi_3 - \xi_4)^2 
=  4 (\xi_1 - \xi_4)(\xi_2 - \xi_3).
\]
Then taking \eqref{j4-sch} as the definition of the extension of $J^4$ away from $\Delta^4 \xi = 0$ gave  the pointwise nonnegativity of the associated quartic form $J_4$
\[
J^4(u,\bu,v,\bv) = 4 |\partial_x (u \bv)|^2.
\]

\medskip

In our context here we no longer have such a clean factorization,
but we still observe that the difference $p(\xi_1,\xi_2) - p(\xi_3,\xi_4)$ restricted to 
$\Delta^4 \xi = 0$ vanishes if $\xi_1 = \xi_4$ 
or equivalently if $\xi_2 = \xi_3$. This would indicate that we might have a factorization of the form 
\[
j^4(\xi_1,\xi_2,\xi_3,\xi_4) =  q(\xi_1,\xi_2,\xi_3,\xi_4) (\xi_1-\xi_4)(\xi_2-\xi_3)
\qquad \text{on } \Delta^4 \xi=0.
\]
However, such a representation is not enough if we want to have a clean square
$L^2$ norm on the right. Instead we would like to have a relation of the form
\begin{equation} \label{j4-factor}
j^4(\xi_1,\xi_2,\xi_3,\xi_4) =  q(\xi_1,\xi_4) q(\xi_2,\xi_3) (\xi_1-\xi_4)(\xi_2-\xi_3) \qquad \text{on } \Delta^4 \xi = 0,
\end{equation}
which would yield  the positivity
\[
J^4(u,\bu,v,\bv) = 4 |\partial_x Q(u \bv)|^2.
\]
The problem is that  this is too much to ask on the full diagonal $\Delta^4 \xi = 0$.

Our solution to this problem, which is a 
key new idea in this paper, is to 
restrict ourselves first to 
the resonant set $\calR$. Then it would be enough 
to have the above relation \eqref{j4-factor}
when $\xi_1=\xi_2$ and $\xi_3=\xi_4$. Since we have $p(\xi,\xi) = -a'(\xi)$  
we can write
\[
p(\xi_1,\xi_2) - p(\xi_3,\xi_4)= - (\xi_1-\xi_4)\, 
q(\xi_1,\xi_4) \qquad \text{on } \calR,
\]
with a smooth symmetric quotient
\begin{equation}\label{def-q}
q(\xi_1,\xi_4) := \frac{a'(\xi_1)-a'(\xi_4)}{\xi_1-\xi_4}.
\end{equation}
We further note this is a positive symbol by the 
convexity of $a$. In particular, if both $\xi_1$ and $\xi_4$ are comparable to $\lambda$ then by \eqref{a-convex} we have 
\begin{equation}\label{q-positive}
   q(\xi_1,\xi_4) \approx \lambda^\gamma,
\end{equation}
which we heuristically interpret as an ellipticity condition for $q$. We caution the  reader that this is not a classical ellipticity condition, as $Q$ is a bilinear form rather than a linear operator.

This leads us to the following
relation
\begin{equation}
  j^4(\xi_1,\xi_2,\xi_3,\xi_4) =  q(\xi_1,\xi_4) q(\xi_2,\xi_3) (\xi_1-\xi_4)(\xi_2-\xi_3)  
\qquad \text{on } \calR .
\end{equation}
This is however only half of the solution 
to the problem described above, as we also need to make the transition between the 
resonant set $\calR$ and the full diagonal $\Delta^4 \xi = 0$. We achieve this by 
introducing a correction to the interaction 
Morawetz functional, which is another key 
idea of this article. At the algebraic level,
this idea relies on solving another division problem: 

\begin{lemma}
\label{l:division-a}
Assume that $\gamma > -1$. Then for the symbol $j^4$ in \eqref{j4-def} we have the following division property:
\begin{equation}
\begin{aligned}
 j^4(\xi_1,\xi_2,\xi_3,\xi_4) = & \ q(\xi_1,\xi_4) q(\xi_2,\xi_3) (\xi_1-\xi_4)(\xi_2-\xi_3)
 \\ &\ + b^4_I(\xi_1,\xi_2,\xi_3,\xi_4) \Delta^4 a(\xi)
 + r^4_I(\xi_1,\xi_2,\xi_3,\xi_4) \Delta^4 \xi
\end{aligned} 
\end{equation} 
with smooth symbols $b^4_I$ and $r^4_I$. Furthermore, if all frequencies are comparable to $\lambda$ then \eqref{q-positive} holds and we have
the regularity
\begin{equation}
q \in S^{\gamma}_\lambda, \qquad b^4_I \in S^{-2}_\lambda, \qquad 
r^4_I \in S^{2\gamma+1}_\lambda.
\end{equation}

\end{lemma}
We note that when this lemma is used, the $R^4_{I}$ term appears in the flux, and its contribution integrates out to zero, while the $B^4_{I}$ term corresponds to a normal form type energy correction for the interaction Morawetz functional.

\begin{proof}
By construction, the difference 
\[
D :=  j^4(\xi_1,\xi_2,\xi_3,\xi_4) -  q(\xi_1,\xi_4) q(\xi_2,\xi_3) (\xi_1-\xi_4)(\xi_2-\xi_3)
\]
is smooth and vanishes on the resonant set $\calR$. Further, if all frequencies are comparable to $\lambda$ then $D \in S^{2\gamma+2}_\lambda$. Then the division problem is resolved exactly as in the proof of Proposition~\ref{p:division-a},
\emph{Step 2}, with obvious adjustments to the orders of symbols. 
\end{proof}

Based on the  above Lemma~\ref{l:division-a}  we introduce a modified interaction Morawetz 
functional, 
\begin{equation}\label{Isharp}
\is(u,v) := \bI (u,v) +  \bB_I^4(u,\bu,v,\bv).
\end{equation}
With this notation, our interaction Morawetz identity  for linear homogeneous waves reads as follows
\begin{equation}
\frac{d}{dt} \is(u,v) =   \|\partial_x Q(u,\bv)\|_{L^2}^2  ,
\end{equation}
where the right hand side has the desired positivity.

We will primarily use this type of interaction Morawetz identity in the 
diagonal case, where $u$ and $v$ are 
localized at the same dyadic frequency
$\lambda$. For convenience, we summarize
the outcome as follows:

\begin{proposition}\label{p:divisionI-a}
Assume that $\gamma > -1$.  Then given two frequency localized solutions
$u_\lambda$ and $v_\lambda$ for the homogeneous linear equation, we have the interaction Morawetz
identity 
\begin{equation}
\frac{d}{dt} \is(u_\lambda,v_\lambda) =   \|\partial_x Q(u_\lambda,\bv_\lambda)\|_{L^2}^2  ,
\end{equation}
where 
\begin{equation}\label{Isharp-bal}
\is(u_\lambda,v_\lambda) := \bI(u_\lambda,v_\lambda) + \bB_{I,\lambda}^4(u_\lambda,\bu_\lambda,v_\lambda,\bv_\lambda) , \qquad b^4_{I,\lambda} \in S^{2\gamma}_\lambda .
\end{equation}
\end{proposition}

\subsubsection{The GKG case $\gamma < 1$}

With minor differences, the above ideas
carry over to this case. The primary difference is that we now need to work with the relative mass, momentum and energy
$M$, $P^\pm$ and $E^{\pm}$, with the 
choice for signs  determined by the frequency sign in the dyadic frequency localization. 

The interaction Morawetz functional in this case is changed to
\begin{equation}\label{I-def-b}
\bI(u,v) := \int_{x > y} M(u)(x) P^\pm(v)(y) - P^\pm(u)(x) M(v)(y)  \,dx dy,
\end{equation}
with the (matched) signs chosen as discussed above. Nevertheless, the symbol $j^4$ remains the same, as in \eqref{j4-def}. 
The symbol $q$ is also the same, given by \eqref{def-q}, and still satisfies
\eqref{q-positive}. The division lemma needs a slight adjustment, which is the same as the one between Proposition~\ref{p:division-a} and Proposition~\ref{p:division-b}:

\begin{lemma}
\label{l:division-b}
Assume that $\gamma < -1$. Then for the symbol $j^4$ in \eqref{j4-def} we have the following division property:
\begin{equation}
\begin{aligned}
 j^4(\xi_1,\xi_2,\xi_3,\xi_4) = & \ q(\xi_1,\xi_4) q(\xi_2,\xi_3) (\xi_1-\xi_4)(\xi_2-\xi_3)
 \\ &\ + b^4_I(\xi_1,\xi_2,\xi_3,\xi_4) (\Delta^4 a(\xi)+v^\pm \Delta^4 \xi)
 + r^4_I(\xi_1,\xi_2,\xi_3,\xi_4) \Delta^4 \xi
\end{aligned} 
\end{equation} 
with smooth symbols $b^4_I$ and $r^4_I$. Furthermore, if all frequencies are comparable to $\lambda$ with $\pm$ sign then \eqref{q-positive} holds and we have
the regularity
\begin{equation}
q \in S^{\gamma}_\lambda, \qquad b^4_I \in S^{-2}_\lambda, \qquad 
r^4_I \in S^{2\gamma+1}_\lambda.
\end{equation}

\end{lemma}

With these changed notations, the interaction Morawetz 
identity in Proposition~\ref{p:divisionI-a}
remains unchanged:

\begin{proposition}\label{p:divisionI-b}
Assume that $\gamma < -1$.  Then given two solutions $u_\lambda$ and $v_\lambda$ 
for the homogeneous linear equation
with dyadic frequency localization at frequency $\pm \lambda$, we have the interaction Morawetz
identity 
\begin{equation}
\frac{d}{dt} \is(u_\lambda,v_\lambda) =   \|\partial_x Q(u_\lambda,\bv_\lambda)\|_{L^2}^2  ,
\end{equation}
where 
\begin{equation}\label{Isharp-bal-b}
\is(u_\lambda,v_\lambda) := \bI(u_\lambda,v_\lambda) + \bB_{I,\lambda}^4(u_\lambda,\bu_\lambda,v_\lambda,\bv_\lambda) , \qquad b^4_{I,\lambda} \in S^{2\gamma}_\lambda .
\end{equation}
\end{proposition}

\subsection{The linear equation, semi-balanced case} \label{s:semi}

Compared to the balanced case considered 
above, in this case we seek to measure 
the interaction of two waves with separated group velocities, therefore we expect the bilinear $L^2$ estimate to be more robust. 
While the analysis in the balanced case
can be extended to this setting as well,
it is instead more interesting to simplify
the arguments and use a simpler interaction functional, namely
\[
\bI(u,v) = \int_{x > y} M(u)(x) \, M(v)(y)  \,dx dy.
\]
Its time derivative is 
\[
\begin{aligned}
\frac{d}{dt} \bI(u,v) =  & \ \int_{x > y} \partial_x P(u)(x)\,
P(v)(y) +  M(u)(x)\,  \partial_y  E(v)(y)\, dxdy
\\
= &  \int M(u) P(v) - M(v) P(u) 
 \,dx:= 
\int J^4(u,\bu,v,\bv)\, dx,
\end{aligned}
\]
where the quartic form $J^4$ has symbol
\begin{equation}\label{j4-semi}
j^4(\xi_1,\xi_2,\xi_3,\xi_4) = p(\xi_1,\xi_2) - p(\xi_3,\xi_4), \qquad \Delta^4 \xi = 0. 
\end{equation}
In general this does not have a sign, but it  has a sign when the pairs $(\xi_1,\xi_2)$, 
respectively $(\xi_3,\xi_4)$ are restricted to distinct, separated  dyadic regions. We will assume that $\xi_{1},\xi_{2} < \xi_{3},\xi_{4}$  in order to have positivity for $j^4$. This corresponds to the $u$ waves having a group velocity which is larger than the group velocity of the $v$ waves.

Suppose $u$ is at frequency $\lambda$ and $v$ is at frequency $\mu$, with $\lambda \neq \mu$.
The resonant set $\calR$ restricted to these frequency ranges is $\{ \xi_1=\xi_2, \xi_3 = \xi_4\}$.
On this set we represent $j^4$ as 
\[
j^4(\xi_1,\xi_2,\xi_3,\xi_4)
= q(\xi_1,\xi_4) \, q(\xi_2,\xi_3) \qquad 
\text{ on } \calR,
\]
where
\begin{equation}\label{q-semi}
q(\xi,\eta) = (a'(\xi) - a'(\eta))^\frac12.
\end{equation}

\medskip

Now we can solve the appropriate division problem.

\begin{lemma}\label{l:division-semi} Assume that $\gamma > -1$.
Assume that $\xi_1,\xi_2$ are comparable to $\lambda$ and $\xi_3,\xi_4$  are comparable to $\mu$, with $\mu \neq \lambda$ but $\mu \approx \lambda$. Then have the following division property:
\begin{equation}\label{j4-division}
 j^4(\xi_1,\xi_2,\xi_3,\xi_4) =  q(\xi_1,\xi_4)\,  q(\xi_2,\xi_3) 
 + b^4_I(\xi_1,\xi_2,\xi_3,\xi_4)\, \Delta^4 a(\xi)
 + r^4_I(\xi_1,\xi_2,\xi_3,\xi_4)\, \Delta^4 \xi,
\end{equation}
with $q$ as in \eqref{q-semi} and smooth symbols $b^4_I$ and $r^4_I$ with   size and regularity
\begin{equation}\label{j4-symbols}
b^4_I \in S_{\lambda}(\lambda^{-1}), \qquad 
r^4_I \in S_{\lambda}(\lambda^\gamma).
\end{equation}   
\end{lemma}
\begin{proof}
For convenience  we assume $\mu < \lambda$ but $\lambda \approx \mu$.
Then it is easily verified that both symbols 
$j^4(\xi_1,\xi_2,\xi_3,\xi_4)$
and $q(\xi_1,\xi_4)\, q(\xi_2,\xi_3) $ have regularity
$S_{\lambda}(\lambda^{\gamma+1})$.
Further, by definition their difference vanishes when $\xi_1=\xi_2$ and $\xi_3 = \xi_4$. Then it suffices to apply the same argument as in Proposition~\ref{p:division-a},
\emph{Step 2}, in order to obtain the symbols
$b^4_I$ and $r^4_I$.

\end{proof}

We summarize
the outcome in the semi-balanced case as follows:

\begin{proposition}\label{p:IM-semi}
Given two separated\footnote{ Here we also include the case when $\mu = \lambda$ but 
the two frequency localizations have opposite signs.}
dyadic frequencies $\mu \neq \lambda$ with $\mu \approx \lambda$ and frequency localized solutions
$u_\lambda$ and $v_\mu$ for the linear equation, we have the interaction Morawetz identity 
\begin{equation}
\frac{d}{dt} \is(u_\lambda,v_\mu) =   \|Q(u_\lambda,\bv_\lambda)\|_{L^2}^2  ,
\end{equation}
where $q$ is given by \eqref{q-semi} and the modified interaction functional is
\begin{equation}\label{Isharp-unbal}
\is(u_\lambda,v_\lambda) := \bI(u_\lambda,v_\lambda) + \bB_{I,\lambda}^4(u_\lambda,\bu_\lambda,v_\lambda,\bv_\lambda)  \qquad b^4_{I,\lambda} \in S^{2\gamma}_\lambda.
\end{equation}
\end{proposition}

We remark that algebraically the above computation does not require $\lambda \approx \mu$. However, if we were to pursue this venue in the unbalanced case $\mu \ll \lambda$, then we  would face 
two distinct difficulties:

\begin{enumerate}[label=(\roman*)]
    \item The symbols $q$, $b^4_I$ and $r^4_I$
operate on two distinct dyadic scales and have a more complex regularity structure.

\item  The size of $b^4_I$ is $\mu^{-1}$ where $\mu$ is the small frequency, so the corresponding estimates will be  unbalanced in an unfavourable way. This is particularly troublesome in the nonlinear case where we have to compute the nonlinear contributions to the time derivative of  the $B^4_I$ term; this not only becomes daunting computationally, but also
seems to introduce substantial restrictions 
on our parameters $\gamma,\delta,s$.
\end{enumerate}

Due to the above difficulties, in the unbalanced case it pays to modify the interaction functional in order to simplify the principal part of $J^4$ and avoid the use of the $B^4_I$ correction altogether. This will be the aim of the next subsection.

\subsubsection{The GKG case $\gamma < -1$.}

Since the relative mass and the mass 
are the same, the interaction Morawetz functional remains unchanged. Then $j^4$ 
and $q$ remain unchanged. Then the  interaction Morawetz identity in Proposition~\ref{p:IM-semi} remains unchanged, \emph{under the additional proviso
that the two frequency localization have matched signs}; this restriction serves to insure that we use the same asymptotic velocities in the density-flux relation
\eqref{df-m-lin-rel}.

The only obvious adjustment 
occurs in the division lemma, which now reads as follows:

\begin{lemma} 
\label{l:division-semi-b}
Assume that $\gamma < -1$.
Assume that $\xi_1,\xi_2$ are comparable to $\lambda$ and $\xi_3,\xi_4$  are comparable to $\mu$, with $\mu \neq \lambda$ but $\mu \approx \lambda$,  with matched $\pm$ signs. Then have the following division property:
\begin{equation}\label{j4-division-b}
 j^4(\xi_1,\xi_2,\xi_3,\xi_4) =  q(\xi_1,\xi_4)\,  q(\xi_2,\xi_3) 
 + b^4_I(\xi_1,\xi_2,\xi_3,\xi_4)\, (\Delta^4 a(\xi)+v^\pm \Delta^4 \xi)
 + r^4_I(\xi_1,\xi_2,\xi_3,\xi_4)\, \Delta^4 \xi,
\end{equation}
with $q$ as in \eqref{q-semi} and smooth symbols $b^4_I$ and $r^4_I$ with   size and regularity
\begin{equation}\label{j4-symbols-b}
b^4_I \in S_{\lambda}(\lambda^{-1}), \qquad 
r^4_I \in S_{\lambda}(\lambda^\gamma).
\end{equation}   
\end{lemma}
We note that the scenario where the frequency signs are mismatched 
is relegated to the unbalanced case.

\subsection{ The linear equation, unbalanced case}

Here we consider the case $\mu \ll \lambda$, where we will improve the choice of our 
interaction Morawetz functional in order 
to avoid the need for the normal form correction $B^4_I$.

To motivate our strategy we remark that in the $j^4$ computation in the previous case, we have two terms in \eqref{j4-semi}which are now of very different sizes\footnote{Assuming matched signs and after removing the asymptotic velocity in the GKG case.}. Then we can treat perturbatively the smaller one, which is 
\begin{enumerate}
    \item the frequency $\mu$ contribution in the GNLS case  $\gamma > -1$,
    \item the frequency $\lambda$ contribution in the GKG case $\gamma < -1$.
\end{enumerate}

We now consider the GNLS case; the GKG case
will be discussed later, separately for 
matched signs and mismatched signs in the frequency localizations.

We neglect for now the perturbative term and focus on simplifying the main one, which is achieved by using the \emph{reverse momentum} $\rP$. Precisely, in the first case above we define our interaction functional as 
\begin{equation}\label{I-unbal}
\bI(u,v) := \int_{x > y} \rP(u)(x) \, M(v)(y)  \,dx dy.    
\end{equation}
Here we recall that if $\gamma > -1$ then 
$\rP$ is not well defined at low frequency, which is why this definition 
is meaningful only if $u$ is restricted to high frequencies. That will always be the case when we use this interaction Morawetz
functional.

Computing its time derivative, this yields
\[
\begin{aligned}
\frac{d}{dt} \bI(u,v) =   \int M(u) M(v) - \rP(u) P(v) 
 \,dx:= 
\int J^4(u,\bu,v,\bv)\, dx,
\end{aligned}
\]
where the quartic form $J^4$ has symbol
\begin{equation}\label{j4-unbal}
j^4(\xi_1,\xi_2,\xi_3,\xi_4) = 1 - \rp(\xi_1,\xi_2) p(\xi_3,\xi_4), \qquad \Delta^4 \xi = 0. 
\end{equation}
The key point here is that, when  $\gamma > -1$, $\mu \ll \lambda$, $\xi_1,\xi_2 \approx \lambda$ and $\xi_3,\xi_4 \approx \mu$, the second term 
\begin{equation}
    j^{4,small}(\xi_1,\xi_2,\xi_3,\xi_4) :=  \rp(\xi_1,\xi_2) p(\xi_3,\xi_4)
\end{equation}
is smaller, with symbol size 
\begin{equation}\label{j4-small}
|j^{4,small}(\xi_1,\xi_2,\xi_3,\xi_4)| \lesssim \left(\frac{\mu}{\lambda}\right)^{\gamma+1} \ll 1,
\end{equation}
smooth on the corresponding dyadic scales.
This property will allow us to  treat this term perturbatively, rather than use an interaction Morawetz functional correction.
Precisely, our interaction Morawetz identity 
in this case reads as follows:
\begin{proposition}
    \label{p:I-unbal}
Assume that $\gamma > -1$. Consider two dyadic frequency scales  $\mu \ll \lambda$. Given two frequency localized solutions
$u_\lambda$ and $v_\mu$ for the linear equation and the associated interaction Morawetz functional \eqref{I-unbal},
we have the interaction Morawetz
identity 
\begin{equation}
\frac{d}{dt} \bI(u_\lambda,v_\mu) = \|u_\lambda\bv_\mu\|_{L^2}^2 + \bJ^{4,small}(u_\lambda,\bu_\lambda,v_\mu,\bv_\mu).
\end{equation}
\end{proposition}
We remark that the second term will later be estimated perturbatively, using the smallness in \eqref{j4-small}. But this can only be done if $\mu \ll \lambda$, which is why the argument here does not apply in the semi-balanced case.

\subsubsection{The GKG case $\gamma < -1$ with matched signs} Here we consider linear waves $u_\lambda$ and $v_\mu$ 
with unbalanced frequencies $\mu \ll \lambda$ but with matched frequency localization signs.  To insure we cover all cases, we also include mismatched signs if $\mu \lesssim 1$. 

The analysis is similar to the case $\gamma > -1$, but with two additional twists:
\begin{itemize}
    \item We use the \textbf{ same } relative mass, momentum and reverse momentum (which is where we need matched frequency signs),
    \item The $j^4$ component coming 
    from frequency $\lambda$ is the small, perturbative one.
\end{itemize}

The latter fact leads us to interchange the roles of $\lambda$ and $\mu$ in the interaction Morawetz functional, and to set 
\begin{equation}\label{I-unbal-2}
\bI(u,v) := \int_{x > y} M(u)(x) \, \rP^{\pm}(v)(y)  \,dx dy.    
\end{equation}
Then we have 
\begin{equation}\label{j4-unbal-b}
j^4(\xi_1,\xi_2,\xi_3,\xi_4) = 1 - p(\xi_1,\xi_2) \rp(\xi_3,\xi_4), \qquad \Delta^4 \xi = 0, 
\end{equation}
where we set 
\begin{equation}
 j^{4,small} =  p(\xi_1,\xi_2) \rp(\xi_3,\xi_4),  
\end{equation}
and the smallness condition \eqref{j4-small} is replaced by 
\begin{equation}\label{j4-small-b}
|j^{4,small}(\xi_1,\xi_2,\xi_3,\xi_4)| \lesssim \left(\frac{\lambda}{\mu}\right)^{\gamma+1} \ll 1,
\end{equation}
again with $j^{4,small}$ smooth on the corresponding dyadic scales. Then Proposition~\ref{p:I-unbal} still applies.

\subsubsection{The GKG case $\gamma < -1$ with mismatched signs} Here we consider linear waves $u_\lambda$ and $v_\mu$ 
with large frequencies $\mu,\lambda \gg 1$ but with mismatched frequency localization signs. Notably, we also allow $\lambda$ and $\mu$ to be comparable.

The difference in this case is that 
we need to use different asymptotic velocities in the corresponding density-flux relations, and then the leading term in the interaction Morawetz identity 
comes from the difference of the asymptotic velocities $v^\pm$. Because of this we no longer need the reverse momentum, and we can simply work with the mass, defining the interaction Morawetz functional as
\begin{equation}\label{I-unbal+}
\bI(u,v) := \int_{x > y} M(u)(x) \, M(v)(y)  \,dx dy.    
\end{equation}
To fix the signs, we assume that 
$u$ is at positive frequency and $v$ is at negative frequency. Then we compute the
time derivative of $\bI(u,v)$ using 
\eqref{df-m-lin-rel} with the $+$ sign for $u$, respectively the $-$ sign for $v$.
Integrating by parts as usual we obtain 
\begin{equation}
j^4(\xi_1,\xi_2,\xi_3,\xi_4) = (v^- -v^+)  + p^-(\xi_3,\xi_4) - p^+(\xi_1,\xi_2) 
: = v^- -v^+ + j^{4,small}(\xi_1,\xi_2,\xi_3,\xi_4),
\end{equation}
where the smallness of $j^{4,small}$
is given by
\begin{equation}\label{j4-small-c}
|j^{4,small}(\xi_1,\xi_2,\xi_3,\xi_4)| \lesssim \mu^{\gamma+1} + \lambda^{\gamma+1} \ll 1.
\end{equation}

This concludes our discussion of the linear case. Now we turn our attention to the nonlinear counterpart. The nonlinear contributions to the 
interaction Morawetz identities will play a perturbative role, with the key exception of the 
six-linear contributions in the balanced case, which is where the defocusing character of the 
problem is of the essence.

\subsection{ The nonlinear equation, balanced case}

Compared to the linear case, here we have the modified localized mass and momentum densities
\[
\ms_\lambda = M_\lambda(u,\bar u) + B^4_{m,\lambda}(u,\bar u, u,\bar u),
\]
\[
\ps_{\lambda} = P_{\lambda}(u,\bar u) + B^4_{p,\lambda}(u,\bar u, u,\bar u),
\]
which in the case $\gamma > -1$ satisfy the conservation laws
\[
\begin{aligned}
&\partial_t \ms_\lambda(u) = \partial_x(P_{\lambda}(u)
+ R^4_{m,\lambda}(u)) + F^4_{m,\lambda}(u) + R^6_{m,\lambda}(u),\\
&\partial_t \ps_{\lambda}(u) = \partial_x(E_{\lambda}(u)
+ R^4_{p,\lambda}(u)) + F^4_{p,\lambda}(u) +  R^6_{p,\lambda}(u).
\end{aligned}
\]
Similar conservation laws hold in the case
$\gamma < -1$ with $P_\lambda$, $E_\lambda$
replaced by $P_\lambda^\pm$, $E_\lambda^{\pm}$ and
$\partial_t$ replaced by $\partial_t+ v^{\pm} \partial_x$.

By analogy with the linear case, we define the interaction Morawetz functional 
\begin{equation}\label{Ia-sharp-def}
\bI_{\lambda}(u,v) :=   \iint_{x > y} \ms_\lambda(u)(x) \ps_{\lambda}(v) (y) -  
\ps_{\lambda}(u)(x) \ms_{\lambda}(v) (y) \, dx dy,
\end{equation}
with the obvious substitutions if $\gamma < -1$.

Repeating the computation in the linear case, but adding the corresponding nonlinear contributions as in \eqref{df-msharp}, \eqref{df-psharp},
the time derivative of $\bI_{\lambda}$ is
\begin{equation}\label{interaction-xi}
\frac{d}{dt} \bI_{\lambda} =  \bJ^4_{\lambda} + \bJ^{6,1}_{\lambda} + \bJ^8_{\lambda} + \bK_\lambda.
\end{equation}

Here the quartic contribution $\bJ^4_{\lambda}$ is 
the same as in the linear case,
\[
\bJ^4_{\lambda}(u,v) = \int M_\lambda(u) E_{\lambda}(v) 
+ M_\lambda(v) E_{\lambda}(u)
- 2 P_{\lambda}(u) P_{\lambda}(v)\, dx.
\]
 The nonlinear contributions can be  first seen in he sixth order term $\bJ^6_{\lambda}$ which  has the form
\begin{equation}\label{J6-def}
\begin{aligned}
\bJ^{6,1}_\lambda(u,v) =  \int & \ M_\lambda(u) R^4_{p,\lambda}(v)+ B^4_{m,\lambda}(u) E_{\lambda}(v)
- P_{\lambda}(u)B^4_{p,\lambda}(v)- R^4_{m,\lambda}(u)P_{\lambda}(v) 
\\ & \ 
+ M_\lambda(v) R^4_{p,\lambda}(u)+ B^4_{m,\lambda}(v) E_{\lambda}(u)
- P_{\lambda}(v)B^4_{p,\lambda}(u)- R^4_{m,\lambda}(v)P_{\lambda}(u)\,  dx .
\end{aligned}
\end{equation}
Next, we detail the expression of $\bJ_8$
\begin{equation}\label{J8-def}
\bJ^8_{\lambda}(u,v) =   \int 
B^4_{m,\lambda}(u) R^4_{p,\lambda}(v) - R^4_{m,\lambda}(u) B^4_{p,\lambda}(v)
+ B^4_{m,\lambda}(v) R^4_{p,\lambda}(u) - R^4_{m,\lambda}(v) B^4_{p,\lambda}(u)
 \, dx .
\end{equation}

Finally the last error term $\bK_{\lambda}$ has the form
\begin{equation}\label{K8-def-ab}
\begin{aligned}
\bK_\lambda(u,v) = \iint_{x > y} & \ \ms_\lambda(u)(x) (F^4_{p,\lambda}(v)+ R^6_{p,\lambda}(v))(y)  + \ps_{\lambda}(v)(y) ( F^4_{m,\lambda}(u)+  R^6_{m,\lambda}(u))(x) 
\\ & \ - 
\ms_\lambda(v)(y) (F^4_{p,\lambda}(u)+ R^6_{p,\lambda}(u))(x)  - \ps_{\lambda}(u)(x)(F^4_{m,\lambda}(v)+ R^6_{m,\lambda}(v))(y) \,
dx dy.
\end{aligned}
\end{equation}

As in the linear case, we apply the normal form type correction \eqref{Isharp} to the interaction Morawetz functional  in order to obtain the final relation
\begin{equation}\label{interaction-bal}
\frac{d \is}{dt} = \|\partial_xQ(u_\lambda,\bv_\lambda)\|_{L^2}^2 + \bJ^{6}_{\lambda} + \bJ^8_{\lambda} + \bK_{\lambda},
\end{equation}
with $Q$ defined as in \eqref{def-q}, and
where  $\bJ^{6,1}_{\lambda}$ has  to be modified as follows
\begin{equation}
  \bJ^{6}_{\lambda}:= \bJ^{6,1}_{\lambda}+ \bJ^{6,2}_{\lambda}  ,
\end{equation}
where
\[
\bJ^{6,2}_{\lambda}:= \Lambda^6( \frac{d}{dt}\bB^4_I) 
\]
represents 
the nonlinearity contributions in the time derivative of $\bB^4_I$. This computation is equally valid in the GKG case $\gamma < -1$.

We remark that all contributions in the above relation are balanced, except for two terms,
namely $\bK_\lambda$ and $\bJ^{6,2}_{\lambda}$. These terms 
we will want to estimate perturbatively in the 
sequel. To set the notations, we will separate  
$\bJ^{6,2}_{\lambda}$ and correspondingly $\bJ^6_\lambda$ into a balanced and an unbalanced part, e.g.
\begin{equation}
\bJ^6_\lambda=   \bJ^{6,bal}_\lambda +  \bJ^{6,unbal}_\lambda,
\end{equation}
where $  \bJ^{6,bal}_\lambda:=  \bJ^{6,1}_\lambda + \bJ^{6,2,bal}_\lambda$, and   $\bJ^{6,unbal}_\lambda$ represents the unbalanced component of $\bJ^{6,2}_{\lambda}$.
By direct examination, the balanced part has a regular symbol 
\begin{equation}
 j^{6,bal}_{\lambda} \in S^{3\delta+\gamma }_\lambda  .
\end{equation}

This last symbol plays a key role in establishing 
the $L^6_{t,x}$ estimates via the defocusing condition.
From an algebraic perspective, in order to achieve this it is important to compute the symbol of $\bJ^{6,bal}_{\lambda}$ on the diagonal
$\xi_1 = \xi_2=\xi_3=\xi_4=\xi_5=\xi_6:=\xi$. 

\begin{lemma}\label{l:j6-diag}
The diagonal trace of the symbol $j^6_{\lambda}$
is 
\begin{equation}\label{good-J6}
j^{6,bal}_{\lambda}(\xi) = 2 (\phi_\lambda (\xi) )^4 c (\xi, \xi, \xi) a''(\xi).
\end{equation}
\end{lemma}
This result applies equally in the case $\gamma > -1$ and in the case $\gamma < -1$.
\begin{proof}
We begin with $j^{6,1}_\lambda$, whose symbol is 
\[
\begin{aligned}
j^{6,1}_\lambda(\xi) = & \ (\phi_\lambda(\xi))^2 (2 r^4_{p,\lambda}(\xi)
+ 2 b^4_{m,\lambda}(\xi) a'^2 (\xi) -2 a'(\xi)
b^4_{p,\lambda}(\xi) - 2 r^4_{m,\lambda}(\xi) a'(\xi)).
\end{aligned}
\]
One could do here a direct computation, but it is
more elegant to take advantage of a Galilean type transformation, i.e. a linear change of coordinates
$y = x+vt$, as discussed in Section~\ref{s:galilei}. This leaves all the terms in the relation \eqref{interaction-xi} unchanged. However,
it replaces the dispersion relation $a(\xi) \to a(\xi)+v\xi$, therefore, given any frequency $\xi$,
we can make a suitable choice for $v$ to insure that $a'(\xi) = 0$. This effectively  reduces 
the problem to the case when $a'(\xi) = 0$.
Then we are left with the simpler expression
\[
j^{6,1}_{\lambda}(\xi_0) = 2 (\phi_\lambda(\xi) )^2r^4_{p,\lambda}(\xi).
\]
For $r^4_{p,\lambda}$ we have the relation 
\[
c^4_{p,\lambda} +  i  \Delta^4 a \,  b^4_{p,\lambda} = i \Delta^4 \xi r^4_{p,\lambda} + i f^4_{p,\lambda},
\]
where we recall that the balanced component of $f^4_{p, \lambda}$, displayed in \eqref{fbal-m}, vanishes of second order on the diagonal. 
We differentiate with respect to $\xi_1$ and then set all $\xi$'s equal, using $a'(\xi) = 0$ to obtain
\[
\partial_1 c^4_{p,\lambda} (\xi) = i r^4_{p,\lambda} (\xi).
\]
It remains to compute the $\xi_1$ derivative of 
$c^4_{p,\lambda}$ on the diagonal. This is a direct 
computation using the formula \eqref{c4pl}. 
Since we are assuming that $a'(\xi) = 0$, the only 
nonzero contributions in the derivative of 
$c^4_{p,\lambda}$ arise from the derivative falling
on $p$, where we note that 
\[
\partial_1 p(\xi,\xi) = \frac12 (\partial_1+\partial_2) p(\xi,\xi) = \frac12 a''(\xi).
\]
Then we obtain 
\[
r^4_{p,\lambda} (\xi) = (\phi_{\lambda}(\xi))^2 c(\xi) a''(\xi),
\]
where we recall that $c$ is real on the diagonal.
Therefore
\begin{equation*}
j^{6,1}_{\lambda}(\xi) = 2 (\phi_\lambda(\xi))^4  c(\xi) a''(\xi),
\end{equation*}
as needed.

It remains to show that the second component $j^{6,2}_{\lambda}$ vanishes on the diagonal.
Fortunately this does not require the exact expression for the symbol $b^4_I$, only the fact
that it is real and symmetric. Precisely, we 
have 
\[
\begin{aligned}
j^{6,2}(\xi_1,\xi_2,\xi_3,\xi_4,\xi_5,\xi_6)
=  \Sym & \ [2 i b^4_I(\xi_1+\xi_5-\xi_6,\xi_2,\xi_3,\xi_4)
c(\xi_1,\xi_6,\xi_5)\\
& \ \qquad - 2i b^4_I(\xi_1,\xi_2-\xi_5+\xi_6,\xi_3,\xi_4)
\bar c(\xi_2,\xi_5,\xi_6)]
\end{aligned}
\]
which clearly vanishes on the diagonal since $c$ 
is real there.

\end{proof}


\subsection{The nonlinear equation, semi-balanced case}  We recall that here we are considering the case  $\mu \neq \lambda$ but $\mu \approx \lambda$. 
Compared with the linear case,  
we use the modified mass \eqref{Ml-sharp} in the interaction Morawetz functional
\begin{equation}\label{interaction-semibal}   
\bI(u,v) = \int_{x > y} \ms_\lambda(u)(x)  \ms_\mu(v)(y)  \,dx dy.
\end{equation}
Its time derivative is 
\[
\begin{aligned}
\frac{d}{dt} \bI(u,v) =  & \ \iint_{x > y} \partial_x (P_\lambda(u)+R^4_{m,\lambda}(u_\lambda))(x) \ms(v_\mu)(y)
 +  \ms(u_\lambda)(x) \partial_y  
(P(v_\mu)+R^4_{m,\mu}(v))(y)
\\ & + (F^4_{\lambda,m}(u)+R^6_{m,\lambda}(u))(x) \ms_\mu(v)(y) 
+ \ms_\lambda(u)(x)(F^4_{\mu,m}(v)+R^6_{m,\mu}(v))(y) \, dx dy.
\end{aligned}
\]
The contribution of the second line 
is denoted by $\bK_{\lambda\mu}$. For the first line we integrate by parts and expand by homogeneity to arrive at 
\[
\begin{aligned}
 \int & (P_\lambda(u)+R^4_{m,\lambda}(u)) (M_\mu(v) + B^4_{m,\mu}(v))
 -  (M_\lambda(u) + B^4_{m,\lambda}(u))  
(P(v_\mu)+R^4_{m,\mu}(v))\, dx
\\
& : = \bJ_{\lambda\mu}^4 + \bJ_{\lambda\mu}^{6,1}+ \bJ_{\lambda\mu}^8
.\end{aligned}
\]
Here $\bJ_{\lambda\mu}^4$ is the same as in the linear case,
\[
\bJ_{\lambda\mu}^4(u,v) = \bJ(u_\lambda,v_\mu).
\]
Hence, applying the normal form type correction to the interaction Morawetz functional we obtain the final relation
\begin{equation}
\frac{d \is}{dt} = \|Q(u_\lambda,\bv_\mu)\|_{L^2}^2 + \bJ^{6}_{\lambda\mu} + \bJ^8_{\lambda\mu} + \bK^8_{\lambda\mu},
\end{equation}
where $\bJ^{6,1}_{\lambda\mu}$ has  to be modified to 
\begin{equation}
  \bJ^{6}_{\lambda\mu}= \bJ^{6,1}_{\lambda\mu}+ \bJ^{6,2}_{\lambda\mu} , 
\end{equation}
where
$\bJ^{6,2}_{\lambda\mu}$ represents 
the contribution of the time derivative of $\bB^4_I$. This has a symbol with the the same 
size and regularity as in the balanced case, 
with the only difference that it is supported 
away from the diagonal. So in this case there
is no balanced component in $\bJ^{6,2}_{\lambda\mu}$.

As in the balanced case, the outcome of these computations is the same in the case
$\gamma > -1$ and $\gamma < -1$, with the restriction that the signs are matched for 
the dyadic regions in the latter case.

\subsection{The nonlinear equation, unbalanced case}\label{s:unbal} 
This is the case where we need to distinguish between the GNLS case 
and the GKG case, and further separate
matched and mismatched frequencies in the latter situation.

\subsubsection{The GNLS case $\gamma > -1$.}
We carry out this computation for  frequency localized pieces $u_\lambda$ and $v_\mu$ of two solutions $u$ and $v$ to \eqref{eq:main}, under the assumption $\mu \ll \lambda$, without distinguishing between the matched and mismatched signs.
We use the interaction Morawetz functional
\begin{equation}\label{interaction-unbal}   \bI_{\lambda\mu}(u,v) = \int_{x > y} \rPs_\lambda(u)(x)  \ms_\mu(v)(y)  \,dx dy.
\end{equation}
Its time derivative is computed as before, 
\[
\begin{aligned}
\frac{d}{dt} \bI_{\lambda\mu}(u,v) =  & \ \iint_{x > y} \partial_x (M_\lambda(u)(x)+R^4_{\rp,\lambda}(u) \ms(v_\mu)
 +  \rP^{\sharp,\pm}(u_\lambda)(x) \partial_y  
(P^\pm(v_\mu)+R^4_{m,\mu}(v))(y)
\\ & + (F^4_{\lambda,\rp}(u)+R^6_{\rp,\lambda}(u))(x) \ms_\mu(v)(y) 
+ \rP^{\sharp,\pm}_\lambda(u)(x)(F^4_{\mu,m}(v)+R^6_{m,\mu}(v))(y) \, dx dy.
\end{aligned}
\]
The contribution of the second line 
is denoted by $\bK_{\lambda\mu}$. For the first line we integrate by parts and expand to arrive at 
\[
\begin{aligned}
 \int & (M_\lambda(u)+R^4_{\rp,\lambda}(u)) (M_\mu(v) + B^4_{m,\mu}(v))
 +  (\rP_\lambda(u) + B^4_{\rp,\lambda}(u))  
(P(v_\mu)+R^4_{m,\mu}(v))\, dx
\\
& : = \bJ_{\lambda\mu}^4 + \bJ_{\lambda\mu}^{6}+ \bJ_{\lambda\mu}^8
.\end{aligned}
\]
Here $\bJ_{\lambda\mu}^4$ is the same as in the linear case. Hence we obtain the final relation
\begin{equation}\label{interaction-xi-unbal}
\frac{d \is}{dt} = \bJ^{4}_{\lambda\mu} + \bJ^{6}_{\lambda\mu} + \bJ^8_{\lambda\mu} + \bK^8_{\lambda\mu},
\end{equation}
where we roughly expect that
\[
\bJ^{6}_{\lambda\mu}=
\|u_\lambda\bv_\mu\|_{L^2}^2(1+ O((\mu/\lambda)^{\gamma+1})) .
\]

\subsubsection{The GKG case $\gamma < -1$, matched signs.}
The computation here is for  frequency localized pieces $u_\lambda$ and $v_\mu$ of two solutions $u$ and $v$ to \eqref{eq:main}, under the assumption $\mu \ll \lambda$, where the signs are matched for the two dyadic frequency regions. 

The interaction Morawetz functional is slightly changed, with the roles of the 
frequencies $\lambda$ and $\mu$ interchanged, namely
\begin{equation}\label{interaction-unbal-b}   \bI_{\lambda\mu}(u,v) = \int_{x > y} \ms_\lambda(u)(x)  \rP^{\sharp, \pm}_\mu(v)(y)  \,dx dy.
\end{equation}
Its time derivative is computed as before but  using \eqref{df-msharp-b} and \eqref{rPl-sharp-rel}. This gives
\[
\begin{aligned}
\frac{d}{dt} \bI(u,v) =  & \ \iint_{x > y} \partial_x (M_\lambda(u)(x)+R^4_{\rp,\lambda}(u)) \ms(v_\mu)
 +  \rP^{\sharp,\pm}(u_\lambda)(x) \partial_y  
(P^{\pm}(v_\mu)+R^4_{m,\mu}(v))(y)
\\ & + (F^4_{\lambda,\rp}(u)+R^6_{\lambda,\rp}(u))(x) \ms_\mu(v)(y) 
+ \rP^{\sharp,\pm}_\lambda(u)(x)(F^4_{\mu,m}(v)+R^6_{\mu,m}(v))(y) \, dx dy.
\end{aligned}
\]
For the first line we integrate by parts and expand as before, arriving at the
multilinear expansion
\begin{equation}\label{interaction-xi-unbal-b}
\frac{d \is}{dt} = \bJ^{4}_{\lambda\mu} + \bJ^{6}_{\lambda\mu} + \bJ^8_{\lambda\mu} + \bK^8_{\lambda\mu},
\end{equation}
where $\bJ^4_{\lambda\mu}$ is the same as in the linear case, and for which
we roughly expect that
\[
\bJ^{6}_{\lambda\mu}=
\|u_\lambda\bv_\mu\|_{L^2}^2(1+ O((\lambda/\mu)^{\gamma+1})) .
\]

\subsubsection{The GKG case, $\gamma < -1$, mismatched signs.}

In this last configuration we consider  frequency localized pieces $u_\lambda$ and $v_\mu$ of two solutions $u$ and $v$ to \eqref{eq:main}, under the assumption $\mu, \lambda \gg 1$, where the signs are mismatched for the two dyadic frequency regions but no correlation is assumed between $\lambda$ and $\mu$. To fix the signs we assume that $u_\lambda$ is at positive frequency, and $v_\mu$ is at negative frequency.

The interaction Morawetz functional in this case is
\begin{equation}\label{interaction-unbal-c}   \bI_{\lambda\mu}(u,v) = \int_{x > y} \ms_\lambda(u)(x)  \ms_\mu(v)(y)  \,dx dy.
\end{equation}
Its time derivative is computed using  
\eqref{df-msharp-b} with the $+$ sign for $u_\lambda$, respectively the $-$ sign for 
$v_\mu$,
\[
\begin{aligned}
\frac{d}{dt} \bI(u,v) =  & \ \iint_{x > y} 
(v^+-v^-) \partial_x \ms_\lambda(u)  \ms_\mu(v)  
\\ & \ 
+ \partial_x (P^+_\lambda(u)(x)+R^4_{\lambda,m}(u) \ms(v_\mu)
 +  \ms(u_\lambda)(x) \partial_y  
(P^-(v_\mu)+R^4_{\mu,m}(v))(y)
\\ & + (F^4_{\lambda,m}(u)+R^6_{\lambda,m}(u))(x) \ms_\mu(v)(y) 
+ \ms_\lambda(u)(x)(F^4_{\mu,m}(v)+R^6_{\mu,m}(v))(y) \, dx dy,
\end{aligned}
\]
where the first term appears due to the mismatched asymptotic velocities.

The contribution of the last line is denoted by $\bK_{\lambda\mu}$.
For the first two lines we integrate by parts and expand, arriving at the
multilinear expansion
\begin{equation}\label{interaction-xi-unbal-c}
\frac{d \is}{dt} = \bJ^{4}_{\lambda\mu} + \bJ^{6}_{\lambda\mu} + \bJ^8_{\lambda\mu} + \bK^8_{\lambda\mu},
\end{equation}
where $\bJ^4_{\lambda\mu}$ is the same as in the linear case, and for which
we roughly expect that
\[
\bJ^{6}_{\lambda\mu}=
\|u_\lambda\bv_\mu\|_{L^2}^2((v^--v^+)+ O(\lambda^{\gamma+1} + \mu^{\gamma+1}) ) .
\]


\section{Frequency envelopes and the bootstrap argument}

\label{s:boot}

The main objective of the proof of our main result in Theorem~\ref{t:main} is to establish a global $L^\infty_t H^s_x$ bound for solutions whose initial data 
is small in the critical Sobolev space $H^{s_c}$; by the local well-posedness result in Section~\ref{s:local}, this implies
the desired global well-posedness result. However, along the way, we will also establish bilinear $L^2_{t,x}$ and the $L^6_{t,x}$ Strichartz bounds for the solutions. These will both play an essential role in the proof of Theorem~\ref{t:main},
and will  also establish the scattering properties of our global solutions.

The proofs of our three estimates,
\begin{itemize}
    \item energy bounds
    \item bilinear $L^2_{t,x}$ bounds
    \item $L^6_{t,x}$ Strichartz bounds
\end{itemize}
are all connected in a complex manner,
which makes it natural to frame them in  setting of a bootstrap argument, where we already assume that the desired energy, bilinear and Strichartz estimates hold but with weaker constants. 

The set-up for the bootstrap is most conveniently described using the language  of frequency envelopes. This was originally introduced  in work of Tao, see e.g. \cite{Tao-WM} in the context of dyadic Littlewood-Paley decompositions, which is exactly the organization we 
also adopt for this paper.

\bigskip

To start with, we assume that the initial data has small size,
\[
\| \du_0\|_{H^s} \lesssim \epsilon.
\]
We consider a frequency decomposition for the initial  data on the dyadic scale, 
\[
\du_0 = \sum_{\pm}\sum_{\lambda  \in 2^\N} \du_{0,\lambda}^\pm.
\]
Then we place the initial data components under an admissible  frequency envelope on the unit scale,
\[
\|\du_{0,\lambda}^\pm\|_{H^{s_c}} \leq \epsilon c_\lambda^\pm, \qquad c \in \ell^2,
\]
where the envelope $\{c_\lambda^\pm \}$ is not too large, in the sense that we have 
both a critical norm bound 
\begin{equation}\label{l2-env}
\| c\|_{\ell^2} \approx 1,
\end{equation}
and also an $H^s$ bound\footnote{This is where the unbalanced slowly varying condition \eqref{slow-unbal} comes into play.},
\begin{equation}\label{Hs-data-env}
\epsilon^2 \sum_{\pm} \sum_\lambda \lambda^{2(s-s_c)} (c_\lambda^\pm)^2
\approx \| \du_0\|_{H^s}^2.
\end{equation}
Our frequency envelope $c_\lambda^\pm$ will be assumed to be slowly varying. Here we adopt the unbalanced version \eqref{slow-unbal} of the slowly varying condition, which allows
us to use the frequency envelope in order
to measure not only the critical Sobolev 
norm $H^{s_c}$ but also a range of higher\footnote{ Also lower, but only in a small range near $s_c$.} 
Sobolev norms, including the $H^s$ norm where we have local well-posedness, in particular allowing us to assume that
\eqref{Hs-data-env} holds.

Our goal will be to establish the following frequency envelope bounds for the solution:

\begin{theorem}\label{t:boot}
Let $s$ be as in the local well-posedness result in Theorem~\ref{t:main}.
 Let $u \in C([0,T];H^s)$ be a solution for the equation \eqref{eq:main}
with initial data $\du_0$  which has $H^{s_c}$ size at most $\epsilon$.
Let $\{\epsilon c_\lambda^\pm\}$ be an admissible frequency envelope for the initial data 
in $H^{s_c}$, with $c_\lambda^\pm$ normalized in $\ell^2$, as in \eqref{l2-env}. Then the solution $u$ satisfies 
the following dyadic bounds:
\begin{enumerate}[label=(\roman*)]
\item Uniform frequency envelope bound:
\begin{equation}\label{uk-ee}
\| u_\lambda^{\pm} \|_{L^\infty_t L^2_x} \lesssim \epsilon c_\lambda^\pm \lambda^{-s_c},
\end{equation}
\item Localized Strichartz bound:
\begin{equation}\label{uk-se}
\| u_\lambda^{\pm} \|_{L_{t,x}^6} \lesssim (\epsilon c_\lambda)^\frac23 \lambda^{-\frac{4s_c-1+3\delta}{6}},
\end{equation}
\item Transversal bilinear $L^2$ bound:
\begin{equation} \label{uab-bi}
\| \partial_x(u_\lambda^{\pm} \bu_\mu^{\pm,x_0})  \|_{L^2_{t,x}} \lesssim \epsilon^2 c_\lambda^{\pm} c_\mu^{\pm}\, \lambda^{-s_c} \mu^{-s_c} \frac{\lambda + \mu}{(\lambda^{\gamma+1}
+ \mu^{\gamma+1})^\frac12}, \qquad x_0 \in \R,
\end{equation}
with the following improvement if $\gamma < -1$:
\begin{equation} \label{uab-bi-pm}
\| u_\lambda^{+} \bu_\mu^{-,x_0}  \|_{L^2_{t,x}} \lesssim \epsilon^2 c_\lambda^{+} c_\mu^{-}\, \lambda^{-s_c} \mu^{-s_c}  \qquad x_0 \in \R.
\end{equation}

\end{enumerate}
\end{theorem}
We emphasize that the time $T$ plays no role in the proof, and in particular  all the implicit constants in the theorem are independnt of $T$.
We also note that, as a particular case of \eqref{uab-bi}, we have the 
diagonal bound
\begin{equation}\label{uk-bi}
\| \partial_x |u_\lambda|^2  \|_{L^2_{t,x}} \lesssim \epsilon^2 c_\lambda^2 \lambda^{-2s - \frac{\gamma-1}{2}}.
\end{equation}
Assuming this result, we can immediately prove the global well-posedness result in Theorem~\ref{t:main}:

\begin{proof}[Proof of Theorem~\ref{t:main}]
Given initial data $u_0 \in H^s$, we choose 
the frequency envelope $c_k^\pm$ so that 
both \eqref{l2-env} and \eqref{Hs-data-env} hold. By the local well-posedness result,
a local $H^s$ solution exists on some interval $[0,T]$, and then Theorem~\ref{t:boot} applies
there.  We claim that this solution is global.
Indeed, otherwise the forward solution exists only on a maximal time interval $[0,T_{max})$, where, by the local well-posedness result, we
must have 
\begin{equation}\label{blow}
\lim_{T \nearrow T_{max}} \|u(T)\|_{H^s} = \infty.    
\end{equation}
On the other hand, Theorem~\ref{t:boot} can be applied on $[0,T]$, where \eqref{uk-ee} combined with \eqref{Hs-data-env} yields in particular the uniform bound 
\[
\| u(T)\|_{H^s} \lesssim \|u_0\|_{H^s}.
\]
This contradicts \eqref{blow}, and thus the maximality of $T_{max}$.

It then follows that the solution $u$ is global, and the conclusion of Theorem~\ref{t:boot} applies uniformly on any interval $[0,T]$, and thus on $[0,\infty)$.
Now it is a straightforward exercise to show, by direct dyadic summation,  that the dyadic bounds in Theorem~\ref{t:boot} imply the global bounds in Theorem~\ref{t:main}. Finally we not that the same argument applies for $t < 0$.
\end{proof}

It remains to prove Theorem~\ref{t:boot}, which is the goal of the rest of the paper. For this
we make a bootstrap assumption where we assume the same bounds but with a worse constant $C$, as follows:

\begin{enumerate}[label=(\roman*)]
\item Uniform frequency envelope bound,
\begin{equation}\label{uk-ee-boot}
\| u_\lambda^\pm \|_{L^\infty_t L^2_x} \lesssim C \epsilon c_\lambda^\pm \lambda^{-s_c},
\end{equation}
\item Localized Strichartz bound,
\begin{equation}\label{uk-se-boot}
\| u_\lambda^\pm \|_{L_{t,x}^6} \lesssim C (\epsilon c_\lambda^\pm)^\frac23 \lambda^{-\frac{4s_c-1+3\delta}{6}},
\end{equation}
\item Transversal Interaction Morawetz, 
\begin{equation} \label{uab-bi-boot}
\| \partial_x(u_\lambda^\pm \bu_\mu^{\pm,x_0})  \|_{L^2_{t,x}} \lesssim C^2 \epsilon^2 c_\lambda^{\pm} c_\mu^{\pm}\, \lambda^{-s_c} \mu^{-s_c} \frac{\lambda + \mu}{(\lambda^{\gamma+1}
+ \mu^{\gamma+1})^\frac12}
\end{equation}
uniformly for all $x_0 \in \R$.
\end{enumerate}

Then we seek to improve the constant in these bounds. The gain will  come from the fact that the $C$'s will always come paired with extra $\epsilon$'s. Thus it will suffice to first choose $C$ a sufficiently large,  universal constant, and then to assume that $\epsilon$ is 
sufficiently small relative to $C$,
\begin{equation}\label{C-epsilon}
\epsilon \ll_C 1.
\end{equation}

We conclude this section with three 
remarks, which clarify different aspects 
of the proofs in the following sections:

\begin{remark}[On the need to add translations to the bilinear $L^2_{t,x}$ estimates] This is necessary because, unlike the linear bounds
\eqref{uk-ee-boot} and \eqref{uk-se-boot} which are inherently invariant with respect to translations, bilinear estimates are not invariant 
with respect to separate translations for the two factors.
One immediate corollary of \eqref{uab-bi-boot} is that for any translation invariant bilinear operator $L$ with integrable kernel we have 
\begin{equation} \label{uab-bi-boot-trans}
\| \partial_x L(u_\lambda^\pm, \bu_\mu^\pm)  \|_{L^2_{t,x}} \lesssim C^2 \epsilon^2 c_\lambda^\pm c_\mu^\pm \, \lambda^{-s_c} \mu^{-s_c} \frac{\lambda + \mu}{(\lambda^{\gamma+1}
+ \mu^{\gamma+1})^\frac12}.
\end{equation}
This is essentially the only way we will use this translation invariance
in our proofs. We note that the uniformity 
with respect to $x_0$ is universal in the unbalanced case, when $\lambda$ and $\mu$ 
are separated. However, in the balanced case one only expects this to happen for semilinear problems, whereas in the quasilinear case we need to penalize 
the bound if $x_0$ goes above the uncertainty principle threshold, see e.g.
\cite{IT-qnls}.
\end{remark}

\begin{remark} [On the $\pm$ set-up]
In this article we have chosen to introduce the $\pm$ notation in order to differentiate between the positive and negative frequencies, which are separate in one space dimension, and thus may 
even correspond to different symbol behaviors for $a$ and $c$.

However, at a technical level this makes little difference. To prove Theorem~\ref{t:main}, it suffices to take 
$c_\lambda^+ = c_\lambda^-$ in Theorem~\ref{t:boot}. The only estimate where the signs make a difference, namely the bound \eqref{uab-bi-pm}, is not needed
within the bootstrap loop, and instead it is just an interesting consequence of our 
analysis, which is stated for the sake of completeness, but otherwise of little consequence.

Because of the above considerations, within the proof of Theorem~\ref{t:boot}
we simply \textbf{drop the $\pm$ notation},
and simply remark on the differences 
in the proof of \eqref{uab-bi-pm} at the 
appropriate place.

\end{remark}

\begin{remark}[On the continuity argument]
The fact that it suffices to prove 
Theorem~\ref{t:boot} under the bootstrap assumptions \eqref{uk-ee-boot}-\eqref{uab-bi-boot} relies on a standard continuity argument. We omit the details and instead refer the reader to the argument in \cite{IT-global}, which is equally valid
here. 
\end{remark}

  
\section{The frequency envelope bounds}
\label{s:fe-bounds}

The aim of this section is to prove the frequency envelope bounds in 
Theorem~\ref{t:boot}, given the bootstrap assumptions \eqref{uk-ee-boot}-\eqref{uab-bi-boot}. In the proof we will rely
on our modified energy and momentum density-flux relations, whose components
we estimate first. The frequency localized energy estimate \eqref{uk-ee} will be
an immediate consequence of these bounds. For the Strichartz and $L^2_{t,x}$
bilinear bounds we will then use the interaction Morawetz identities,
first in a localized diagonal setting and then in a transversal setting. 

The two cases $\gamma > -1$ and $\gamma < -1$ will be treated in a similar fashion
up until the point we discuss the unbalanced case in the last subsection, 
where we discuss them separately.

\subsection{Spatial and space-time \texorpdfstring{$L^1$}{} bounds}
Here we consider the localized mass corrections $B^4_{\lambda,m}$ and errors $F^4_{\lambda,m}$, $R^6_{\lambda,m}$ 
and their momentum counterparts. For  $B^4_{\lambda,m}$ we will prove a fixed time $L^1$ bound, while for $F^4_{\lambda,m}$ and $R^6_{\lambda,m}$ we will prove a space-time $L^1$ bound.

These bounds will be repeatedly used in each of the following subsections, and essentially allow us  to treat these expressions perturbatively.  We begin with the 
$B^4_{\lambda,m}$ bound.

\begin{lemma}\label{l:B4-multi}
Assume that the bootstrap bound \eqref{uk-ee-boot} holds. Then we have the fixed time estimates 
\begin{equation}\label{b4-ma}
\| B^4_{\lambda,m}(u) \|_{L^1_x} \lesssim \epsilon^4 C^4 \lambda^{-2s_c}
c_\lambda^2,
\end{equation}
\begin{equation}\label{b4-pa}
\| B^4_{\lambda,p}(u) \|_{L^1_x} \lesssim \epsilon^4 C^4 \lambda^{-2s_c+ \gamma+1}
c_\lambda^2,
\end{equation}
\begin{equation}\label{b4-rpa}
\| B^4_{\lambda,\rp}(u) \|_{L^1_x} \lesssim \epsilon^4 C^4 \lambda^{-2s_c- \gamma-1}
c_\lambda^2.
\end{equation}
\end{lemma}

We remark that the fixed time mass bound above is also needed in order to propagate  the energy bounds, while the momentum bound and the reverse momentum bound 
above will only be needed as a component of the interaction Morawetz functional
in the balanced case, respectively the unbalanced case. 

\begin{proof}
The bounds \eqref{b4-ma}, \eqref{b4-pa}
and \eqref{b4-rpa} are similar, the only difference arises from the additional $\lambda^{\gamma+1}$ factor in the size of the momentum symbol, respectively the $\lambda^{-\gamma-1}$ factor in the size of the reverse momentum symbol $p_\lambda$. So we will only prove the first bound.

Based on the symbol bounds \eqref{brq4-size-nls} we can write 
\[
B^4_{\lambda,m}(u) = \lambda^{3\delta-\gamma-2} L(u_\lambda,\bu_\lambda,u_\lambda,\bu_\lambda) .
\]
Hence, using also Bernstein's inequality at frequency $\lambda$ we obtain
\[
\| B^4_{\lambda,m}(u)\|_{L^1} \lesssim
\|u_\lambda\|_{L^2}^2 \|u_\lambda\|_{L^\infty}^2
C^4 \epsilon^4 c_\lambda^4 \lambda^{-4s_c + 3\delta-\gamma-1}.
\]
Comparing the exponent above with the
desired exponent $-2s_c$, it suffices to verify  that 
\[
-2s_c = -4s_c + 3\delta-\gamma-1.
\]
This concludes the proof of the lemma.
\end{proof}

Next we turn our attention to $F^{4}_{\lambda,m}$, $F^{4}_{\lambda,p}$ and 
$F^{4}_{\lambda,\rp}$,
for which we need space-time bounds.
We split the analysis into the balanced and the unbalanced case, which require different
approaches. For the balanced case we have
\begin{lemma}\label{l:F4-bal}
Under our bootstrap assumptions \eqref{uk-ee-boot}-\eqref{uab-bi-boot}, we have the space-time bounds
\begin{equation}\label{F4-bal-m-bd}
\|F^{4,bal}_{\lambda,m}\|_{L^1_{t,x}} \lesssim \epsilon^4 C^4 c_\lambda^2 \lambda^{-2s_c},
\end{equation}
\begin{equation}\label{F4-bal-p-bd}
\|F^{4,bal}_{\lambda,p}\|_{L^1_{t,x}} \lesssim \epsilon^4 C^4 c_\lambda^2 \lambda^{-2s_c+\gamma+1},
\end{equation}
\begin{equation}\label{F4-bal-rp-bd}
\|F^{4,bal}_{\lambda,\rp}\|_{L^1_{t,x}} \lesssim \epsilon^4 C^4 c_\lambda^2 \lambda^{-2s_c-\gamma-1}.
\end{equation}
\end{lemma}

\begin{proof} It suffices to consider \eqref{F4-bal-m-bd}.
Based on the expression for the symbol $f^{4,bal}_{\lambda,m}$ in \eqref{fbal-m} and the symbol bound \eqref{brq4-size-nls}, we can write 
\[
F^{4,bal} (u,\bu,u,\bu) = \lambda^{3\delta-2} L ( \partial_x 
L(u_\lambda,\bu_\lambda), \partial_x L(u_\lambda,\bu_\lambda)).
\]
Using the bilinear $L^2_{t,x}$ bootstrap assumption 
\eqref{uab-bi-boot} and its consequence \eqref{uab-bi-boot-trans} we have 
\[
\|F^{4,bal} (u,\bu,u,\bu) \|_{L^1_{t,x}} 
\lesssim C^4 \epsilon^4 c_\lambda^4 \lambda^{3\delta-2} \lambda^{-4s_c- \gamma+1} 
= C^4 \epsilon^4 c_\lambda^4 \lambda^{3\delta-4s_c- \gamma-1} ,
\]
which again suffices.
\end{proof}

We now consider the unbalanced case, which will determine the primary constraints
on $\delta$ relative to $\gamma$.

\begin{lemma}\label{l:F4-unbal}
Under our bootstrap assumptions \eqref{uk-ee-boot}-\eqref{uab-bi-boot}, we have the space-time bounds
\begin{equation}\label{F4-unbal-m-bd}
\|F^{4,unbal}_{m,\lambda}\|_{L^1_{t,x}} \lesssim \epsilon^4 C^4 c_\lambda^2 \lambda^{-2s_c},
\end{equation}
\begin{equation}\label{F4-unbal-p-bd}
\|F^{4,unbal}_{p,\lambda}\|_{L^1_{t,x}} \lesssim \epsilon^4 C^4 c_\lambda^2 \lambda^{-2s_c+\gamma+1},
\end{equation}
\begin{equation}\label{F4-unbal-rp-bd}
\|F^{4,unbal}_{\rp,\lambda}\|_{L^1_{t,x}} \lesssim \epsilon^4 C^4 c_\lambda^2 \lambda^{-2s_c-\gamma-1}.
\end{equation}
\end{lemma}

\begin{proof}
We will focus on \eqref{F4-unbal-m-bd},
as the proofs of \eqref{F4-unbal-p-bd} and 
\eqref{F4-unbal-rp-bd}
are virtually identical.
Here $F^{4,unbal}_{\lambda,m}$ has four 
entries, one of which has frequency $\lambda$, and so that the remaining three also add up to $\lambda$,
\begin{equation}\label{F4-unbal-dyadic}
F_{\lambda,m}^{4,unbal} =  u_\lambda  \sum_{\lambda_1,\lambda_2,\lambda_3}^{unbal} (\lambda_1 \lambda_2 \lambda_3)^\delta P_\lambda L(\bfu_{\lambda_1},\bfu_{\lambda_2},\bfu_{\lambda_3}) : = \sum_{\lambda_1,\lambda_2,\lambda_3}^{unbal}  F_{\lambda,\lambda_1,\lambda_2,\lambda_3}^{4}. 
\end{equation}
In these summations we have excluded the balanced case, so
we cannot have more than two dyadic frequencies which are equal or adjacent. Precisely, we can split the above summation range into several cases, depending on the relative size of the four input frequencies $\{\lambda,\lambda_1,\lambda_2,\lambda_3\}$. To start with  we note that of  the three new frequencies, the smallest must be at least $\lambda$ and, if larger then $\lambda$, then the two largest must be comparable.

\bigskip

\textbf{I. The resonant case}
$\{\lambda,\lambda_1,\lambda_2,\lambda_3\} = \{\lambda,\lambda,\mu,\mu\}$ with $\mu \neq \lambda$,   where we can have resonant interactions.
Using two bilinear $L^2_{t,x}$ bounds \eqref{uab-bi} for two $(\lambda,\mu)$ pairs we obtain
\[
\| F_{\lambda,\lambda,\mu,\mu}\|_{L^1_{t,x}} \lesssim C^4 \epsilon^4 
\lambda^\delta \mu^{2\delta}  \mu^{-2s_c} \lambda^{-2s_c}  |\lambda^{\gamma+1} + \mu^{\gamma+1}|^{-1} c_\lambda^2 c_\mu^2.
\]
This bound is tight if $\lambda = \mu$, so in order to guarantee 
the dyadic summation with respect to $\mu$,
it remains to check the power of $\mu$  when $\mu$ is above and below $\lambda$. We consider two ranges for $\gamma$:
\medskip

a) $\gamma >-1$. For $\mu < \lambda$ 
the power of $\mu$ is $2\delta -2s_c$.
This should be nonnegative so we need $s_c \leq \delta$.
For $\mu >\lambda$ the power of $\mu$ is
$2\delta - 2s_c -\gamma - 1$. 
This should be non-positive so we need $2s_c \geq 2\delta - \gamma -1$. 
Overall we obtain the
constraints
\begin{equation}\label{F4-res-a}
\blue{ 0 \leq \delta \leq \gamma+1, 
\qquad \gamma >  -1,}
\end{equation}
after which the $\mu$ summation is straightforward.
\medskip

b) $\gamma <-1$. For $\mu < \lambda$ 
the power of $\mu$ is $2\delta -2s_c-\gamma -1$, which should be nonnegative.
For $\mu >\lambda$ the power of $\mu$
is $2\delta - 2s_c$ which should be nonpositive. We obtain the constraints
\begin{equation}\label{F4-res-b}
\blue{ \gamma+1  \leq \delta \leq 0, 
\qquad \gamma < -1.}
\end{equation}
Thus for the resonant case we obtain the 
bound \eqref{F4-unbal-m-bd} in the range
given by \eqref{F4-res-a}, \eqref{F4-res-b}.

\begin{remark}
The range given by \eqref{F4-res-a}, \eqref{F4-res-b} is important because we cannot defeat it with current tools, and may even be necessary in the absence of additional structure for the nonlinearity. The bound from above on $\delta$ is the requirement that the problem is semilinear. The one from below 
corresponds to our requirement that we can reach the scaling Sobolev exponent $s_c$; without it, we would
still get a corresponding range for the Sobolev exponent $s$, $s\geq s_{global}$, but with $s_{global}$
strictly larger than $s_c$.
\end{remark}

\bigskip

\textbf{II. The nonresonant case}, where we have three scenarios:

\bigskip

II(i): $\{\lambda,\lambda_1,\lambda_2,\lambda_3\} =\{\alpha,\lambda,\mu,\mu\}$  with  $\alpha < \lambda < \mu$, where the strict inequality is interpreted to mean full dyadic separation. Using two bilinear $L^2_{t,x}$ bounds \eqref{uab-bi-boot} for the frequency pairs $(\alpha,\mu)$ and $(\lambda,\mu)$ yields the bound 
\[
\| F_{\lambda,\alpha,\mu,\mu}\|_{L^1_{t,x}} \lesssim C^4 \epsilon^4 
\lambda^{-s_c} \alpha^{-s_c+\delta} \mu^{-2s_c+2\delta}  |\lambda^{\gamma+1} + \mu^{\gamma+1}|^{-\frac12} |\alpha^{\gamma+1} + \mu^{\gamma+1}|^{-\frac12} c_\alpha c_\lambda c_\mu^2.
\]
Here we pull out a $\lambda^{2s_c}$ factor 
as needed in \eqref{F4-unbal-m-bd}, 
and then examine the remaining coefficient
\[
\coeff = \lambda^{s_c} \alpha^{-s_c+\delta} \mu^{-2s_c+2\delta}  |\lambda^{\gamma+1} + \mu^{\gamma+1}|^{-\frac12} |\alpha^{\gamma+1} + \mu^{\gamma+1}|^{-\frac12}.
\]
The total exponent here is zero, so if $\alpha = \lambda = \mu$ then we have 
$\coeff\approx 1$. Hence in order to insure summation with respect to both 
$\alpha$ and $\mu$ we need to have 
some off-diagonal decay, precisely
a negative power of $\alpha$ and a nonpositive power of $\mu$.
Depending on $\gamma$ we separate 
as before into two cases.

\medskip

a) If $\gamma > -1$ then we obtain the coefficient
\[
\coeff= \lambda^{s_c} \alpha^{-s_c+\delta} 
\mu^{-2s_c+2\delta-\gamma-1}.
\]
Hence, in order to carry out the dyadic summation in $\alpha$ and $\mu$  we need 
\[
s_c-\delta < 0, \qquad -2s_c +2\delta -\gamma -1 \geq  0
\]
which are essentially the same constraints as in case I,
\begin{equation}\label{F4-nres-i}
\blue{ 0 \leq \delta < \gamma+1, 
\qquad \gamma >  -1,}
\end{equation}
but excluding the borderline quasilinear case.
\medskip

b) 
If $\gamma < -1$ then we obtain the coefficient
\[
\coeff= \lambda^{s_c - \frac{\gamma+1}2}
\alpha^{-s_c+\delta-\frac{\gamma+1}2} 
\mu^{-2s_c+2\delta}.
\]
This is again favourable 
if we have a positive power for $\alpha$ and a nonpositive power for $\mu$, 
\[
s_c - \delta \geq  0, \qquad -s_c+ \delta -\frac{\gamma+1}2 >0,
\]
which gives
\begin{equation}
\blue{ \gamma+1  \leq \delta <  0, 
\qquad \gamma < -1,}
\end{equation}
again similar to the same constraints as in the case I but excluding the borderline quasilinear case.

\bigskip

II(ii): $\{\lambda,\lambda_1,\lambda_2,\lambda_3\} =\{\lambda,\alpha,\mu,\mu\}$  with $ \lambda < \alpha < \mu $. 
Using two bilinear $L^2_{t,x}$ bounds for the pairs $(\alpha,\mu)$ and $(\lambda,\mu)$ yields\footnote{ Here we cannot have $\alpha \approx \mu$, as then the output of the three $\mu$ sized frequencies would also be at frequency $\mu$. }
\[
\| F_{\lambda,\alpha,\mu,\mu}\|_{L^1_{t,x}} \lesssim C^4 \epsilon^4 \lambda^{-s_c} \alpha^{-s_c+\delta} \mu^{-2s_c+2\delta}  |\lambda^{\gamma+1} + \mu^{\gamma+1}|^{-\frac12} |\alpha^{\gamma+1} +\mu^{\gamma+1}|^{-\frac12} c_\alpha c_\lambda c_\mu^2,
\]
where pulling out the $\lambda^{-2s_c}$ 
factor yields the same coefficient as in II(i), with the powers adding up to zero. For the dyadic summation with respect to $\alpha$ and $\mu$ we need a non-positive power for $\mu$ and a positive power for $\lambda$. We  separate as usual the two cases for $\gamma$:
\medskip

a)  If $\gamma > -1$ then we obtain the coefficient
\[
\coeff= \lambda^{s_c} \alpha^{-s_c+\delta} \mu^{-2s_c+2\delta-\gamma -1}.
\]
It is directly verified that the powers add to zero, so to insure summation with respect to $\alpha$ and $\mu$ we need 
\[
s_c > 0, \qquad -2s_c + 2\delta -\gamma -1 \leq 0.
\]
This gives
\begin{equation}\label{F4-nres-ii}
 \blue{\delta > \frac{\gamma+1}{3}, 
\qquad \gamma >  -1,}
\end{equation}
which is a more restrictive than in bound from below than in case $I$.

\medskip 

b) If $\gamma < -1$ then we obtain the coefficient
\[
\coeff= \lambda^{s_c-\frac{\gamma+1}2} \alpha^{-s_c+\delta-\frac{\gamma+1}2} \mu^{-2s_c+2\delta},
\]
where the dyadic summation  requires
again a positive power for $\lambda$ and a nonpositive one for $\mu$,
\[
s_c - \frac{\gamma+1}2 > 0, \qquad s_c-\delta >0,
\]
or equivalently
\[
\blue{  \delta > \frac23 (\gamma+1), 
\qquad \gamma <  -1},
\]
again more restrictive than in I.

\bigskip

II(iii) $\{\lambda,\lambda_1,\lambda_2,\lambda_3\} =\{\alpha,\mu,\lambda,\lambda\}$  with  $\alpha < \mu <  \lambda$. Using two bilinear $L^2_{t,x}$ bounds for the pairs $(\alpha,\lambda)$ and $(\mu,\lambda)$ yields\footnote{ Here we cannot have $\mu \approx \lambda$, as this  would also imply $\alpha \approx \lambda$.} 
\[
 \| F_{\lambda,\alpha,\mu,\mu}\|_{L^1} \lesssim C^4 \epsilon^4  \alpha^{-s_c+\delta} \mu^{-s_c+\delta}
 \lambda^{-2s_c+\delta} |\lambda^{\gamma+1} + \mu^{\gamma+1}|^{-\frac12} |\lambda^{\gamma+1} + \alpha^{\gamma+1}|^{-\frac12}c_\alpha c_\mu c_\lambda^2,
\]
where after pulling out the $\lambda^{-2s_c}$ factor we are left with the coefficient
\[
\coeff= \alpha^{-s_c+\delta} \mu^{-s_c+\delta}
 \lambda^{-2s_c+\delta} |\lambda^{\gamma+1} + \mu^{\gamma+1}|^{-\frac12} |\lambda^{\gamma+1} + \alpha^{\gamma+1}|^{-\frac12}.
\]
The power for $\alpha$ and $\mu$ is the same, and in order to insure dyadic summation they must be both positive.
We separate again the two cases for $\gamma$:

\medskip

a) If $\gamma > -1$ then we obtain the coefficient
\[
\coeff = \alpha^{-s_c+\delta} \mu^{-s_c+\delta}
 \lambda^{\delta-\gamma -1},
\]
where the summation in $\alpha$ and $\mu$
requires 
\[
- s_c+ \delta > 0,
\]
or equivalently
\[
\blue{\delta < \gamma+1,
 \qquad \gamma > -1}.
 \]

\medskip 

b) If $\gamma < -1$ then we obtain the coefficient
\[
\coeff = \alpha^{-s_c +\delta- \frac{\gamma+1}2} \mu^{-s_c+\delta- \frac{\gamma +1}2}
 \lambda^\delta.
\]
Hence for the dyadic summation we need
\[
-s_c+ \delta - \frac{\gamma+1}2 < 0,
\]
or equivalently
\[
\blue{\delta < 0,
 \qquad \gamma < -1}.
 \]
 This concludes the proof of the lemma.
\end{proof}

Finally, we turn our attention to the six-linear term $R^6_{\lambda}$ in our density-flux identities, which we estimate as follows:

\begin{lemma}\label{l:R6-ab}
Under our bootstrap assumptions \eqref{uk-ee-boot}-\eqref{uab-bi-boot}, we have the space-time bounds
\begin{equation}\label{R6-m-bd}
\|R^6_{\lambda,m}\|_{L^1_{t,x}} \lesssim \epsilon^4 C^6 c_\lambda^4 \lambda^{-2s_c},
\end{equation}
\begin{equation}\label{R6-p-bd}
\|R^6_{\lambda,p}\|_{L^1_{t,x}} \lesssim \epsilon^4 C^6 c_\lambda^4 \lambda^{-2s_c+\gamma+1},
\end{equation}
\begin{equation}\label{R6-rp-bd}
\|R^6_{\lambda,\rp}\|_{L^1_{t,x}} \lesssim \epsilon^4 C^6 c_\lambda^4 \lambda^{-2s_c-\gamma-1}.
\end{equation}
\end{lemma}

\begin{proof} 
As in the case of the previous Lemma, we will focus 
on \eqref{R6-m-bd}, as the proofs of \eqref{R6-p-bd} and \eqref{R6-rp-bd} are  essentially the same.
We recall that $R^6_{m,\lambda}$ is obtained from the cubic terms in the time derivative of $B^4_{\lambda,m}$,
\[
R^6_{m,\lambda}(u) = \Lambda^6 \left(\frac{d}{dt} B^4_{\lambda,m}(u)\right).
\]

All of the four inputs in $B^4_{\lambda,m}$ have frequencies of size 
$\lambda$. 
When one of the factors gets differentiated in time,
at which point we produce three additional 
dyadic frequencies which we denote by $\lambda_1$, $\lambda_2$ and $\lambda_3$, whose joint outcome 
must be at frequency $\lambda$.
Precisely, we have the  representation
\[
R^6_{\lambda,m} = \lambda^{3\delta-\gamma-2} 
\sum_{\lambda_1,\lambda_2.\lambda_3}
(\lambda_1 \lambda_2 \lambda_3)^\delta 
L(\bfu_\lambda,\bfu_\lambda,\bfu_\lambda, P_\lambda L(\bfu_{\lambda_1},\bfu_{\lambda_2},
\bfu_{\lambda_3})):= \sum_{\lambda_1,\lambda_2.\lambda_3} R^6_{\lambda,\lambda_1,\lambda_2,\lambda_3,m}.
\]

Depending on the balance of the four frequencies  we have two main  cases:
\bigskip

\textbf{I. The balanced case}, where we have  six $\lambda$ comparable frequencies. Then the $L^6_{t,x}$ bound 
\eqref{uk-se-boot} used six times exactly suffices.

\bigskip

\textbf{II. The unbalanced case}, 
where we have at least two frequencies away from $\lambda$.
The simplifying observation here 
is that in the unbalanced case, the 
expressions $R^6_{\lambda,\lambda_1,\lambda_2,\lambda_3,m}$ are exactly similar with the expressions
$F^{4,unbal}_{\lambda,\lambda_1,\lambda_2,\lambda_3,m}$, with two minor differences:
\begin{itemize}
    \item  $R^6_{m,\lambda,\lambda_1,\lambda_2,\lambda_3}$ has an additional $\lambda^{3\delta-\gamma-2}$ factor.

    \item $R^6_{m,\lambda,\lambda_1,\lambda_2,\lambda_3}$ has two additional $u_\lambda$ arguments.
    But these can be estimated by Bernstein's inequality  in $L^\infty$, 
   \[
\|u_\lambda\|_{L^\infty_x} \lesssim C\epsilon \lambda^{-s_c+\frac12} c_\lambda.
\]
\end{itemize}
Taking into account these two differences, the same arguments 
in the proof of Lemma~\ref{l:F4-unbal} will lead to the bound
\begin{equation}
\|  R^{6,unbal}_{\lambda,m} \|_{L^1_{t,x}}
\lesssim \lambda^{3\delta-\gamma-2}
(C\epsilon \lambda^{-s_c+\frac12} c_\lambda)^2 \epsilon^4 C^4 \lambda^{-2s_c} c_\lambda^2,
\end{equation}
which is exactly as needed for \eqref{R6-m-bd}.

\end{proof}

We remark that the same argument as in the 
last lemma can also be used to estimate 
the unbalanced part of $J^6_\lambda$,
which arises in our interaction Morawetz 
identity \eqref{interaction-bal} in the balanced case
from the time derivative of $B^4_I$, 
\[
J^{6,unbal}_\lambda(u) = \Lambda^{6,unbal}\left(\frac{d}{dt} 
B^4_I(u)\right),
\]
and has exactly the same structure as $R^6_\lambda$:

\begin{lemma}\label{l:J6-unbal}
Under our bootstrap assumptions \eqref{uk-ee-boot}-\eqref{uab-bi-boot}, in the balanced case we have the space-time bounds
\begin{equation}\label{J6-unbal-l}
\|J^{6,unbal}_{\lambda}\|_{L^1_{t,x}} \lesssim \epsilon^6 C^6 c_\lambda^6 \lambda^{-2s_c+\gamma+1}.
\end{equation}
\end{lemma}

This will be used in conjunction with the 
balanced interaction Morawetz estimates. The 
same applies in the semi-balanced case,
where we can bound $J^{6}_{\lambda\mu}$
in its entirety, including both the contributions of $J^{6,1}_{\lambda\mu}$ and $J^{6,2}_{\lambda\mu}$, as it contains no balanced contributions:

\begin{lemma}\label{l:J6-semibal}
Under our bootstrap assumptions \eqref{uk-ee-boot}-\eqref{uab-bi-boot}, in the semi-balanced case we have the space-time bounds
\begin{equation}\label{J6-semibal-l}
\|J^{6}_{\lambda\mu}\|_{L^1_{t,x}} \lesssim \epsilon^6 C^6 c_\lambda^6 \lambda^{-2s_c+\gamma+1}.
\end{equation}
\end{lemma}

Finally, for completeness we remark that in the unbalanced case we no longer need the interaction Morawetz functional correction $B^4_I$, and, on the other hand, the bound 
for $J^{6}_{\lambda \mu}$ is no longer similar to the one for $R^6_\lambda$, and will be discussed later in the last section.

\bigskip


\subsection{The energy estimate}
Our objective here is to prove the bound \eqref{uk-ee}.  Since
\[
\bM_\lambda(u) = \| u_\lambda \|_{L^2_x}^2,
\]
all we need is to bound this quantity uniformly in time,
\begin{equation}\label{unif-0}
\bM_\lambda(u) \lesssim c_\lambda^2 \epsilon^2 \lambda^{-2s_c}.
\end{equation}
For this we use the density-flux relation \eqref{df-msharp} in the case $\gamma > -1$, respectively \eqref{df-msharp-b} in the case $\gamma < -1$ The former yields
\[
\partial_t \ms_\lambda(u) = \partial_x (P_\lambda(u) + R^4_{m,\lambda}
(u))
+ F^4_{m,\lambda} + R^6_{m,\lambda}(u),
\]
where 
\[
\ms_\lambda(u,\bar u) = M_\lambda(u,\bar u) + B^4_{m,\lambda}(u).
\]
The latter yields a similar relation but with
$\partial_t$ replaced by $\partial_t+ v^\pm \partial_x$.

To prove \eqref{unif-0} we integrate the 
above density-flux relation in $t,x$ to obtain, in both cases,
\begin{equation}\label{en-ident}
\left. \int M_\lambda(u) + B^4_{m,\lambda}(u) \, dx  \right|_0^T = \int_{0}^T \int_\R F^4_{m,\lambda}(u) + R^6_{m,\lambda}(u) \ dx dt.
\end{equation}
Finally, we can estimate the contributions of $B^4_{m,\lambda}$, $F^4_{m,\lambda}$ and $R^6_{m,\lambda}$
using Lemma~\ref{l:B4-multi}, respectively Lemma~\ref{l:F4-bal}, Lemma~\ref{l:F4-unbal} and Lemma~\ref{l:R6-ab}. This yields
\begin{equation}
\Vert M_\lambda(u) \Vert_{L^1_x}\lesssim    c_\lambda^2 \epsilon^2 \lambda^{-2s_c}(1+ \epsilon^2 C^6), 
\end{equation}
which implies \eqref{unif-0} provided that $\epsilon$ is small enough, $\epsilon \ll_C 1$.

\begin{remark}
For later use, we observe that once the energy bounds \eqref{uk-ee}
have been established, then they can be used instead of 
the bootstrap assumption \eqref{uk-ee-boot} in the proof
of Lemma~\ref{l:B4-multi}. This leads to a stronger form of 
\eqref{b4-ma}, \eqref{b4-pa}, \eqref{b4-rpa}with the constant $C$ removed:
\begin{equation}\label{B4-in-L1-re}
\| B^4_{\lambda,m}(u)\|_{L^\infty_t L^1_x}
\lesssim  c_\lambda^2 \epsilon^4 \lambda^{-2s_c},
\end{equation}
and similarly for \eqref{b4-pa}, \eqref{b4-rpa}.
\end{remark}

\subsection{ The localized interaction Morawetz, balanced case}
Our objective here is to prove the bounds \eqref{uk-se} 
and \eqref{uab-bi} in the case when 
$\lambda \approx \mu$ using our bootstrap assumptions. 
 This will be achieved using our interaction Morawetz  identity \eqref{interaction-bal},  first with $v=u$ and then with $v = u^{x_0}$.
There is no difference here between the cases $\gamma > -1$ and $\gamma < -1$.
 
We will estimate the quantities in \eqref{interaction-bal} as follows:

\begin{equation}\label{Ia-bound}
|\bI^\sharp_{\lambda}(u,v)| \lesssim \epsilon^4 c_\lambda^4 \lambda^{-4s_c+\gamma+1} ,
\end{equation}
\begin{equation}\label{J4-formula}
  \bJ^4_{\lambda}(u,v) \approx  \| \partial_x Q(u_\lambda, \bv_\lambda)\|_{L^2_x}^2   , 
\end{equation}
\begin{equation}\label{J8-bound}
\int_0^T \bJ^8_\lambda(u,v)\, dt = O(\epsilon^5  C^6 c_\lambda^4 \lambda^{-4s_c+\gamma+1}), 
\end{equation}
\begin{equation}\label{K-bound}
\int_0^T\bK_\lambda(u,v)\, dt = O(\epsilon^5 C^8  c_\lambda^4 \lambda^{-4s_c+\gamma+1} ). 
\end{equation}
For $\bJ^6_{\lambda}$ we differentiate between the two cases. If $v=u$ then, using the defocusing assumption, we have
\begin{equation}\label{J6-bound}
\int_0^T \bJ^6_\lambda(u,u)\, dt  \approx  \lambda^{3\delta+\gamma} \|  P_\lambda^\frac23 u \|_{L^6_{t,x}}^6
+ O(\epsilon^5 C^6 c_\lambda^4 \lambda^{-4s_c+\gamma+1} ) .
\end{equation}
On the other hand in the shifted case
$v = u^{x_0}$ we no longer extract 
an $L^6_{t,x}$ norm, but simply estimate the corresponding term,
\begin{equation}\label{J6-bound+}
\int_0^T \bJ^6_\lambda(u,v)\, dt  \approx  O(\lambda^{3\delta+\gamma} \| ( u_\lambda,v_\lambda) \|_{L^6_{t,x}}^6)
+ O(\epsilon^5 C^6 c_\lambda^4 \lambda^{-4s_c+\gamma+1} ).
\end{equation}

\bigskip

Assuming we have the bounds \eqref{Ia-bound}-\eqref{J6-bound+}, we now show how to conclude the proof of the bounds \eqref{uk-se} and \eqref{uab-bi}.

In the case $u=v$, using the above bounds in the interaction Morawetz  identity \eqref{interaction-xi} allows  us to estimate the localized interaction Morawetz term,
\begin{equation}\label{bi-Q}
     \| \partial_x Q(u_\lambda, \bv_\lambda)\|_{L^2_x}^2  \lesssim \epsilon^4 c_\lambda^4 \lambda^{-4s_c+\gamma+1} ,
\end{equation}
as well as the localized $L^6_{t,x}$ norm
as in \eqref{uk-se}, provided that $\epsilon$ is small enough, $\epsilon \ll_C 1$; this allows us to use extra $\epsilon$ factors to defeat the $C$ factors.  Once we have the $L^6_{t,x}$ bound
\eqref{uk-se}, we return to the shifted case 
and, using \eqref{J6-bound+}, establish \eqref{bi-Q} as well.
 
The remaining step is to show that having \eqref{bi-Q} for $v = u^{x_0}$ uniformly with respect to $x_0 \in \R$ implies the bilinear bound \eqref{uab-bi}.
Ideally we would like to have $q$ constant but this is not true in general.  The redeeming feature is that  our argument allows us to control not only the $L^2$ 
norm $\| \partial_x Q(u_\lambda,\bu_\lambda)\|_{L^2}$ but also its translated versions $\| \partial_x Q(u_\lambda,\bu_\lambda^{x_0})\|_{L^2}$. So one should ask whether we could use all of these bounds in order to get the $L^2$ bound for  $\partial_x(u_\lambda  \bu_\lambda^{x_0})$. This is achieved in the following inversion lemma:

\begin{lemma}
Assume that the symbol $q$ of a translation invariant bilinear form $Q$ is smooth and positive within the compact Fourier support of two functions $u$ and $v$. Then there exists 
a representation 
\begin{equation}\label{rep}
u \bar v = \int K_0(x_0,y_0) Q(u^{x_0},\bar v^{y_0})\, dx_0 dy_0   
\end{equation}
with a Schwartz kernel $K_0$.
\end{lemma}

\begin{proof}
The action of the bilinear form $Q$ can be 
represented using its kernel $K_q$, which is the 
inverse Fourier transform of $q$,
as 
\[
Q(u,\bv)(x) = \int K_q(x_1,y_1) u(x-x_1) \bv(x-y_1)\, dx_1 dy_1  .
\]
Then the integral  in \eqref{rep}, which we denote by $I$,  can be represented as 
\[
I = \int K_1(x_0,y_0) Q(u^{x_0},\bar v^{y_0}) \, dx_0 dy_0 ,
\]
where $K_1 = K_0 \ast K_q$. This is also a bilinear form, whose symbol is the product
\[
q_1 = q_0 q.
\]
To have \eqref{rep} we need $q_1 = 1$ so we simply choose
\[
q_0 = \frac{1}{q}.
\]
We note that it suffices to localize $q_0$ to a neighbourhood of the frequency support  of $u$ and $v$. Then $q_0$ is Schwartz, ans so is the associated kernel $K_0$. 
\end{proof}

To apply the lemma in our case, we recall that in the balanced case $q$ 
is an elliptic symbol with regularity
$q \in S^{\gamma}_\lambda$ at frequency $\lambda$. Hence we can take $q_0 \in  S^{-\gamma}_\lambda$, which implies that 
\[
\|K_0\|_{L^1} \lesssim \lambda^{-\gamma}.
\]
This allows us 
to obtain 
\begin{equation}\label{Q-elliptic}
\sup_{x_0 \in \R} \| \partial_x (u_\lambda \bu_\lambda^{x_0})\|_{L^2}
\lesssim \lambda^{-\gamma} \sup_{x_0 \in \R} \| \partial_x Q(u, \bu_{x_0})\|_{L^2}.
\end{equation}
Combined with \eqref{bi-Q}, this implies the desired estimate \eqref{uab-bi}.

\bigskip

Now we return to the proof of the estimates 
\eqref{Ia-bound}-\eqref{J6-bound+}.
There is nothing to do for $\bJ^4_{\lambda}$ so we consider the remaining contributions:

\subsubsection{The $\bI_\lambda$ bound}
The interaction Morawetz functional
$\bI_\lambda$ is as in \eqref{Ia-sharp-def}, 
\begin{equation}\label{Ia-sharp-def-re}
\bI_{\lambda} =   \iint_{x > y} \ms_\lambda(u)(x) \ps_{\lambda}(v) (y) -  
\ps_{\lambda}(u)(x) \ms_{\lambda}(v) (y)\, dx dy
\end{equation}
with
\[
\ms_\lambda(u)= M_\lambda(u) + B^4_{m,\lambda}(u),
\qquad 
\ps_\lambda(u)= P_\lambda(u) + B^4_{p,\lambda}(u).
\]
For $B^4_{m,\lambda}$ and $B^4_{p,\lambda}$ we have the $L^{\infty}_tL^1_x$
bound \eqref{B4-in-L1-re}. For $M_\lambda(u)$ and $P_\lambda(u)$
we have the straightforward uniform in time bounds
\begin{equation}\label{M-L1}
\|M_\lambda(u)\|_{L^\infty_t L^1_x} \lesssim \epsilon^2 c_\lambda^2 \lambda^{-2s_c},
\qquad
 \|P_\lambda(u)\|_{L^\infty_t L^1_x} \lesssim \epsilon^2 c_\lambda^2 \lambda^{-2s_c+\gamma+1}
\end{equation}
with $P_\lambda$ replaced by $P_\lambda^\pm$ in the case $\gamma < -1$.
Combining this with \eqref{B4-in-L1-re}, 
we obtain
\begin{equation}\label{Ms-L1}
\|\ms_\lambda(u)\|_{L^\infty_t L^1_x} \lesssim \epsilon^2 c_\lambda^2 \lambda^{-2s_c},
\qquad
 \|\ps_\lambda(u)\|_{L^\infty_t L^1_x} \lesssim \epsilon^2 c_\lambda^2 \lambda^{-2s_c+\gamma+1},
\end{equation}
and the estimate \eqref{Ia-bound} for $\bI_\lambda$ immediately follows. It remains to 
estimate the interaction Morawetz correction $\bB^4_{I,\lambda}$.  But this is done exactly as in Lemma~\ref{l:B4-multi}, using the symbol bound in \eqref{Isharp-bal}.

\subsubsection{The $\bJ^6_\lambda$ bounds} This is a $6$-linear expression in $u$, $v$
and their complex conjugates, which is composed of two terms, $J^6_\lambda = J^{6,1}_\lambda+ J^{6,2}_\lambda$. Its unbalanced 
part was already estimated in $L^1_{t,x}$ in Lemma~\ref{l:J6-unbal}. Its balanced part, on the other hand, has a symbol which is smooth on the dyadic scale, 
\[
j^{6,bal}_\lambda \in S^{3\delta+\gamma}_\lambda,
\]
whose behavior on the diagonal is described 
in Lemma~\ref{l:j6-diag}. The symbol bound
allows us to estimate directly
\[
|\bJ^{6,bal}_\lambda(u,v)|\lesssim \lambda^{3\delta+\gamma} \| (u_\lambda,v_\lambda)\|_{L^6_{t,x}}^6, 
\]
thereby concluding the proof of \eqref{J6-bound+}. It remains to prove \eqref{J6-bound}
for the case when $u = v$.

The important feature here is the symbol of the 6-linear form $J^6_0$ on the diagonal
\[
\{ \xi_1 = \xi_2 = \xi_3 = \xi_4 = \xi_5 = \xi_6 \},
\]
which is positive by \eqref{good-J6}, which shows that this equals
\[
j^{6,bal}_\lambda(\xi) = p_\lambda^4(\xi) c(\xi, \xi, \xi) a''(\xi).
\]

It follows that  we can write the symbol $j^6_\lambda$ in the form
\[
j^{6,bal}_\lambda(\xi_1,\xi_2,\xi_3,\xi_4,\xi_5,\xi_6) = b_\lambda(\xi_1) b_\lambda(\xi_2) b_\lambda(\xi_3) b_\lambda(\xi_4) b_\lambda(\xi_5) b_\lambda(\xi_6)  + 
j^{6,rem}_\lambda(\xi_1,\xi_2,\xi_3,\xi_4,\xi_5,\xi_6) ,
\]
where $b_\lambda(\xi) = \phi_\lambda(\xi)^\frac23 c(\xi, \xi, \xi)^\frac16 a''(\xi)^\frac16$ and $j^{6,rem}_0$
vanishes when all $\xi$'s are equal. Then we can write
$j^{6,rem}_0$ as a linear combination of terms $\xi_{even} -\xi_{odd}$ with smooth coefficients, which at the operator level 
gives
\[
J^{6,rem}_\lambda(u) = \lambda^{3\delta+\gamma} L(\partial_x L(u_\lambda,\bu_\lambda), u_\lambda,\bu_\lambda,u_\lambda,\bu_\lambda).
\]
The first term in $j^{6,bal}_\lambda$ yields the desired $L^6_{t,x}$ norm,
\[
\bJ^6_\lambda(u) = \| B_\lambda(D) u\|_{L^6_x}^6 + \bJ^{6,rem}_\lambda(u).
\]
On the other hand the contribution   $\bJ^{6,rem}_\lambda$ of the second term  be estimated using a bilinear $L^2_{t,x}$ bound \eqref{uab-bi-boot}, 
three $L^6_{t,x}$ bounds \eqref{uk-se-boot} and one $L^\infty$ via Bernstein's inequality,
\[
\begin{aligned}
\left|\int_0^T\bJ^{6,rem}_\lambda(u)\, dt \right| \lesssim & \ \|J^{6,rem}_\lambda(u)\|_{L^1_{t,x}}
\\ \lesssim & \ \lambda^{3\delta+\gamma-1}
(C \epsilon^2 c_\lambda^2 \lambda^{-2s_c-\frac{\gamma-1}2}) (C^3(\epsilon c_\lambda)^2
\lambda^{-\frac{4s_c-1+3\delta}{2}}) (C \epsilon c_\lambda \lambda^{-s_c+\frac12}) 
\\
= & \ C^5 \epsilon^5
c_\lambda^5 \lambda^{\frac{3\delta_\gamma+1}2 -5s_c},
\end{aligned}
\]
which suffices for small enough $\epsilon$.

\subsubsection{The bound for $\bJ^8_\lambda$}
We recall that $\bJ^8_\lambda$ has an expression of the form
\begin{equation}\label{J8-def-re}
\bJ^8_{\lambda} =   \int
B^4_{m,\lambda}(u) R^4_{p,\lambda}(u) - R^4_{m,\lambda}(u) B^4_{p,\lambda}(u)
+ B^4_{m,\lambda}(u) R^4_{p,\lambda}(u) - R^4_{m,\lambda}(u) B^4_{p,\lambda}(u)
 \, dx, 
\end{equation}
see \eqref{J8-def}, with a symbol localized at frequency $\lambda$ and with size 
\[
j^8_\lambda \in S^{6\delta-2}_\lambda.
\]

For this we need to show that 
\[
\left| \int_0^T\bJ^8_\lambda\, dt\right| \lesssim \epsilon^6 c_\lambda^6 \lambda^{-4s_c+\gamma+1} .
\]
Indeed, combining six $L^6_{t,x}$ bounds with two 
$L^\infty$ bounds we obtain
\[
\left| \int_0^T\bJ^8_\lambda\, dt\right| \lesssim \lambda^{6\delta-2} C^4 \epsilon^6 c_\lambda^6 \lambda^{-4s_c+1-3\delta} \lambda^{-2s_c+1} = C^8 \epsilon^6 c_\lambda^6
\lambda^{-6s_c+3\delta}
\]
as needed.

\subsubsection{The bound for $\bK^8_\lambda$} 
We recall that $\bK^8_\lambda$ has the form
\begin{equation}\label{K8-def-re}
\begin{aligned}
\bK^8_\lambda(u) = \iint_{x > y}  &\ \ms_\lambda(u)(x)  R^6_{p,\lambda}(u)(y)  + \ps_{\lambda}(u)(y) R^6_{m,\lambda}(u)(x) 
\\
& \ -
\ms_\lambda(u)(y)  R^6_{p,\lambda}(u)(x)  - \ps_{\lambda}(u)(x) R^6_{m,\lambda}(u)(y) \, dx dy,
\end{aligned}
\end{equation}
with $\ps_{\lambda}$ substituted by the relative momentim $P^{\sharp,\pm}_\lambda$ if $\gamma < -1$.

The time integral of $\bK^8_\lambda(u)$ is estimated directly using the $L^1_{t,x}$ bound for $R^6$ in Lemma~\ref{l:R6-ab} and the uniform $L^1_x$ bound
for $\ms_\lambda$ and $\ps_\lambda$ provided by  Lemma~\ref{l:B4-multi}, 
together with the simpler bound \eqref{M-L1}.


\subsection{The localized interaction Morawetz estimate, semibalanced  case} 
Here we prove the bilinear $L^2_{t,x}$ bound \eqref{uab-bi} in the range  
where $\mu \neq \lambda$ but $\mu \approx \lambda$, which corresponds to fully transversal interactions. As discussed in Section~\ref{s:semi}, the cases $\gamma > -1$ and $\gamma < -1$ are treated in the same way, with the only difference that 
 mismatched signs in the frequency localizations are allowed  in the former case, but not in the latter case.

This repeats the same analysis as before, but
using the interaction Morawetz functional  associated to two separated dyadic frequencies $\lambda$ and $\mu$.  Here we let $v = u^{x_0}$. The parameter $x_0 \in \R$ is arbitrary and the estimates are uniform in $x_0$.

The  modified interaction functional  takes the form (see \eqref{interaction-semibal})
\begin{equation}\label{interaction-bi-re}
 \bI^\sharp_{\lambda\mu}(u,v) = \iint_{x > y} \ms_\lambda(u)(x)  \ms_\mu(v)(y) \,dx dy  + \bB^4_{I}(u,v).
\end{equation}
Its time derivative is given, see \eqref{interaction-xi}, by
\begin{equation}\label{interaction-xi-re}
\frac{d}{dt} \bI^\sharp_{\lambda\mu} =  \bJ^4_{\lambda\mu} + \bJ^6_{\lambda\mu} + \bJ^8_{\lambda\mu} + \bK^8_{\lambda\mu} .
\end{equation}
Following the same pattern as in the earlier case of the localized interaction Morawetz case, we will estimate each of these
terms as follows:
\begin{equation}\label{Ilm-bound}
|\bI^\sharp_{\lambda\mu}(u,v)| \lesssim   \epsilon^4 c_\lambda^2 c_\mu^2 \lambda^{-4s_c},
\end{equation}
\begin{equation}\label{J4lm-formula}
  \bJ^4_{\lambda\mu}(u,v) \approx  \| Q(u_\lambda \bv_\mu)\|_{L^2_x}^2   , 
\end{equation}
\begin{equation}\label{J6lm-bound}
\left|\int_0^T\bJ^6_{\lambda\mu}\, dt \right| \lesssim 
\epsilon^6 C^6  c_\lambda^2  c_\mu^2 \lambda^{-4s_c},
\end{equation}
\begin{equation}\label{J8lm-bound}
\left|\int _0^T\bJ^8_{\lambda\mu}\, dt\right| \lesssim n \epsilon^8  C^8  c_\lambda^2 c_\mu^2\lambda^{-4s_c},
\end{equation}
\begin{equation}\label{K8lm-bound}
\left| \int_0^T \bK_{\lambda\mu}\, dt \right| \lesssim n \epsilon^6  C^8  c_\lambda^2 c_\mu^2 \lambda^{-4s_c}.
\end{equation}
Combining these bounds in the integrated form of \eqref{interaction-xi-re},   estimate the leading term in $\bJ^4_{\lambda\mu}$, 
\[
\| Q(u_\lambda, \bv_\mu)\|_{L^2_{t,x}}^2 
\lesssim \epsilon^4 c_\lambda^2 c_\mu^2 \lambda^{-2s_c} \mu^{-2s_c}(1+ C^8 \epsilon^2),
\]
where the $C^8 \epsilon^2$ can be discarded  if $\epsilon \ll_C 1$.
From here, we obtain the bound \eqref{uab-bi} by using Proposition~\ref{p:I-unbal}, exactly as in the balanced case.

It remains to prove the bounds 
\eqref{Ilm-bound}-\eqref{K8lm-bound},
which is accomplished in the rest of the subsection.

\subsubsection{ The fixed time estimate for $\bI^\sharp_{\lambda\mu}$}
Here we prove the bound \eqref{Ilm-bound}, which is a consequence of fixed time $L^1$ estimates for the energy densities proved earlier, namely
\begin{equation}\label{MPA-bd}
\| \ms_\lambda(u)\|_{L^1_x} \lesssim \epsilon^2 c_\lambda^2 \lambda^{-2s_c},  
\end{equation}
and the similar estimate with $\lambda$ replaced by $\mu$ and $u$ replaced by $v$, 
as well as
\begin{equation}
|  \bB^4_{I}(u,v)|\lesssim    \epsilon^4 c_\lambda^2 c_\mu^2 \lambda^{-4s_c} .
\end{equation}
This is straightforward since the symbol $b^4_I$ is 
smooth and has size $\lambda^{-1}$.

\subsubsection{ The bound for $\bJ^6_{\lambda\mu}$}
The bound for $\bJ^6_{\lambda\mu}$ in \eqref{J6lm-bound} was already proved in Lemma~\ref{l:J6-semibal}.

\subsubsection{ The bound for $\bJ^8_{\lambda\mu}$}
Here we prove the bound \eqref{J8lm-bound}.
We recall that $\bJ^8_{\lambda\mu}$ has the form
\[
\bJ^8_{\lambda\mu} = \iint B^4_{m,\lambda}(u) R^4_{m,\mu}(v) - B^4_{m,\mu}(v) R^4_{m,\lambda}(u) 
\, dxdt.
\]
The two terms here are similar, so it suffices to consider the first one. We use two bilinear $L^2_{t,x}$ bounds and four $L^\infty$ bounds via Bernstein to get
\[
\begin{aligned}
\left|\int_0^T \bJ^8_{\lambda\mu} \, dt\right| \lesssim & \ 
C^8 \epsilon^8 c_\lambda^4 c_\mu^4\lambda^{-4s_c} \mu^{-4s_c} \lambda^{3\delta -\gamma-2}
\mu^{3\delta -1} \lambda \mu (\lambda^{\gamma+1}+\mu^{\gamma+1})^{-1} 
\\
\lesssim & \ C^8 \epsilon^8 c_\lambda^4 c_\mu^4 \lambda^{-2s_c} \mu^{-2s_c}
(\lambda\mu)^{-2s_c+3\delta -\gamma-1},
\end{aligned}
\]
which suffices.

\subsubsection{ The bound for   $\bK^8_{\lambda\mu}$} 
This is immediate by combining the bound 
\eqref{MPA-bd} with the $F^4$ and $R^6$ bounds in 
Lemma~\ref{l:F4-bal}, Lemma~\ref{l:F4-unbal} and Lemma~\ref{l:R6-ab}. 

\subsection{The localized interaction Morawetz estimate, unbalanced  case} 
Here our objectives differ somewhat in the two cases for $\gamma$, following our earlier discussion in Section~\ref{s:unbal}:

\begin{enumerate}[label=(\roman*)]
\item In the case $\gamma > -1$ 
we prove the bilinear $L^2_{t,x}$ bound \eqref{uab-bi} in the range  
where $\mu \ll\lambda$, which also corresponds to fully transversal interactions, regardless of whether
the two frequency signs are matched or not.

\item In the case $\gamma < -1$ 
we prove the bilinear $L^2_{t,x}$ bound \eqref{uab-bi} in the range  
where $\mu \ll\lambda$, assuming matched 
signs.

\item In the case $\gamma < -1$ 
we prove the bilinear $L^2_{t,x}$ bound \eqref{uab-bi-pm} in the range  
where $\mu, \lambda \gg 1$, assuming mismatched 
signs.
\end{enumerate}

This repeats the same analysis as before, but using the interaction Morawetz functionals  associated to the dyadic frequencies $\lambda$ and $\mu$, as defined in Section~\ref{s:unbal}
for each of the three cases above.  As usual we let $v = u^{x_0}$. The parameter $x_0 \in \R$ is arbitrary and the estimates are uniform in $x_0$.

In this case we recall that we use different  interaction Morawetz functionals in each of the three cases. We detail the first case, and then, to avoid extensive repetitions, 
briefly discuss the differences in the remaining two cases.

\subsubsection{The case $\gamma > -1$}
Them the Morawetz functional takes the form (see \eqref{interaction-unbal}),
\begin{equation}\label{interaction-bi-unbal}
 \bI_{\lambda\mu}(u,v) = \iint_{x > y} \rPs_\lambda(u)(x)  \ms_\mu(v)(y) \,dx dy  .
\end{equation}
Its time derivative is given, see \eqref{interaction-xi-unbal}, by
\begin{equation}\label{interaction-xi-unbal-re}
\frac{d}{dt} \bI_{\lambda\mu} =  \bJ^4_{\lambda\mu} + \bJ^6_{\lambda\mu} + \bJ^8_{\lambda\mu} + \bK^8_{\lambda\mu} ,
\end{equation}
with the terms on the right as defined in Section~\ref{s:unbal}.
Following the same pattern as in the earlier case of the localized interaction Morawetz case, we will estimate each of these
terms as follows:
\begin{equation}\label{Ilm-bound-un}
|\bI_{\lambda\mu}(u,v)| \lesssim   \epsilon^4 c_\lambda^2 c_\mu^2 \lambda^{-2s_c-\gamma-1} \mu^{-2s_c},
\end{equation}
\begin{equation}\label{J4lm-formula-un}
  \int_0^T \bJ^4_{\lambda\mu}(u,v) \, dt  =   \| u_\lambda \bv_\mu\|_{L^2_{xt}}^2 + O(C^4 \epsilon^4 
  (\mu/\lambda)^{\gamma+1} \lambda^{-2s_c-\gamma-1} \mu^{-2s_c}), 
\end{equation}
\begin{equation}\label{J6lm-bound-un}
\left|\int_0^T\bJ^6_{\lambda\mu}\, dt \right| \lesssim 
\epsilon^6 C^6  c_\lambda^2  c_\mu^2 \lambda^{-2s_c-\gamma-1} \mu^{-2s_c},
\end{equation}
\begin{equation}\label{J8lm-bound-un}
\left|\int _0^T\bJ^8_{\lambda\mu}\, dt\right| \lesssim n \epsilon^8  C^8  c_\lambda^2 c_\mu^2\lambda^{-2s_c-\gamma-1} \mu^{-2s_c} ,
\end{equation}
\begin{equation}\label{K8lm-bound-un}
\left| \int_0^T \bK_{\lambda\mu}\, dt \right| \lesssim n \epsilon^6  C^8  c_\lambda^2 c_\mu^2 \lambda^{-2s_c-\gamma-1} \mu^{-2s_c} .
\end{equation}

Combining these bounds in the integrated form of \eqref{interaction-xi-unbal}, we 
obtain the bound \eqref{uab-bi} by estimating the leading term in $\bJ^4_{\lambda\mu}$, 
\[
\| u_\lambda \bv_\mu\|_{L^2_{t,x}}^2 
\lesssim \epsilon^4 c_\lambda^2 c_\mu^2 \lambda^{-2s_c-\gamma-1} \mu^{-2s_c}(1+ C^8 \epsilon^2),
\]
which suffices if $\epsilon \ll_C 1$.
It remains to prove the bounds 
\eqref{Ilm-bound-un}-\eqref{K8lm-bound-un},
which is accomplished in the rest of the section.

\medskip 

\emph{a)  The fixed time estimate for $\bI_{\lambda\mu}$.}
Here we prove the bound \eqref{Ilm-bound-un}, which is a consequence of fixed time $L^1$ estimates for the energy densities proved earlier, namely
\begin{equation}\label{MPA-bd-re}
\| \rPs_\lambda(u)\|_{L^1_x} \lesssim \epsilon^2 c_\lambda^2 \lambda^{-2s_c-\gamma -1},  
\qquad 
\| \ms_\mu(v)\|_{L^1_x} \lesssim \epsilon^2 c_\mu^2 \mu^{-2s_c}.
\end{equation}
 
 \medskip

\emph{b) The bound for $\bJ^4_{\lambda\mu}$.} Here we use the representation given in Proposition~\ref{p:I-unbal}. Then is remains to estimate the 
small part $\bJ^{4,small}_{\lambda\mu}$, 
for which we use twice the bilinear $L^2_{t,x}$ 
bootstrap bound \eqref{uab-bi-boot}, together with the symbol bound \eqref{j4-small}. This gives
\[
\left| \int_0^T \bJ^{4,small}_{\lambda \mu} \, dt\right| \lesssim C^4 \epsilon^4 
(\mu/\lambda)^{\gamma+1} c_\lambda^2 c_\mu^2 \lambda^{-2s_c-\gamma-1} \mu^{-2s_c},
\]
as needed.

\medskip

\emph{ c) The bound for $\bJ^6_{\lambda\mu}$.}
Here we prove the bound for $\bJ^6_{\lambda\mu}$ in \eqref{J6lm-bound-un}.  This has the form
\[
\bJ^{6}_{\lambda\mu} = \int \rP_\lambda(u) R^4_{\mu,m}(v) - M_\mu(v) R^4_{\lambda,\rp}(u)
+ B^4_{\lambda,\rp}(u) P_\mu(v) - B^4_{\mu,m}(v) M_\lambda(u) 
\, dx.
\]
We recall that the symbols for the bilinear forms  $M_\lambda$, $P_\lambda$, $\rP_\lambda$ have size  $1$, $\lambda^{\gamma+1}$, respectively 
$\lambda^{-\gamma-1}$, while the symbols
for the quadrilinear forms $B^4_{\lambda,m}$ and $R^4_{\lambda,m}$
 have size $\lambda^{3\delta -\gamma-2}$
and $\lambda^{3\delta-1}$ respectively, with an additional $\lambda^{-\gamma-1}$
for their reverse momentum counterparts.
Then we consider two of the terms, as the 
other two are similar:

\bigskip

A) The  term $M_\mu(v) R^4_{\lambda,\rp}(u)$ ($B^4_{\lambda,\rp}(u) P_\mu(v)$ is similar but better). Here we use two bilinear unbalanced $L^2_{t,x}$ bounds and two $L^\infty$ bounds via Bernstein to obtain
\[
\begin{aligned}
\left| \int_0^T \int M_\mu(v) R^4_{\lambda,\rp}(u)\, dx dt \right|
\lesssim &\ C^6 \epsilon^6 c_\mu^2 c_\lambda^4
\lambda^{-4s_c+3\delta-\gamma-2} \mu^{-2s_c} \lambda^{-\gamma -1}
\lambda
\\
\lesssim & \   C^6 \epsilon^6 c_\mu^2 c_\lambda^4 \lambda^{-2s_c-\gamma-1} \mu^{-2s_c} \lambda^{-2s_c  +3\delta -\gamma -1},
\end{aligned}
\]
which suffices.

B) The  term $B^4_{\mu,m}(v) M_\lambda(u) $ ($\rP_\lambda(u) R^4_{\mu,m}(v)$ is similar but better). Here we 
use again two bilinear $L^2_{t,x}$ bounds and two $L^\infty$ bounds
via Bernstein to obtain
\[
\begin{aligned}
\left| \int_0^T \int B^4_{\mu,m}(v) M_\lambda(u)\, dx dt \right|
\lesssim &\ C^6 \epsilon^6 c_\mu^4 c_\lambda^2 \, 
\mu^{-4s_c+3\delta-\gamma-2} \lambda^{-2s_c} \lambda^{-\gamma -1}
\mu
\\
\lesssim & \   C^6 \epsilon^6 c_\mu^2 c_\lambda^4 \lambda^{-2s_c-\gamma-1} \mu^{-2s_c} \mu^{-2s_c  +3\delta -\gamma -1},
\end{aligned}
\]
which suffices.

\medskip

\emph{d) The bound for $\bJ^8_{\lambda\mu}$.}
Here we prove the bound \eqref{J8lm-bound-un}.
We recall that $\bJ^8_{\lambda\mu}$ has the form
\[
\bJ^8_{\lambda\mu} = \int B^4_{\lambda,\rp}(u) R^4_{\mu,m}(v) - B^4_{\mu,m}(v) R^4_{\lambda,\rp}(u) 
\, dx.
\]
The terms here are similar, so it suffices to consider the second one, which is larger. We use two bilinear $L^2_{t,x}$ bounds and four $L^\infty$
bounds via Bernstein to obtain
\[
\begin{aligned}
\left| \int_0^T \bJ^8_{\lambda\mu} \, dt \right|\lesssim & \ 
C^8 \epsilon^8 c_\mu^4 c_\lambda^4
\lambda^{-4s_c} \mu^{-4s_c} \mu^{3\delta -\gamma-2}
\lambda^{3\delta -1} \lambda \mu  \, \lambda^{-\gamma-1}
\\
\lesssim & \ 
C^8 \epsilon^8 c_\mu^4 c_\lambda^4
\lambda^{-2s_c} \mu^{-2s_c}
(\lambda\mu)^{-2s+3\delta -\gamma-1}
\end{aligned}
\]
as needed.

\medskip

\emph{e) The bound for   $\bK^8_{\lambda\mu}$.} 
This is immediate by combining the bound 
\eqref{MPA-bd-re} with the $F^4$ and $R^6$ bounds in 
Lemma~\ref{l:F4-bal}, Lemma~\ref{l:F4-unbal} and Lemma~\ref{l:R6-ab}. 

\subsubsection{The case $\gamma < -1$ with matched frequency signs.}
Here the Morawetz functional takes the form (see \eqref{interaction-unbal-b}),
\begin{equation}\label{interaction-bi-unbal-b}
 \bI_{\lambda\mu}(u,v) = \iint_{x > y} \ms_\lambda(u)(x)  \rP^{\sharp,\pm}_\mu(v)(y) \,dx dy  .
\end{equation}
Its time derivative has an expression 
similar to \eqref{interaction-xi-unbal-re}, where the terms in the expansion 
satisfy the counterparts of  \eqref{Ilm-bound-un}-\eqref{K8lm-bound-un}, which suffice for the proof of the bilinear $L^2_{t,x}$ bound \eqref{uab-bi}. These have the form:
\begin{equation}\label{Ilm-bound-un-b}
|\bI_{\lambda\mu}(u,v)| \lesssim   \epsilon^4 c_\lambda^2 c_\mu^2 \lambda^{-2s_c} \mu^{-2s_c-\gamma-1},
\end{equation}
\begin{equation}\label{J4lm-formula-un-b}
  \int_0^T \bJ^4_{\lambda\mu}(u,v) \, dt  =   \| u_\lambda \bv_\mu\|_{L^2_{xt}}^2 + O(C^4 \epsilon^4 
  (\lambda/\mu)^{\gamma+1}\lambda^{-2s_c} \mu^{-2s_c-\gamma-1}), 
\end{equation}
\begin{equation}\label{J6lm-bound-un-b}
\left|\int_0^T\bJ^6_{\lambda\mu}\, dt \right| \lesssim 
\epsilon^6 C^6  c_\lambda^2  c_\mu^2 \lambda^{-2s_c} \mu^{-2s_c-\gamma-1},
\end{equation}
\begin{equation}\label{J8lm-bound-un-b}
\left|\int _0^T\bJ^8_{\lambda\mu}\, dt\right| \lesssim n \epsilon^8  C^8  c_\lambda^2 c_\mu^2 \lambda^{-2s_c} \mu^{-2s_c-\gamma-1},
\end{equation}
\begin{equation}\label{K8lm-bound-un-b}
\left| \int_0^T \bK_{\lambda\mu}\, dt \right| \lesssim n \epsilon^6  C^8  c_\lambda^2 c_\mu^2 \lambda^{-2s_c} \mu^{-2s_c-\gamma-1}.
\end{equation}
The proof of these bounds is similar to the case $\gamma > -1$ and is omitted. We
only note the two differences, compared to \eqref{Ilm-bound-un}-\eqref{K8lm-bound-un}:
\begin{itemize}
    \item The factor $ \lambda^{-2s_c-\gamma-1} \mu^{-2s_c}$ is replaced by 
    $\lambda^{-2s_c} \mu^{-2s_c-\gamma-1} $, which is an obvious consequence of switching the arguments of the mass and the reverse momentum in the 
    interaction Morawetz functional.

    \item The smallness coefficient 
     $(\mu/\lambda)^{\gamma+1}$
    in the $\bJ^4_{\lambda \mu}$ error bound is replaced 
    by $ (\lambda/\mu)^{\gamma+1}$, which also is a consequence of replacing the roles of $\lambda$ and $\mu$.
\end{itemize}

\subsubsection{The case $\gamma < -1$ with mismatched frequency signs.}
Here, assuming $+-$ signs, the Morawetz functional takes the form (see \eqref{interaction-unbal-c}),
\begin{equation}\label{interaction-bi-unbal-c}
 \bI_{\lambda\mu}(u,v) = \iint_{x > y} M_\lambda^{\sharp,\pm}(u)(x)  M^{\sharp,\pm}_\mu(v)(y) \,dx dy  .
\end{equation}
Its time derivative has again the expansion \eqref{interaction-xi-unbal-re}, where the terms in the expansion 
satisfy the bounds 
\begin{equation}\label{Ilm-bound-un-c}
|\bI_{\lambda\mu}(u,v)| \lesssim   \epsilon^4 c_\lambda^2 c_\mu^2 \lambda^{-2s_c} \mu^{-2s_c},
\end{equation}
\begin{equation}\label{J4lm-formula-un-c}
  \int_0^T \bJ^4_{\lambda\mu}(u,v) \, dt  =  (v^--v^+) \| u_\lambda \bv_\mu\|_{L^2_{xt}}^2 + O(C^4 \epsilon^4 
  (\lambda^{\gamma+1}+\mu^{\gamma+1})\lambda^{-2s_c} \mu^{-2s_c}), 
\end{equation}
\begin{equation}\label{J6lm-bound-un-c}
\left|\int_0^T\bJ^6_{\lambda\mu}\, dt \right| \lesssim 
\epsilon^6 C^6  c_\lambda^2  c_\mu^2 \lambda^{-2s_c} \mu^{-2s_c},
\end{equation}
\begin{equation}\label{J8lm-bound-un-c}
\left|\int _0^T\bJ^8_{\lambda\mu}\, dt\right| \lesssim n \epsilon^8  C^8  c_\lambda^2 c_\mu^2 \lambda^{-2s_c} \mu^{-2s_c},
\end{equation}
\begin{equation}\label{K8lm-bound-un-c}
\left| \int_0^T \bK_{\lambda\mu}\, dt \right| \lesssim n \epsilon^6  C^8  c_\lambda^2 c_\mu^2 \lambda^{-2s_c} \mu^{-2s_c}.
\end{equation}
The proof of these bounds is again  similar to the other cases and is omitted.


\bibliographystyle{plain}

\end{document}